\documentclass[1p]{elsarticle}

\usepackage{lineno,hyperref}
\modulolinenumbers[5]

\usepackage{graphicx}
\usepackage{hyperref}
\usepackage{amssymb,amscd,mathtools,amsmath,amsthm}
\usepackage{mathptmx} 
\usepackage{float,wrapfig,graphicx,epsfig,esint,mathrsfs}
\usepackage{caption}
\usepackage[caption=false]{subfig}
\usepackage{booktabs}
\usepackage{multirow}
\usepackage[]{algorithm2e}

\newtheorem{proposition}{Proposition}[section]

\theoremstyle{definition}
\newtheorem{definition}{Definition}[section]

\begin{document}

\begin{frontmatter}

\title{A generalized optimal fourth-order finite difference scheme for a 2D Helmholtz equation with the perfectly matched layer boundary condition}

\author[mymainaddress]{Hatef Dastour\corref{mycorrespondingauthor}}
\cortext[mycorrespondingauthor]{Corresponding author}
\ead{hatef.dastour@ucalgary.ca}
\author[mymainaddress]{Wenyuan Liao}
\address[mymainaddress]{Department of Mathematics \& Statistics, University of Calgary, AB, T2N 1N4, Canada}

\begin{abstract}
A crucial part of successful wave propagation related inverse problems is an efficient and accurate numerical scheme for solving the seismic wave equations. In particular, the numerical solution to a multi-dimensional Helmholtz equation can be troublesome when the perfectly matched layer (PML) boundary condition is implemented. In this paper, we present a general approach for constructing fourth-order finite difference schemes for the Helmholtz equation with PML in the two-dimensional domain based on point-weighting strategy. Particularly, we develop two optimal fourth-order finite difference schemes, optimal point-weighting 25p and optimal point-weighting 17p. It is shown that the two schemes are consistent with the Helmholtz equation with PML. Moreover, an error analysis for the numerical approximation of the exact wavenumber is provided. Based on minimizing the numerical dispersion, we implement the refined choice strategy for selecting optimal parameters and present refined point-weighting 25p and refined point-weighting 17p finite difference schemes. Furthermore, three numerical examples are provided to illustrate the accuracy and effectiveness of the new methods in reducing numerical dispersion.
\end{abstract}


\end{frontmatter}


\section{Introduction}\label{S0}
In realistic heterogeneous mediums, an essential component for understanding complex wave phenomena is modeling seismic wave propagation. In particular, numerical solutions from finite-difference modeling are crucial since they can provide the complete wavefield response \cite{hustedt2004mixed}.
A key step in numerically solving the Helmholtz equation, which defines computational accuracy, is the pollution effect of high wavenumbers. In this phenomenon, the accuracy of the numerical results often deteriorates as the wavenumber $k$ increases \cite{wu2017dispersion,ihlenburg1995dispersion,ihlenburg1995finite}. The pollution effect of high wavenumbers is inevitable for two and three dimensional Helmholtz equations in practical applications \cite{wu2017dispersion,ihlenburg1995dispersion}.  Thus, solving the Helmholtz equation numerically with high wavenumbers is an important task in the field of computational mathematics \cite{jo1996optimal,chen2013optimal,shin1998frequency}.

Owing to the inevitable presence of the pollution effect of high wavenumbers, the approximated wavenumber is usually different from the exact wavenumber. This is known as the numerical dispersion \cite{chen2012dispersion}. Hence, minimizing the numerical dispersion can reduce the pollution effect as the numerical dispersion is closely related to the pollution effect \cite{jo1996optimal,chen2013optimal,shin1998frequency}.

Since the 1980s, finite difference frequency-domain (FDFD) modeling for the generation of synthetic seismograms and cross-hole tomography has been a popular field of research \cite{shin1998frequency}. The classical 5-point finite difference scheme was developed by Pratt and Worthington \cite{pratt1990inverse}. When the sampling intervals are large, the main disadvantage of this scheme is suffering from severe numerical dispersion errors as it requires ten gridpoints per wavelength. One can reduce the numerical dispersion by incorporating very small sampling intervals; however, this can lead to not only a significant increase in both storage requirements and CPU time but also linear systems with huge and ill-conditioned matrices. Solving the linear systems associated with these huge and ill-conditioned matrices, direct methods might fail to perform efficiently. Usually, iterative methods with pre-conditioners are considered to solve these linear systems \cite{de1990matrix,erlangga2006novel,van2007spectral,chen2013generalized}. An alternative approach is to avoid such linear systems is to use optimal finite difference methods.

Furthermore, artificial boundary conditions are implemented to truncate the infinite computing domain into a finite domain. Theoretically speaking, artificial boundary conditions absorb waves of any wavelength and any frequency without reflection \cite{collino1998optimizing}; nonetheless, boundary conditions on the artificial boundary are not available in general. An ideal artificial boundary condition should be computationally stable, not require extensive computational resources, and have an acceptable level of accuracy \cite{collino1998optimizing}. The perfectly matched layer (PML) absorbing boundary condition has been the most popular method, which was introduced by Bérenger in 1994 \cite{berenger1994perfectly}, and is used to eliminate artificial reflection near the boundary. PML technique introduces an artificial layer with an attenuation parameter around the interior area (the domain of interest). On the other hand, applying PML will modify the original Helmholtz equation and the resulting equation is more challenging to solve as existing numerical methods may fail to solve the modified Helmholtz equation effectively and accurately (for more details about PML see \cite{collino1998optimizing,berenger1994perfectly,turkel1998absorbing,singer2004perfectly}).

During the past few decades, there have been many researches on reducing the numerical dispersion of finite difference methods for solving the Helmholtz equation. The rotated 9-point finite difference method for the Helmholtz equation was developed by Jo et al. \cite{jo1996optimal} in 1996. Their method consists of linearly combining the two discretizations of the second derivative operator on the classical Cartesian coordinate system and the $45^\circ$ rotated system. In addition, the idea of the rotated 9-point scheme was extended by Shin and Sohn \cite{shin1998frequency} to the 25-point formula, and they obtained a group of optimal parameters by the singular-value decomposition method in 1998.
Furthermore, Chen et al. proved that the rotated 9-point is inconsistent with the Helmholtz equation in the presence of PML, and constructed another 9-point finite difference scheme for the Helmholtz equation which is consistent with PML \cite{chen2013optimal}. Based on minimizing the numerical dispersion, refined and global choice strategies were also proposed by the authors for choosing optimal parameters of the optimal 9-point scheme.
Moreover, a generalized optimal 9-point scheme for frequency-domain scalar wave equation was developed by Chen \cite{chen2013generalized} which is an extension of the rotated 9-point scheme for the case that different spacial increments along x-axis and z-axis are used.
Additionally, Cheng et al. \cite{cheng2017dispersion}, in 2017, presented a new dispersion minimizing finite difference method in which the combination weights are determined by minimizing the numerical dispersion with a flexible selection strategy.

Moreover, many authors have made a great deal of efforts to improve numerical accuracy by developing higher-order finite difference schemes.
Harari et al. \cite{harari1995accurate} presented various fourth-order methods for time-harmonic wave propagation which depends on the angle of wave propagation. Furthermore, Singer et al. \cite{singer1998high} developed and analyzed a fourth-order compact finite difference scheme which depends on uniform grids for the two dimensional Helmholtz equation with constant wavenumbers and is less sensitive to the direction of the propagation.
Moreover, Wu  \cite{wu2017dispersion} proposed an optimal compact finite difference scheme whose parameters are chosen based on minimizing the numerical dispersion. Additionally, a fourth-order accurate finite difference scheme for the variable coefficient Helmholtz equation was developed by Britt et al. \cite{britt2011numerical} which reduces phase error compared with a second order. Sutmann \cite{sutmann2007compact} derived sixth-order compact finite difference schemes for the 2D and 3D Helmholtz equation with constant coefficients. Sixth-order finite difference schemes for the 2D and 3D Helmholtz equation with variable coefficients have been the subject of a number of researchers and the interested readers are referred to \cite{wu2018optimal,turkel2013compact}.

In addition, many of these higher-order schemes, in particular, compact finite difference schemes, require the source term to be smooth enough to obtain higher-order accuracy, and this is not always the case in many practical problems; however, non-compact finite difference schemes do not need this requirement. Dastour et al. \cite{dastour2019} proposed two non-compact optimal finite difference schemes, optimal 25-point and optimal 17-point finite difference schemes, for the Helmholtz equation with PML. They demonstrated that the 17-point finite difference method is inconsistent with the Helmholtz equation with PML and is impractical when different spacial increment along x-axis and z-axis are used.

In \cite{dastour2019}, the authors investigated extensions of the rotated 9-point finite difference scheme for the Helmholtz equation with PML \cite{jo1996optimal} and the 9-point finite difference scheme for the Helmholtz equation with PML \cite{chen2013optimal} to 17-point and 25-point finite difference schemes for the Helmholtz equation with PML, respectively. It was shown that the presented finite difference schemes are fourth-order; however, the 17-point finite difference scheme is inconsistent in the presence of PML while the 25-point finite difference scheme is pointwise consistent \cite{dastour2019}. In addition, the 17-point finite difference could not approximate the Laplacian operator of the wave equation with fourth-order accuracy when different spacial increments along x-axis and z-axis are used.
However, there are a number of advantages in using a fourth-order 17-point finite difference scheme for the Helmholtz equation with PML. For example, the generated banded matrix for this finite difference has less width than the 25-point finite difference scheme.
In this paper, to further reduce the numerical dispersion and fix the problems of the fourth-order 17-point finite difference scheme with the Helmholtz equation with PML, we present a general approach for constructing non-compact fourth-order finite difference schemes based on the point-weighting strategy \cite{cheng2017dispersion}.

The rest of this paper is organized as follows. In Section \ref{S1}, a general approach for constructing non-compact fourth-order finite difference methods based on the point-weighting strategy is presented. Particularly, we develop two new schemes, point-weighting 25p and 17p finite difference methods. Furthermore, we prove that the new schemes are consistent with the Helmholtz equation and fourth-order when PML is applied and different spacial increments along x-axis and z-axis are used. In Section \ref{S2}, we analyze the error between the numerical wavenumbers and the exact wavenumber and introduce refined point-weighting 25p and 17p schemes. In Section \ref{S3}, three examples are given to demonstrate the efficiency of the schemes. We demonstrate the importance of consistency of a scheme with the Helmholtz equation in the presence of PML and the necessity of implementing the refined point-weighting 17p instead of the refined 17p from \cite{dastour2019}. Finally, in Section \ref{S4} conclusions of this paper and possible future works are discussed.

\section{Fourth-order optimal finite difference schemes based on point-weighting strategy}\label{S1}
In this work, we consider the numerical solution of the 2D Helmholtz equation with PML given by \cite{singer2004perfectly,turkel1998absorbing}:
\begin{align}\label{S1.eq.02}
\frac{\partial}{\partial x}\left(A\left(x,z\right)\frac{\partial }{\partial x}p\left(x,z\right)\right)+
\frac{\partial}{\partial z}\left(B\left(x,z\right)\frac{\partial }{\partial z}p\left(x,z\right)\right)+
C\left(x,z\right) k^2\left(x,z\right) p\left(x,z\right)=\tilde{g}\left(x,z\right),
\end{align}
where $k = 2\pi f/v$ is the wavenumber in which $f$ and $v$ represent the frequency and the velocity respectively, $p$  is the pressure wavefield in the Fourier domain. Moreover, $A\left(x,z\right)=s_z/s_x,~B\left(x,z\right)=s_x/s_z,~\text{and }C\left(x,z\right)=s_xs_z$ in which $s_x = 1 -i\sigma_x/\omega$,\linebreak $s_z = 1 - i\sigma_z/\omega$ with $\omega = 2\pi f$ denotes the angular frequency, and
\begin{align*}
\tilde{g}=\begin{cases}0 , &\text{inside PML},\\ g, &\text{outside PML.}\end{cases}
\end{align*}
with $g$ is the Fourier transform of the source function.

Here, $\sigma_x$ and $\sigma_z$ are usually chosen as differentiable functions depending on the variable $x$ and $z$ only, respectively. For example, one may consider defining them as follows,
\begin{align}\label{S1.eq.03}
\sigma_x&=
\begin{cases}
2\pi a_0 f_M \left(\frac{l_x}{L_{PML}}\right)^2,& \text{inside PML},\\
0,& \text{outside PML},
\end{cases}
\\
\sigma_z&=
\begin{cases}
2\pi a_0 f_M \left(\frac{l_z}{L_{PML}}\right)^2,& \text{inside PML},\\
0,& \text{outside PML},
\end{cases}
\end{align}
where $f_M$ is the peak frequency of the source, $L_{PML}$ is the thickness of PML, $l_x$ and $l_z$ are the distance from the point $(x, z)$ inside PML to the interface between the interior region and PML region. Furthermore, $a_0$ is a constant, and we choose $a_0 = 1.79$ according to the paper \cite{zeng2001application}.

In the interior domain, $s_x = 1$ and $s_z = 1$ lead to $A = B = C = 1$. Thus,  equation \ref{S1.eq.02} can be regarded as a general form of the Helmholtz equation \eqref{S1.eq.01} with its corresponding PML,
\begin{align}\label{S1.eq.01}
\Delta p\left(x,z\right) +k^2\left(x,z\right) p\left(x,z\right)=g\left(x,z\right),
\end{align}
where $\Delta=\partial^2/\partial x^2 +\partial^2/\partial z^2$  is the Laplacian.

The number of wavelengths in a square domain of size $H$ equals $H /\lambda$, where $\lambda$ is the wavelength, and it is defined by $\lambda = v/ f$.
For the convenience of analysis, the two dimensional square computational domain is often normalized into $[0, ~1 ] \times [0, ~1 ]$, and then the dimensionless wavenumber is equal to $2\pi fH /v$  \cite{ihlenburg1995dispersion,cheng2017dispersion}. In the remainder of the paper, the wavenumber refers to dimensionless wavenumber, which is also denoted by $k$.

There are two common strategies for constructing optimal finite difference methods, derivative-weighting and point-weighting strategies. The main differences between a derivative-weighting scheme and a point-weighting scheme were discussed by Cheng et al. in \cite{cheng2017dispersion}.
The first difference is the way for discretizing the Laplacian operator with PML
$$\frac{\partial}{\partial x}\left(A\frac{\partial p}{\partial x}\right)+\frac{\partial}{\partial z}\left(B\frac{\partial p}{\partial z}\right).$$ Moreover, the second difference lies in their capability of reducing numerical dispersion \cite{chen2013optimal}. For more details about the derivative-weighting scheme, please refer to \cite{chen2013optimal,cheng2017dispersion}.

Dastour et al. in \cite{dastour2019} proposed two finite difference schemes, refined 17-point and refined 25-point schemes, which are derivative-weighting schemes. They proved that the 17-point scheme for the Helmholtz equation with PML is inconsistent. However, the 17-point scheme has some advantages over the 25-point scheme in terms of computational complexity and CPU time mainly as the generated banded matrix for the 17-point scheme has less width than the 25-point scheme. We will discuss the computational complexity of the new schemes in details in Example \ref{Ex.1}. For the two dimensional Helmholtz equation with PML \eqref{S1.eq.02}, we next construct a general fourth-order finite difference scheme bases on point-weighting strategy and will introduce two fourth-order optimal finite difference schemes.

Consider the network of grid points $(x_m ,z_n )=(x_0 + m\,\Delta x,~z_0 +n\,\Delta z)$ for $m,n=0,1,2,\ldots$. Let $p_{m,n} = \left. p \right|_{x=x_m ,z=z_n}$ and $k_{m,n} = \left.k \right|_{x=x_m ,z=z_n}$ represent the pressure of the wavefield and the wavenumber  at the location $(x_m ,z_n)$, respectively. Moreover,  The discretization of $A(x,z)$, $B(x,z)$ and $C(x,z)$ at point $(m,n )$ are denoted by $A_{m,n}$, $B_{m,n}$, and $C_{m,n}$, respectively. In addition, we have
\begin{align}\label{S1.eq.04}
&\begin{cases}
A_{m+\frac{j}{2},n+\frac{l}{2}}=A\left(x_m+\dfrac{j}{2}\Delta x,~z_n+\dfrac{l}{2}\Delta z\right),
\\ \vspace{-0.4cm}\\
B_{m+\frac{j}{2},n+\frac{l}{2}}=B\left(x_m+\dfrac{j}{2}\Delta x,~z_n+\dfrac{l}{2}\Delta z\right),
\\ \vspace{-0.4cm}\\
C_{m,n}=C\left(x_m,~z_n\right),
\end{cases}
&j,l\in \{-3,-1,0,1,3\}.
\end{align}

First off, we need to approximate $\dfrac{\partial}{\partial x}\left(A\dfrac{\partial p}{\partial x}\right)$ and $\dfrac{\partial}{\partial z}\left(B\dfrac{\partial p}{\partial z}\right)$ with fourth-order accuracy. It follows from equation \eqref{S1.eq.02} that
\begin{align}
\label{S1.eq.05}
\left. \frac{\partial}{\partial x}\left(A\frac{\partial p}{\partial x}\right)\right|_{x=x_m ,~z=z_n}&=
\left. \alpha_{1} \left(A\frac{\partial p}{\partial x}\right) \right|_{x=x_m-\frac{3\Delta x}{2} ,~z=z_n}+
\left. \alpha_{2} \left(A\frac{\partial p}{\partial x}\right) \right|_{x=x_m-\frac{\Delta x}{2} ,~z=z_n}
\notag \\ &
+\left. \alpha_{3} \left(A\frac{\partial p}{\partial x}\right) \right|_{x=x_m+\frac{\Delta x}{2} ,~z=z_n}+
\left. \alpha_{4} \left(A\frac{\partial p}{\partial x}\right) \right|_{x=x_m-\frac{3\Delta x}{2} ,~z=z_n},
\\
\label{S1.eq.06}
\left. \frac{\partial}{\partial z}\left(B\frac{\partial p}{\partial z}\right)\right|_{x=x_m ,~z=z_n}&=
\left. \beta_{1} \left(B\frac{\partial p}{\partial x}\right) \right|_{x=x_m ,~z=z_n-\frac{3\Delta z}{2}}+
\left. \beta_{2} \left(B\frac{\partial p}{\partial x}\right) \right|_{x=x_m ,~z=z_n-\frac{\Delta z}{2}}
\notag \\ &
+\left. \beta_{3} \left(B\frac{\partial p}{\partial x}\right) \right|_{x=x_m,~z=z_n+\frac{\Delta z}{2} }+
\left. \beta_{4} \left(B\frac{\partial p}{\partial x}\right) \right|_{x=x_m,~z=z_n-\frac{3\Delta z}{2} }.
\end{align}
We determine the coefficients $\alpha_{i}$ and $\beta_{i}$, for $i=1,\ldots,4$, in a way that \eqref{S1.eq.05} and \eqref{S1.eq.06} can approximate the Laplacian operator with PML with fourth-order accuracy. Applying the Taylor theorem on right-hand sides of \eqref{S1.eq.05} and \eqref{S1.eq.06}, and then solving the generated linear systems, we have
\begin{align*}
\begin{cases}
\alpha_{1}=\dfrac{1}{24}\dfrac{1}{\Delta x},~\alpha_{2}=-\dfrac{9}{8}\dfrac{1}{\Delta x},~\alpha_{3}=\dfrac{9}{8}\dfrac{1}{\Delta x}, \text{ and }\alpha_{4}=-\dfrac{1}{24}\dfrac{1}{\Delta x},
\\ \vspace{-0.4cm}\\
\beta_{1}=\dfrac{1}{24}\dfrac{1}{\Delta z},~\beta_{2}=-\dfrac{9}{8}\dfrac{1}{\Delta z},~\beta_{3}=\dfrac{9}{8}\dfrac{1}{\Delta z}, \text{ and }\beta_{4}=-\dfrac{1}{24}\dfrac{1}{\Delta z}.
\end{cases}
\end{align*}
Furthermore, we need to approximate the first derivatives of $p$ with respect to $x$ and $z$ at points $\left(x_m-\frac{3}{2}\Delta x,z_n\right),\ldots,\left(x_m,z_n+\frac{3}{2}\Delta z\right)$ with fourth-order accuracy. As a case in point, at point $\left(x_m-\frac{3}{2}\Delta x,z_n\right)$, we have
\begin{align}
\label{S1.eq.07}
\left. \frac{\partial p}{\partial x}\right|_{x=x_m-\frac{3h}{2} ,~z=z_n}& =
w_{1} p_{m-2,n}+w_{2} p_{m-1,n}+w_{3} p_{m,n}+w_{4} p_{m+1,n}+w_{5} p_{m+2,n}
\end{align}
with
\begin{align}\label{S1.eq.08}
w_{1}=-\frac{11}{12}\dfrac{1}{\Delta x},~w_{2}=\frac{17}{24}\dfrac{1}{\Delta x},~w_{3}=\frac{3}{8}\dfrac{1}{\Delta x},w_{4}=-\frac{5}{24}\dfrac{1}{\Delta x} \text{ and } w_{5}=\frac{1}{24}\dfrac{1}{\Delta x}.
\end{align}
Therefore,
\begin{align}\label{S1.eq.09}
\mathcal{L}_{x}p_{m ,n}&=
\frac{1}{\Delta x^2}\left[
-\frac{9}{8}A_{m-\frac{1}{2},n}\left(\frac{1}{24}p_{m-2,n}-\frac{9}{8}p_{m-1,n}+\frac{9}{8}p_{m,n}-\frac{1}{24}p_{m+1,n}\right)
\right. \notag \\ &
+\frac{1}{24}A_{m-\frac{3}{2},n}\left(-\frac{11}{12}p_{m-2,n}+\frac{17}{24}p_{m-1,n}+\frac{3}{8}p_{m,n}-\frac{5}{24}p_{m+1,n}+\frac{1}{24}p_{m+2,n}\right)
\notag \\ &
-\frac{1}{24}A_{m+\frac{3}{2},n}\left(-\frac{1}{24}p_{m-2,n}+\frac{5}{24}p_{m-1,n}-\frac{3}{8}p_{m,n}-\frac{17}{24}p_{m+1,n}+\frac{11}{12}p_{m+2,n}\right)
\notag \\ & \left.
+\frac{9}{8}A_{m+\frac{1}{2},n}\left(\frac{1}{24}p_{m-1,n}-\frac{9}{8}p_{m,n}+\frac{9}{8}p_{m+1,n}-\frac{1}{24}p_{m+2,n}\right)
\right],
\end{align}
\begin{align}\label{S1.eq.10}
\mathcal{L}_{z}p_{m ,n}&=
\frac{1}{\Delta z^2}\left[
-\frac{9}{8}B_{m,n-\frac{1}{2}}\left(\frac{1}{24}p_{m,n-2}-\frac{9}{8}p_{m,n-1}+\frac{9}{8}p_{m,n}-\frac{1}{24}p_{m,n+1}\right)
\right. \notag \\ &
+\frac{1}{24}B_{m,n-\frac{3}{2}}\left(-\frac{11}{12}p_{m,n-2}+\frac{17}{24}p_{m,n-1}+\frac{3}{8}p_{m,n}-\frac{5}{24}p_{m,n+1}+\frac{1}{24}p_{m,n+2}\right)
\notag \\ &
-\frac{1}{24}B_{m,n+\frac{3}{2}}\left(-\frac{1}{24}p_{m,n-2}+\frac{5}{24}p_{m,n-1}-\frac{3}{8}p_{m,n}-\frac{17}{24}p_{m,n+1}+\frac{11}{12}p_{m,n+2}\right)
\notag \\ & \left.
+\frac{9}{8}B_{m,n+\frac{1}{2}}\left(\frac{1}{24}p_{m,n-1}-\frac{9}{8}p_{m,n}+\frac{9}{8}p_{m,n+1}-\frac{1}{24}p_{m,n+2}\right)
\right].
\end{align}
These $\mathcal{L}_{x}p_{m ,n}$ and $\mathcal{L}_{z}p_{m ,n}$ provides a general approach for constructing fourth-order finite difference scheme based on  point-weighting strategy. We can replace $p_{m-2,n}$, $p_{m-1,n}$, \dots, $p_{m,n+1}$ and $p_{m,n+2}$ from equations \eqref{S1.eq.09} and \eqref{S1.eq.10} with new weighted arithmetic averages of other points in a way that $\mathcal{L}_{x}p_{m ,n}$ and $\mathcal{L}_{z}p_{m ,n}$ can maintain their fourth-order accuracy. For example, we can consider the following approximations of these points to create an optimal finite difference scheme,
\begin{align}\label{S1.eq.11}
\begin{cases}
p_{m-2,n}^{*} = a_{1} p_{m-2,n} + a_{2} \left(\dfrac{1}{72}p_{m-2,n-2} -\dfrac{1}{18}p_{m-2,n-1} -\dfrac{1}{18}p_{m-2,n+1} + \dfrac{1}{72}p_{m-2,n+2}\right),\\
p_{m-1,n}^{*} = a_{1} p_{m-1,n} + a_{2} \left(-\dfrac{2}{9}p_{m-1,n-2} + \dfrac{8}{9}p_{m-1,n-1} + \dfrac{8}{9}p_{m-1,n+1} -\dfrac{2}{9}p_{m-1,n+2}\right),\\
p_{m,n}^{*}   = a_{1} p_{m,n}   + a_{2} \left(\dfrac{5}{12}p_{m,n-2} -\dfrac{5}{3}p_{m,n-1} -\dfrac{5}{3}p_{m,n+1} + \dfrac{5}{12} p_{m,n+2}\right),\\
p_{m+1,n}^{*} = a_{1} p_{m+1,n} + a_{2} \left(-\dfrac{2}{9}p_{m+1,n-2}+ \dfrac{8}{9}p_{m+1,n-1} + \dfrac{8}{9}p_{m+1,n+1}-\dfrac{2}{9}p_{m+1,n+2}\right),\\
p_{m+2,n}^{*} = a_{1} p_{m+2,n} + a_{2} \left(\dfrac{1}{72}p_{m+2,n-2} -\dfrac{1}{18}p_{m+2,n-1} -\dfrac{1}{18}p_{m+2,n+1}+ \dfrac{1}{72}p_{m+2,n+2}\right),
\end{cases}
\end{align}
and
\begin{align}\label{S1.eq.12}
\begin{cases}
p_{m,n-2}^{**} = a_{1} p_{m,n-2} + a_{2} \left(\dfrac{1}{72}p_{m-2,n-2} -\dfrac{1}{18}p_{m-1,n-2} -\dfrac{1}{18}p_{m+1,n-2} + \dfrac{1}{72}p_{m+2,n-2}\right),\\
p_{m,n-1}^{**} = a_{1} p_{m,n-1} + a_{2} \left(-\dfrac{2}{9}p_{m-2,n-1} + \dfrac{8}{9}p_{m-1,n-1} + \dfrac{8}{9}p_{m+1,n-1} -\dfrac{2}{9}p_{m+2,n-1}\right),\\
p_{m,n}^{**}   = a_{1} p_{m,n}   + a_{2} \left(\dfrac{5}{12}p_{m-2,n} -\dfrac{5}{3}p_{m-1,n} -\dfrac{5}{3}p_{m+1,n} + \dfrac{5}{12} p_{m+2,n}\right),\\
p_{m,n+1}^{**} = a_{1} p_{m,n+1} + a_{2} \left(-\dfrac{2}{9}p_{m-2,n+1}+ \dfrac{8}{9}p_{m-1,n+1} + \dfrac{8}{9}p_{m+1,n+1}-\dfrac{2}{9}p_{m+2,n+1}\right),\\
p_{m,n+2}^{**} = a_{1} p_{m,n+2} + a_{2} \left(\dfrac{1}{72}p_{m-2,n+2} -\dfrac{1}{18}p_{m-1,n+2} -\dfrac{1}{18}p_{m+1,n+2}+ \dfrac{1}{72}p_{m+2,n+2}\right).
\end{cases}
\end{align}
where $a_2=1-a_1$, and $a_1$ is a parameter to be determined. It follows from substituting above values into \eqref{S1.eq.09} and \eqref{S1.eq.10} that,
\begin{align}\label{S1.eq.13}
\mathcal{L}_{x}^{*}p_{m ,n}&=
\frac{1}{\Delta x^2}\left[
-\frac{9}{8}A_{m-\frac{1}{2},n}\left(\frac{1}{24}p^{*}_{m-2,n}-\frac{9}{8}p^{*}_{m-1,n}+\frac{9}{8}p^{*}_{m,n}-\frac{1}{24}p^{*}_{m+1,n}\right)
\right. \notag \\ &
+\frac{1}{24}A_{m-\frac{3}{2},n}\left(-\frac{11}{12}p^{*}_{m-2,n}+\frac{17}{24}p^{*}_{m-1,n}+\frac{3}{8}p^{*}_{m,n}-\frac{5}{24}p^{*}_{m+1,n}+\frac{1}{24}p^{*}_{m+2,n}\right)
\notag \\ &
-\frac{1}{24}A_{m+\frac{3}{2},n}\left(-\frac{1}{24}p^{*}_{m-2,n}+\frac{5}{24}p^{*}_{m-1,n}-\frac{3}{8}p^{*}_{m,n}-\frac{17}{24}p^{*}_{m+1,n}+\frac{11}{12}p^{*}_{m+2,n}\right)
\notag \\ & \left.
+\frac{9}{8}A_{m+\frac{1}{2},n}\left(\frac{1}{24}p^{*}_{m-1,n}-\frac{9}{8}p^{*}_{m,n}+\frac{9}{8}p^{*}_{m+1,n}-\frac{1}{24}p^{*}_{m+2,n}\right)
\right],
\end{align}
\begin{align}\label{S1.eq.14}
\mathcal{L}_{z}^{*}p_{m ,n}&=
\frac{1}{\Delta z^2}\left[
-\frac{9}{8}B_{m,n-\frac{1}{2}}\left(\frac{1}{24}p^{**}_{m,n-2}-\frac{9}{8}p^{**}_{m,n-1}+\frac{9}{8}p^{**}_{m,n}-\frac{1}{24}p^{**}_{m,n+1}\right)
\right. \notag \\ &
+\frac{1}{24}B_{m,n-\frac{3}{2}}\left(-\frac{11}{12}p^{**}_{m,n-2}+\frac{17}{24}p^{**}_{m,n-1}+\frac{3}{8}p^{**}_{m,n}-\frac{5}{24}p^{**}_{m,n+1}+\frac{1}{24}p^{**}_{m,n+2}\right)
\notag \\ &
-\frac{1}{24}B_{m,n+\frac{3}{2}}\left(-\frac{1}{24}p^{**}_{m,n-2}+\frac{5}{24}p^{**}_{m,n-1}-\frac{3}{8}p^{**}_{m,n}-\frac{17}{24}p^{**}_{m,n+1}+\frac{11}{12}p^{**}_{m,n+2}\right)
\notag \\ & \left.
+\frac{9}{8}B_{m,n+\frac{1}{2}}\left(\frac{1}{24}p^{**}_{m,n-1}-\frac{9}{8}p^{**}_{m,n}+\frac{9}{8}p^{**}_{m,n+1}-\frac{1}{24}p^{**}_{m,n+2}\right)
\right].
\end{align}

Therefore,  the Laplacian operator with PML can be approximated as follows,
\begin{align}\label{S1.eq.15}
\frac{\partial}{\partial x}\left(A\frac{\partial p}{\partial x}\right)+\frac{\partial}{\partial z}\left(B\frac{\partial p}{\partial z}\right)
\approx
\mathcal{L}_{h} p_{m,n},
\end{align}
where $\mathcal{L}^{*}=\mathcal{L}_{x}^{*}+\mathcal{L}_{z}^{*}$.

Moreover, we can approximate $k^2_{m,n}C_{m,n}p_{m,n}$ with fourth-order accuracy. Let
\begin{align}
\label{S1.eq.17}
I^{(1)}\left(Q_{m,n}\right)& = Q_{m,n},\\
\label{S1.eq.18}
I^{(2)}\left(Q_{m,n}\right)& = \frac{1}{3}\,\left(Q_{m-1,n}+Q_{m+1,n}+Q_{m,n-1}+Q_{m,n+1}\right)
\notag \\ &
-\frac{1}{12}\,\left(Q_{m-2,n}+Q_{m+2,n}+Q_{m,n-2}+Q_{m,n+2}\right),
\\
\label{S1.eq.19}
I^{(3)}\left(Q_{m,n}\right)&=
\frac{1}{3}\,\left(Q_{m-1,n-1}+Q_{m+1,n+1}+Q_{m-1,n+1}+Q_{m+1,n-1}\right)
\notag \\ &
-\frac{1}{12}\,\left(Q_{m-2,n-2}+Q_{m+2,n+2}+Q_{m-2,n+2}+Q_{m+2,n-2}\right),\\
\label{S1.eq.20}
I^{(4)}\left(Q_{m,n}\right)&=
\frac{1}{36}\left(Q_{m-2,n-2}+Q_{m-2,n+2}+Q_{m+2,n-2}+Q_{m+2,n+2}\right)
\notag \\ &
-\frac{1}{9}\left(Q_{m-1,n-2}+Q_{m+1,n-2}+Q_{m-2,n-1}+Q_{m+2,n-1}
\right. \notag \\ & \left.
+Q_{m-1,n+2}+Q_{m+1,n+2}+Q_{m-2,n+1}+Q_{m+2,n+1} \right)
\notag \\ &
+\frac{4}{9}\left(Q_{m-1,n-1}+Q_{m+1,n-1}+Q_{m-1,n+1}+Q_{m+1,n+1}\right),
\end{align}
where $Q_{m,n}=k^2_{m,n}C_{m,n}p_{m,n}$. Therefore,
\begin{align}\label{S1.eq.21}
I^{*}\left(k^2_{m,n}C_{m,n}p_{m,n}\right)&=\sum_{j=1}^{4}c_{j}\,I^{(j)}\left(k^2_{m,n}C_{m,n}p_{m,n}\right),
\end{align}
where $c_j$ are parameters satisfying $\sum_{j=1}^{4}c_j= 1$.

Therefore, an optimal 25-point finite difference scheme for the Helmholtz-PML equation \eqref{S1.eq.02} can be obtained as follows,
\begin{align}\label{S1.eq.22}
\mathcal{L}^{*}\left(p_{m ,n}\right)+I^{*}\left(k^2_{m,n}C_{m,n}p_{m,n}\right)=\tilde{g}_{m,n}.
\end{align}
We refer the finite difference scheme \eqref{S1.eq.22} as optimal point-weighting 25-point finite difference method. In Proposition \ref{S1.P01}, we demonstrate that this finite difference scheme is consistent with the Helmholtz equation with PML \eqref{S1.eq.02}.

Alternatively, we can develop another optimal finite difference method by replacing $p_{m-2,n}$, $p_{m-1,n}$, \dots, $p_{m,n+1}$ and $p_{m,n+2}$ from equations \eqref{S1.eq.09} and \eqref{S1.eq.10} with the following values,
\begin{align}\label{S1.eq.23}
\begin{cases}
\hat{p}_{m-2,n} = b_1 p_{m-2,n} + \dfrac{b_{2}}{2}\left(p_{m-2,n+2}+p_{m-2,n-2}-p_{m,n-2}-p_{m,n+2}\right),
\\ \vspace{-0.4cm} \\
\hat{p}_{m-1,n} = b_1 p_{m-1,n} + \dfrac{b_{2}}{2}\left(p_{m-1,n+1}+p_{m-1,n-1}-p_{m,n-1}-p_{m,n+1}\right),
\\ \vspace{-0.4cm} \\
\hat{p}_{m,n}   = b_1 p_{m,n},
\\ \vspace{-0.4cm} \\
\hat{p}_{m+1,n} = b_1 p_{m+1,n} + \dfrac{b_{2}}{2}\left(p_{m+1,n-1}+p_{m+1,n+1}-p_{m,n+1}-p_{m,n-1}\right),
\\ \vspace{-0.4cm} \\
\hat{p}_{m+2,n} = b_1 p_{m+2,n} + \dfrac{b_{2}}{2}\left(p_{m+2,n-2}+p_{m+2,n+2}-p_{m,n+2}-p_{m,n-2}\right),
\end{cases}
\end{align}
and
\begin{align}\label{S1.eq.24}
\begin{cases}
\tilde{p}_{m,n-2} = b_1 p_{m,n-2} + \dfrac{b_{2}}{2}\left(p_{m+2,n-2}+p_{m-2,n-2}-p_{m-2,n}-p_{m+2,n}\right),
\\ \vspace{-0.4cm} \\
\tilde{p}_{m,n-1} = b_1 p_{m,n-1} + \dfrac{b_{2}}{2}\left(p_{m+1,n-1}+p_{m-1,n-1}-p_{m-1,n}-p_{m+1,n}\right),
\\ \vspace{-0.4cm} \\
\tilde{p}_{m,n}   = b_1 p_{m,n}  ,\\ \vspace{-0.4cm} \\
\tilde{p}_{m,n+1} = b_1 p_{m,n+1} + \dfrac{b_{2}}{2}\left(p_{m-1,n+1}+p_{m+1,n+1}-p_{m+1,n}-p_{m-1,n}\right),
\\ \vspace{-0.4cm} \\
\tilde{p}_{m,n+2} = b_1 p_{m,n+2} + \dfrac{b_{2}}{2}\left(p_{m-2,n+2}+p_{m+2,n+2}-p_{m+2,n}-p_{m-2,n}\right).
\end{cases}
\end{align}
where $b_2=1-b_1$ and $b_1$ is a parameter to be determined. It follows from substituting above values into \eqref{S1.eq.09} and \eqref{S1.eq.10} that,
\begin{align}\label{S1.eq.25}
\mathcal{\tilde{L}}_{x}p_{m ,n}&=
\frac{1}{\Delta x^2}\left[
-\frac{9}{8}A_{m-\frac{1}{2},n}\left(\frac{1}{24}\hat{p}_{m-2,n}-\frac{9}{8}\hat{p}_{m-1,n}+\frac{9}{8}\hat{p}_{m,n}-\frac{1}{24}\hat{p}_{m+1,n}\right)
\right. \notag \\ &
+\frac{1}{24}A_{m-\frac{3}{2},n}\left(-\frac{11}{12}\hat{p}_{m-2,n}+\frac{17}{24}\hat{p}_{m-1,n}+\frac{3}{8}\hat{p}_{m,n}-\frac{5}{24}\hat{p}_{m+1,n}+\frac{1}{24}\hat{p}_{m+2,n}\right)
\notag \\ &
-\frac{1}{24}A_{m+\frac{3}{2},n}\left(-\frac{1}{24}\hat{p}_{m-2,n}+\frac{5}{24}\hat{p}_{m-1,n}-\frac{3}{8}\hat{p}_{m,n}-\frac{17}{24}\hat{p}_{m+1,n}+\frac{11}{12}\hat{p}_{m+2,n}\right)
\notag \\ & \left.
+\frac{9}{8}A_{m+\frac{1}{2},n}\left(\frac{1}{24}\hat{p}_{m-1,n}-\frac{9}{8}\hat{p}_{m,n}+\frac{9}{8}\hat{p}_{m+1,n}-\frac{1}{24}\hat{p}_{m+2,n}\right)
\right],
\end{align}
\begin{align}\label{S1.eq.26}
\mathcal{\tilde{L}}_{z}p_{m ,n}&=
\frac{1}{\Delta z^2}\left[
-\frac{9}{8}B_{m,n-\frac{1}{2}}\left(\frac{1}{24}\tilde{p}_{m,n-2}-\frac{9}{8}\tilde{p}_{m,n-1}+\frac{9}{8}\tilde{p}_{m,n}-\frac{1}{24}\tilde{p}_{m,n+1}\right)
\right. \notag \\ &
+\frac{1}{24}B_{m,n-\frac{3}{2}}\left(-\frac{11}{12}\tilde{p}_{m,n-2}+\frac{17}{24}\tilde{p}_{m,n-1}+\frac{3}{8}\tilde{p}_{m,n}-\frac{5}{24}\tilde{p}_{m,n+1}+\frac{1}{24}\tilde{p}_{m,n+2}\right)
\notag \\ &
-\frac{1}{24}B_{m,n+\frac{3}{2}}\left(-\frac{1}{24}\tilde{p}_{m,n-2}+\frac{5}{24}\tilde{p}_{m,n-1}-\frac{3}{8}\tilde{p}_{m,n}-\frac{17}{24}\tilde{p}_{m,n+1}+\frac{11}{12}\tilde{p}_{m,n+2}\right)
\notag \\ & \left.
+\frac{9}{8}B_{m,n+\frac{1}{2}}\left(\frac{1}{24}\tilde{p}_{m,n-1}-\frac{9}{8}\tilde{p}_{m,n}+\frac{9}{8}\tilde{p}_{m,n+1}-\frac{1}{24}\tilde{p}_{m,n+2}\right)
\right].
\end{align}
Letting $\mathcal{\tilde{L}}=\mathcal{\tilde{L}}_{x}+\mathcal{\tilde{L}}_{z}$, the first two terms of the left hand side of \eqref{S1.eq.02} can be approximated as follows,
\begin{align}\label{S1.eq.27}
\frac{\partial}{\partial x}\left(A\frac{\partial p}{\partial x}\right)+\frac{\partial}{\partial z}\left(B\frac{\partial p}{\partial z}\right)
\approx
\mathcal{\tilde{L}}_{h} p_{m,n},
\end{align}
Similarly, let
\begin{align}\label{S1.eq.28}
\tilde{I}\left(k^2_{m,n}C_{m,n}p_{m,n}\right)&=\sum_{j=1}^{3}d_{j}\,I^{(j)}\left(k^2_{m,n}C_{m,n}p_{m,n}\right),
\end{align}
where $d_j$ are parameters satisfying $\sum_{j=1}^{3}d_j= 1$.

As a result, an optimal 17-point finite difference scheme for the Helmholtz-PML equation \eqref{S1.eq.02} can be obtained as follows,
\begin{align}\label{S1.eq.29}
\mathcal{\tilde{L}}\left(p_{m ,n}\right)+\tilde{I}\left(k^2_{m,n}C_{m,n}p_{m,n}\right)=\tilde{g}_{m,n}.
\end{align}
We refer the finite difference scheme \eqref{S1.eq.29} as optimal point-weighting 17-point finite difference.

Furthermore, let  $\mathcal{L}= \mathcal{L}_{x}p_{m ,n}+\mathcal{L}_{z}p_{m ,n}$. Then, we refer
\begin{align}
\label{S1.eq.30}
\mathcal{L}\left(p_{m ,n}\right)+k^2_{m,n}C_{m,n}p_{m,n}=\tilde{g}_{m,n},
\end{align}
as non-compact fourth-order (NC fourth-order). This scheme will be part of our final analysis in Section \ref{S3}.

In the next section, we will discuss minimizing the numerical dispersion using the dispersion relation formula.

\begin{definition}\label{S1.def1}
Let $(x_m ,z_n ) = (x_0 + m\Delta x,~z_0 + n\Delta z)$ for $m,n=0,1,2,\ldots$, and suppose that the partial differential equation under consideration is $(\Delta +k^2)p = g$, and the corresponding finite difference approximation is $\mathcal{L} P_{m,n} = G_{m,n}$ where $G_{m,n}=g(x_n,y_n)$.  The finite difference scheme $\mathcal{L} P_{m,n} = G_{m,n}$ is pointwise consistent with the partial differential equation $(\Delta +k^2)p = g$ at $(x,z)$ if for any smooth function $\phi(x,z)$,
\begin{align}\label{S1.def.01}
\left\|\left.\left((\Delta +k^2)\varphi - g\right) \right|_{x=x_m,~z=z_n} -\left[\mathcal{L} \varphi(x_m ,z_n) - G_{m,n} \right]\right\| \rightarrow 0
\end{align}
as $\Delta x,~\Delta z \to 0$.
\end{definition}

From now on, for simplicity, we set $\gamma=\Delta z/\Delta x$ (note that $\gamma$ is a positive constant) and let $\Delta x = h$, $\Delta z = \gamma h$ and $\eta=1+1/\gamma^2$.

\begin{proposition}\label{S1.P01}
If $\sum_{j=1}^{2}b_j= 1$, $\sum_{j=1}^{4}c_j= 1$, then optimal 25-point \eqref{S1.eq.22} and optimal 17-point \eqref{S1.eq.29} finite difference schemes are pointwise consistent with the Helmholtz-PML equation \eqref{S1.eq.02} and is a fourth-order scheme.
\end{proposition}
\begin{proof}
Let $(x,z)\in [x_m,x_{m+1})\times [x_n,x_{n+1})$. We expand all $p^{*}_{i,j}$, $p^{**}_{i,j}$, $\hat{p}_{i,j}$ and $\tilde{p}_{i,j}$ using the Taylor's theorem. For example, for $p_{m-2,n}^{*}$, we have,
\begin{align}
p_{m-2,n}^{*}&=b_{1}p-2h\frac{\partial }{\partial x} p
+2h^2\frac{\partial ^2}{\partial x^2} p
+\frac{4}{3}h^3\left(
3\gamma^2\left(b_{1}-1\right)\frac{\partial ^2}{\partial z^2} \frac{\partial }{\partial x} p
-\frac{\partial ^3}{\partial x^3} p
\right)
\notag \\&
-\frac{2}{3}h^4\left(
6\gamma^2\left(b_{1}-1\right)\frac{\partial ^2}{\partial z^2} \frac{\partial ^2}{\partial x^2} p
-\frac{\partial ^4}{\partial x^4} p \right)
+\frac{4}{15}h^5\left(
5\gamma^4\left(b_{1}-1\right)\left(2\frac{\partial ^2}{\partial z^2} \frac{\partial ^3}{\partial x^3} p
+\frac{\partial ^4}{\partial z^4} \frac{\partial }{\partial x} p\right)
\right.  \notag \\&\left.
-\frac{\partial ^5}{\partial x^5} p
\right)
-\frac{4}{45}h^6\left(
15\gamma^2\left(b_{1}-1\right)\left(\frac{\partial ^2}{\partial z^2} \frac{\partial ^4}{\partial x^4} p+\frac{\partial ^4}{\partial z^4} \frac{\partial ^2}{\partial x^2} p
-\frac{\partial ^6}{\partial x^6} p\right)\right)+O(h^7).
\end{align}
It follows from the Taylor's theorem that
\begin{align}
\label{S1.P01.01}
\mathcal{L}^{*}\left(p_{m ,n}\right)&=
\frac{\partial}{\partial x}\left(A\frac{\partial p}{\partial x}\right)+\frac{\partial}{\partial z}\left(B\frac{\partial p}{\partial z}\right)
+\zeta_1\,h^4+O(h^6),\\
\label{S1.P01.02}
\mathcal{\tilde{L}}\left(p_{m ,n}\right)&=
\frac{\partial}{\partial x}\left(A\frac{\partial p}{\partial x}\right)+\frac{\partial}{\partial z}\left(B\frac{\partial p}{\partial z}\right)
+\xi_1\,h^4+O(h^6),\\
\label{S1.P01.03}
I^{*}\left(k^2_{m,n}C_{m,n}p_{m,n}\right)&=k^2Cp + \zeta_2 h^4 + O\left(h^6\right),\\
\label{S1.P01.04}
\tilde{I}\left(k^2_{m,n}C_{m,n}p_{m,n}\right)&=k^2Cp + \xi_2 h^4 + O\left(h^6\right),
\end{align}
where $\zeta_1$, $\zeta_2$, $\xi_1$ and $\xi_2$ are given as follows

\begin{align}\label{S1.P01.05}
\zeta_{1}&=
\frac{a_{1}-1}{6}\left(
B\frac{\partial ^2}{\partial z^2} \frac{\partial ^4}{\partial x^4} p+\gamma^4A\frac{\partial ^4}{\partial z^4} \frac{\partial ^2}{\partial x^2} p \right)
-\frac{1}{90}\left( A\frac{\partial ^6}{\partial x^6} p+\gamma^4B\frac{\partial ^6}{\partial z^6} p
\right)
\notag \\ &
+\frac{a_{1}-1}{6}\left(
\gamma^4\frac{\partial ^4}{\partial z^4} \frac{\partial }{\partial x} p\frac{\partial }{\partial x} A
+\frac{\partial }{\partial z} \frac{\partial ^4}{\partial x^4} p\frac{\partial }{\partial z} B \right)
-\frac{1}{30}\left(
\frac{\partial ^5}{\partial x^5} p\frac{\partial }{\partial x} A
+\gamma^4\frac{\partial ^5}{\partial z^5} p\frac{\partial }{\partial z} B
\right)
\notag \\ &
-\frac{3}{64}\left(
\frac{\partial ^4}{\partial x^4} p\frac{\partial ^2}{\partial x^2} A
+\gamma^4\frac{\partial ^4}{\partial z^4} p\frac{\partial ^2}{\partial z^2} B
+ \frac{\partial ^3}{\partial x^3} A\frac{\partial ^3}{\partial x^3} p
+\gamma^4\frac{\partial ^3}{\partial z^3} p\frac{\partial ^3}{\partial z^3} B \right)
\notag \\ &
-\frac{3}{128}\left(
\frac{\partial ^4}{\partial x^4} A\frac{\partial ^2}{\partial x^2} p
+\gamma^4\frac{\partial ^4}{\partial z^4} B\frac{\partial ^2}{\partial z^2} p \right)
-\frac{3}{640}
\left(
\frac{\partial ^5}{\partial x^5} A\frac{\partial }{\partial x} p
+\gamma^4\frac{\partial ^5}{\partial z^5} B\frac{\partial }{\partial z} p \right),
\end{align}
\begin{align}\label{S1.P01.06}
\zeta_2&=
-\frac{\left( c_{2}+2c_{3}+2c_{4} \right)}{12}\left(\frac{\partial ^4}{\partial x^4} \left(k^2Cp\right)+\gamma^4\frac{\partial ^4}{\partial z^4} \left(k^2Cp\right)\right)
-c_{3}\gamma^2\frac{\partial ^2}{\partial z^2} \frac{\partial ^2}{\partial x^2} \left(k^2Cp\right),
\end{align}
\begin{align}\label{S1.P01.07}
\xi_{1}&=\frac{b_{1}-1}{6}\left( \frac{\partial ^2}{\partial z^2} \frac{\partial ^4}{\partial x^4} p+\gamma^2\frac{\partial ^4}{\partial z^4} \frac{\partial ^2}{\partial x^2} p \right) \left(A\gamma^2+B\right)
-\frac{1}{90}\left( A\frac{\partial ^6}{\partial x^6} p+\gamma^4B\frac{\partial ^6}{\partial z^6} p\right)
\notag \\ &
+\frac{b_{1}-1}{6}\left(
\gamma^4\frac{\partial ^4}{\partial z^4} \frac{\partial }{\partial x} p\frac{\partial }{\partial x} A
+2\gamma^2\frac{\partial ^2}{\partial z^2} \frac{\partial ^3}{\partial x^3} p\frac{\partial }{\partial x}A
+2\gamma^2\frac{\partial ^3}{\partial z^3} \frac{\partial ^2}{\partial x^2} p\frac{\partial }{\partial z} B
+\frac{\partial }{\partial z} \frac{\partial ^4}{\partial x^4} p\frac{\partial }{\partial z} B
\right)
\notag \\ &
-\frac{1}{30}\left(
\frac{\partial ^5}{\partial x^5} p\frac{\partial }{\partial x} A
+\gamma^4\frac{\partial ^5}{\partial z^5} p\frac{\partial }{\partial z} B \right)
+\frac{9}{32}(b_{1}-1)\gamma^2\left(
\frac{\partial ^2}{\partial z^2} \frac{\partial ^2}{\partial x^2} p\right)
\left(\frac{\partial ^2}{\partial x^2} A+\frac{\partial ^2}{\partial z^2} B
\right)
\notag \\ &
-\frac{3}{64}\left(
\frac{\partial ^3}{\partial x^3} A\frac{\partial ^3}{\partial x^3} p
+\gamma^4\frac{\partial ^3}{\partial z^3} p\frac{\partial ^3}{\partial z^3} B\right)
+\frac{9}{64}\gamma^2(b_{1}-1)\left(\frac{\partial ^3}{\partial x^3} A\frac{\partial ^2}{\partial z^2} \frac{\partial }{\partial x} p+\frac{\partial }{\partial z} \frac{\partial ^2}{\partial x^2} p\frac{\partial ^3}{\partial z^3} B \right)
\notag \\ &
-\frac{3}{128}\left(\frac{\partial ^4}{\partial x^4} A\frac{\partial ^2}{\partial x^2} p
+\gamma^4\frac{\partial ^4}{\partial z^4} B\frac{\partial ^2}{\partial z^2} p \right)
-\frac{3}{640}\left(\frac{\partial ^5}{\partial x^5} A\frac{\partial }{\partial x} p
+\gamma^4\frac{\partial ^5}{\partial z^5} B\frac{\partial }{\partial z} p\right),
\end{align}

\begin{align}\label{S1.P01.08}
\xi_{2}=&
-\frac{\left(d_{2}+2d_{3}\right)}{12}\left(
 \frac{\partial ^4}{\partial x^4} \left(k^2Cp\right)+\gamma^4\frac{\partial ^4}{\partial z^4} \left(k^2Cp\right)
\right)
-d_{3}\gamma^2\frac{\partial ^2}{\partial z^2} \frac{\partial ^2}{\partial x^2} \left(k^2Cp\right).
\end{align}

Let $\zeta = \zeta_1 + \zeta_2 $ and $\xi = \xi_1 + \xi_2 $, then it follows from \eqref{S1.eq.22} and \eqref{S1.eq.29} that
\begin{align}
\label{S1.P01.09}
\mathcal{L}^{*}\left(p_{m ,n}\right)+I^{*}\left(k^2_{m,n}C_{m,n}p_{m,n}\right)&=
\frac{\partial}{\partial x}\left(A\frac{\partial p}{\partial x}\right)+\frac{\partial}{\partial z}\left(B\frac{\partial p}{\partial z}\right)+C k^2 p
\notag \\ &
 + \zeta\,h^4 + O\left(h^6\right),\\
\label{S1.P01.10}
\mathcal{\tilde{L}}\left(p_{m ,n}\right)+\tilde{I}\left(k^2_{m,n}C_{m,n}p_{m,n}\right)&=
\frac{\partial}{\partial x}\left(A\frac{\partial p}{\partial x}\right)+\frac{\partial}{\partial z}\left(B\frac{\partial p}{\partial z}\right)+C k^2 p
\notag \\ &
 + \xi\,h^4 + O\left(h^6\right).
\end{align}
The results of this proposition can be concluded from \eqref{S1.P01.09}, \eqref{S1.P01.10} and \eqref{S1.eq.02}.
\end{proof}

What stands out from equations \eqref{S1.P01.09} and \eqref{S1.P01.10} is that optimal point-weighting 17-point and 25-point finite difference methods are fourth-order for arbitrary values of their parameters, $a_i$, $b_i$, $c_i$ and $d_j$, under the conditions $\sum_{i=1}^{2}a_i=1$, $\sum_{i=1}^{2}b_i=1$, \linebreak $\sum_{h=1}^{4}c_j=1$ and $\sum_{h=1}^{3}d_j=1$. Since $\zeta$ and $\eta$ depend on $k$, $A$ and $B$,  as shown in \eqref{S1.P01.05} - \eqref{S1.P01.08}, the convergence order and accuracy may be affected by the values of $k$ and its derivatives, especially when $k$ is large.

\section{Numerical dispersion analysis and parameter selection strategy}\label{S2}
In this section, we present numerical dispersion analyses for optimal point-weighting 17-point and 25-point finite difference methods. In doing so, consider an infinite homogeneous model with constant velocity $v$ to do a dispersion analysis. Let $P(x,z) = \exp\left({-i\,k(x\cos\theta +z \sin\theta )}\right)$, where $\theta$  is the propagation angle from the z-axis, and the wavenumber $k=2\pi f/v$ is also a positive constant.

In the interior area, $A = B = C = 1$, thus replacing $p_{m + i , n + j}$ with $P_{m + i , n + j}$ ( $i , j\in\mathbb{Z}_3$) in the optimal 25-point finite difference scheme \eqref{S1.eq.22} gives
\begin{align}\label{S2.eq.01}
&T^{*}_{1} \left(P_{m-2,n-2}+P_{m+2,n-2}+P_{m-2,n+2}+P_{m+2,n+2}\right) +T^{*}_{3} \left(P_{m,n-2}+P_{m,n+2}\right)+
\notag \\ &
T^{*}_{2} \left(P_{m-1,n-2}+P_{m+1,n-2}+P_{m-1,n+2}+P_{m+1,n+2}\right) +T^{*}_{6} \left(P_{m,n-1}+P_{m,n+1}\right)+
\notag \\ &
T^{*}_{4} \left(P_{m-2,n-1}+P_{m+2,n-1}+P_{m-2,n+1}+P_{m+2,n+1}\right)+T^{*}_{7} \left(P_{m-2,n}+P_{m+2,n}\right)+
\notag \\ &
T^{*}_{5} \left(P_{m-1,n-1}+P_{m+1,n-1}+P_{m-1,n+1}+P_{m+1,n+1}\right) +T^{*}_{8} \left(P_{m-1,n}+P_{m+1,n}\right)+
\notag \\ &
T^{*}_{9} P_{m,n}
=0,
\end{align}
\begin{align}\label{S2.eq.02}
\begin{cases}
\begin{array}{ll}
T^{*}_{1} = \dfrac{(1-a_{1})\eta }{72\,h^2}-\dfrac{3\,c_{3}-c_{4}}{36} k^2,&
T^{*}_{2} = \dfrac{\left(\eta +3 \right)\left(a_{1}-1\right)}{18\,h^2}-\dfrac{c_{4}}{9} k^2,\\ \vspace{-0.4cm} \\
T^{*}_{3} = \dfrac{5-a_{1}\left(\eta -4\right)}{12\,h^2}-\dfrac{c_{2}}{12} k^2,&
T^{*}_{4} = \dfrac{\left(4\,\eta -3 \right)\left( a_{1}-1\right)}{18\,h^2}-\dfrac{c_{4}}{9} k^2,\\ \vspace{-0.4cm} \\
T^{*}_{5} = \dfrac{8\,\eta\left(1-a_{1}\right)}{9\,h^2}+\dfrac{3\,c_{3}+4\,c_{4}}{9} k^2,&
T^{*}_{6} = \dfrac{a_{1}\left(4\,\eta +1\right) -5}{3\,h^2}+\dfrac{c_{2}}{3} k^2,\\ \vspace{-0.4cm} \\
T^{*}_{7} = \dfrac{4\,a_{1}+5\,\eta\left(1 -a_{1}\right) -5}{12\,h^2}-\dfrac{c_{2}}{12} k^2,&
T^{*}_{8} = \dfrac{5\,\eta\left(a_{1} -1\right) -a_{1}+5}{3\,h^2} + \dfrac{c_{2}}{3} k^2,\\ \vspace{-0.4cm} \\
T^{*}_{9} = -\dfrac{5\,a_{1}\,\eta }{2\,h^2}+\left(1-c_{2}-c_{3}-c_{4}\right)\,k^2.
\end{array}
\end{cases}
\end{align}
Similarly, for the optimal 17-point finite difference scheme \eqref{S1.eq.29}, we have,
\begin{align}\label{S2.eq.01A}
&\tilde{T}_{1} \left(P_{m-2,n-2}+ P_{m+2,n-2}+ P_{m-2,n+2}+ P_{m+2,n+2}\right)
+\tilde{T}_{2} \left(P_{m,n-2}+P_{m,n+2}\right)+
\notag \\ &
\tilde{T}_{3} \left(P_{m-1,n-1}+ P_{m+1,n-1}+ P_{m-1,n+1}+ P_{m+1,n+1}\right)
+\tilde{T}_{4} \left(P_{m,n-1}+ P_{m,n+1}\right)+
\notag \\ &
\tilde{T}_{5} \left(P_{m-2,n}+ P_{m+2,n}\right)+
+\tilde{T}_{6} \left(P_{m-1,n}+ P_{m+1,n}\right)
+\tilde{T}_{7} P_{m,n}=0
\end{align}
where
\begin{align}\label{S2.eq.02A}
\begin{cases}
\begin{array}{ll}
\tilde{T}_{1} = \dfrac{\left(b_{1}-1\right)\,\eta }{24\,h^2}-\dfrac{d_{3}}{12}k^2,&
\tilde{T}_{2} = \dfrac{-b_{1}\eta +1}{12\,h^2}-\dfrac{d_{2}}{12}k^2,\\ \vspace{-0.4cm} \\
\tilde{T}_{3} = \dfrac{2\left(1-b_{1}\right)\,\eta }{3\,h^2}+\dfrac{d_{3}}{3}k^2,&
\tilde{T}_{4} = \dfrac{4\,b_{1}\eta -4}{3\,h^2}+\dfrac{d_{2}}{3}k^2,\\ \vspace{-0.4cm} \\
\tilde{T}_{5} = \dfrac{\left(1-b_{1}\right)\,\eta -1}{12\,h^2}-\dfrac{d_{2}}{12}k^2,&
\tilde{T}_{6} = \dfrac{4\left(b_{1}-1\right)\,\eta +4}{3\,h^2}+\dfrac{d_{2}}{3}k^2,\\ \vspace{-0.4cm} \\
\tilde{T}_{7} = -\dfrac{5\,b_{1}\eta }{2\,h^2}+\left(1-d_{2}-d_{3}\right)\,k^2.
\end{array}
\end{cases}
\end{align}
Let $\lambda=2\pi v/\omega$ and $G=\lambda/h$ denote the wavelength and the number of gridpoints per wavelength, respectively. Moreover, let
\begin{align}\label{S2.eq.03}
\begin{cases}
P = \cos(k_x \Delta x) =\cos(kh\cos\theta ) = \cos\left(\dfrac{2\pi}{G}\cos\theta\right),
\\ \vspace{-0.4cm} \\
Q = \cos(k_z \Delta z) = \cos\left(\gamma kh\sin\theta\right)= \cos\left(\gamma\dfrac{2\pi}{G}\sin\theta\right).
\end{cases}
\end{align}

It follows from substituting $P_{m,n} = \exp\left(-ik(x\cos\theta +z \sin\theta )\right)$ into equations \eqref{S2.eq.01} and \eqref{S2.eq.01A}, and simplifying that
\begin{align}
\label{S2.eq.04}
&
4\,\left(2\,P^2-1\right)\,\left(2\,Q^2-1\right)\,T^{*}_{1}+4\,P\,\left(2\,Q^2-1\right)\,T^{*}_{2}+\left(4\,Q^2-2\right)\,T^{*}_{3}
+4\,Q\,\left(2\,P^2-1\right)\,T^{*}_{4}
\notag \\ &
+4\,P\,Q\,T^{*}_{5}+2\,Q\,T^{*}_{6}+\left(4\,P^2-2\right)\,T^{*}_{7}+2\,P\,T^{*}_{8}+T^{*}_{9}
=0,
\\
\label{S2.eq.04A}
&4\,\left(2\,P^2-1\right)\,\left(2\,Q^2-1\right)\,\tilde{T}_{1}+\left(4\,Q^2-2\right)\,\tilde{T}_{2}+4\,P\,Q\,\tilde{T}_{3}
+2\,Q\,\tilde{T}_{4}+\left(4\,P^2-2\right)\,\tilde{T}_{5}
\notag \\ &
+2\,P\,\tilde{T}_{6}+\tilde{T}_{7}=0.
\end{align}

Furthermore, let $k^{*}_{N}$ and $\tilde{k}_{N}$ represent the numerical wavenumber for the finite difference schemes \eqref{S1.eq.22} and \eqref{S1.eq.29}, respectively. For the optimal point-weighting 25-point finite difference method \eqref{S1.eq.22}, it follows from replacing the variable $k$ in the parameters $T^{*}_{1},~T^{*}_{2},\ldots,T^{*}_{9}$ with $k_{N}$ in equation \eqref{S2.eq.04} that
\begin{align}\label{S2.eq.05}
k^{*}_{N}=\frac{1}{h}\sqrt{\frac{N^{*}}{D^{*}}},
\end{align}
where
\begin{align}
\label{S2.eq.06}
N^{*} &=
\eta\left(1-a_{1}\right)\left(2P^2Q^2+32PQ\right)
+4\left(4\eta -3 \right)\left( a_{1}-1\right)P^2Q
+\left(4 \left( \eta +3 \right)\left( a_{1}-1\right)\right)PQ^2
\notag \\ &
+\left(14\left(1-a_{1}\right)\eta +3\left(4a_{1}-5\right)\right)P^2
-\left(\left(2a_{1}+1\right)\eta + \left(4a_{1}-5\right)\right)Q^2
-7\left(2a_{1}+1\right)\eta
\notag \\ &
+\left(28\left(a_{1}-1\right)\eta +12\left(3-a_{1}\right)\right)P
+\left(8\left(2a_{1}+1\right)\eta +12\left(a_{1} - 3\right)\right)Q,
\\
\label{S2.eq.07}
D^{*}&=
4 \left(3c_{3}-c_{4}\right)P^2Q^2
+8c_{4}\left(P^2Q+PQ^2\right)
+\left(3c_{2}-6c_{3}+2c_{4}\right)\left(P^2+Q^2\right)
\notag \\ &
-4 \left(3c_{3}+4c_{4}\right)PQ
-2\left(3c_{2}+2c_{4}\right)\left(P+Q\right)
+6c_{2}+12c_{3}+8c_{4}-9.
\end{align}
Similarly, for the optimal point-weighting 17-point finite difference method \eqref{S1.eq.29}, we have,
\begin{align}\label{S2.eq.05}
\tilde{k}_{N}=\frac{1}{h}\sqrt{\frac{\tilde{N}}{\tilde{D}}},
\end{align}
with
\begin{align}
\label{S2.eq.06A}
\tilde{N} &=
2\eta\left(b_{1}-1 \right)P^2Q^2+\left( \left(2-2b_{1}\right)\eta -1 \right)P^2
-8\eta \left(b_{1}-1 \right)PQ
\notag \\ &
+8\left(\left(b_{1}-1\right)\eta +1\right)P
+\left(\left(1-2b_{1}\right)\eta +1\right)Q^2
+8\left( b_{1}\eta -1\right)Q-\eta -6b_{1}\eta,
\\
\label{S2.eq.07A}
\tilde{D}&=
\left(P^2-2P+Q^2-2Q+2\right)d_{2}
+2\left(2P^2Q^2-P^2-2PQ-Q^2+2\right)d_{3}-3.
\end{align}

The next proposition presents the error between the numerical wavenumbers ($k^{*}_{N}$ and $\tilde{k}_{N}$) and the exact wavenumber $k$ for the finite difference schemes \eqref{S1.eq.22} and \eqref{S1.eq.29}.

\begin{proposition}\label{S2.prop.2}
For optimal point-weighting 25-point and 17-point finite difference methods, respectively, there holds
\begin{align}
\label{S2.P01}
\left(k^{*}_{N}\right)^2&=k^{2}\left(1+O\left(k^{6}\,h^{6}\right)\right),&k\,h\to 0,
\\
\label{S2.P01A}
\left(\tilde{k}_{N}\right)^2&=k^{2}\left(1+O\left(k^{6}\,h^{6}\right)\right),&k\,h\to 0.
\end{align}
\end{proposition}
\begin{proof}
Letting $\tau = kh$, $P(\tau)=\cos(\tau\cos\theta )$ and $Q(\tau)=\cos(\gamma\tau\sin\theta )$, equation \eqref{S2.eq.05} can be expressed as follows,
\begin{align}\label{S2.P01.01}
\left(k^{*}_{N}\right)^2=\frac{1}{h^2}\frac{N^{*}(\tau)}{D^{*}(\tau)},
\end{align}
where
\begin{align*}
N^{*}(\tau)&=
2\left(1-a_{1}\right)\left[
\left(\gamma^2+1\right)P^2Q^2
-2\left( \gamma^2+4 \right)P^2Q
+2\left(4\gamma^2+1 \right)PQ\left(8-Q\right)\right]
\notag \\ &
+2\gamma^2\left(2a_{1}+1\right)P\left(4-P\right)
+14\gamma^2\left(1-a_{1}\right)Q\left(Q-2\right)
+14\left(1-a_{1}\right)\left(P-2\right)
\notag \\ &
+\left(2a_{1}+1\right)Q\left(8-Q\right)
-7 \left(\gamma^2+1 \right)\left( 2a_{1}+1\right)
,\\
D^{*}(\tau)&=
4\gamma^2\left(3c_{3}-c_{4}\right)P^2Q^2
+8c_{4}\gamma^2\left(P^2Q+PQ^2\right)
+\gamma^2\left(3c_{2}-6c_{3}+2c_{4}\right)\left(P^2+Q^2\right)
\notag \\ &
-2\gamma^2\left(3c_{2}+2c_{4}\right)\left(P+Q\right)
-4\gamma^2\left(3c_{3}+4c_{4}\right)PQ
+\gamma^2\left(6c_{2}+12c_{3}+8c_{4}-9\right).
\end{align*}
Applying Taylor theorem on $N^{*}(\tau)$ and $\frac{1}{D^{*}(\tau)}$ at point $\tau = 0$, we have
\begin{align}
\label{S2.P01.04}
N^{*}(\tau)&=
-9\gamma^2\tau ^2
+
\frac{\gamma^2\tau ^6}{10}
\left(
3\left(4-5a_{1}\right)\sin^2\left(\theta \right)
+\left(15\left(1-a_{1}\right)\gamma^4+30a_{1}-27\right)\sin^4\left(\theta \right)
\right. \notag \\ & \left.
+\left(\gamma^2-1 \right)\left( \gamma^2+1 \right)\left( 15a_{1}-14\right)\sin^6\left(\theta \right)+1
\right)
+O(\tau^8)
,
\\
\label{S2.P01.05}
\frac{1}{D^{*}(\tau)}&=
-\frac{1}{9\gamma^2}
-\frac{\tau ^4}{108\gamma^2}\left(
\gamma^4\left(c_{2}+2c_{3}+2c_{4}\right)\sin^4\left(\theta \right)
+12c_{3}\gamma^2\sin^2\left(\theta \right)\cos^2\left(\theta \right)
\right. \notag \\ & \left.
+\left(c_{2}+2c_{3}+2c_{4}\right){\cos\left(\theta \right)}^4
\right)
+O(\tau^6).
\end{align}
It follows from \eqref{S2.P01.01}, \eqref{S2.P01.04} and \eqref{S2.P01.05} that
\begin{align}
\label{S2.P01.06}
\left(k^{*}_{N}\right)^2&=k^{2}
-\frac{k^6h^4}{180}
\left(
2\left(\gamma^2-1\right)\left( \gamma^2+1 \right)\left( 15a_{1}-14\right)\sin^6\left(\theta \right)
+\left(
6 \left( 30c_{3}\gamma^2+15a_{1}-14\right)
\right.\right. \notag \\ & \left.
+15\left(g^4+1\right)\left(2-c_{2}-2c_{3}-2c_{4}-2a_{1}\right)
\right)\sin^4\left(\theta \right)
+2-15 \left( c_{2}+2c_{3}+2c_{4}\right)
\notag \\ & \left.
-6\left(
5 \left(6c_{3}\gamma^2+a_{1}-c_{2}-2c_{3}-2c_{4}\right)
-4\right)\sin^2\left(\theta \right)
\right)+O\left(k^{8}h^{6}\right),\quad kh\to 0.
\end{align}
Thus,
\begin{align}
\left(k^{*}_{N}\right)^2&=k^{2}\left(1+O\left(k^{6}\,h^{6}\right)\right),&k\,h\to 0
\end{align}
Similarly, for the optimal point-weighting 17-point finite difference method
 \eqref{S1.eq.29}, we have,
\begin{align}\label{S2.P01.01A}
\left(\tilde{k}_{N}\right)^2=\frac{1}{h^2}\frac{\tilde{N}(\tau)}{\tilde{D}(\tau)},
\end{align}
where
\begin{align*}
\tilde{N}(\tau)&=
2\left(\gamma^2+1 \right)\left(  b_{1}-1\right)P^2Q^2
+\left(\left(1-2b_{1}\right)\gamma^2+2-2b_{1}\right)P^2
-8\left( \gamma^2+1 \right)\left( b_{1}-1 \right)PQ
\\ &
+8\left(\left(\gamma^2+1\right)b_{1}-1\right)P
+8\left(\left(\gamma^2+1\right)b_{1}-\gamma^2 \right)Q^2
+8\left(\left(\gamma^2+1\right)b_{1}-\gamma^2 \right)Q
\\ &
-6b_{1}\left(\gamma^2+1\right)-\gamma^2-1
,\\
\tilde{D}(\tau)&=
\left(\gamma^2\left(P^2-2P+Q^2-2Q+2\right)\right)d_{2}
+2\gamma^2\left(2P^2Q^2-P^2-2PQ-Q^2+2\right)d_{3}-3\gamma^2.
\end{align*}
Applying Taylor theorem at point $\tau =0$,
\begin{align}
\label{S2.P01.04A}
\tilde{N}(\tau)&=
-3\gamma^2\tau ^2
+
\frac{\gamma^2\tau ^6}{30}
\left(
(\gamma-1)( \gamma+1)( \gamma^2+1) \left( 15b_{1}-14\right){\sin^6\left(\theta \right)}
\right. \notag \\ &
-3\left(\left(5b_{1}-5\right)\gamma^4+\left(5-5b_{1}\right)\gamma^2+9-10b_{1}\right){\sin^4\left(\theta \right)}
\notag \\ & \left.
-3\left(\left(5b_{1}-5\right)\gamma^2+5b_{1}-4\right){\sin^2\left(\theta \right)}+1
\right)
+O(\tau^8)
,
\\
\label{S2.P01.05A}
\frac{1}{\tilde{D}(\tau)}&=
-\frac{1}{3\gamma^2}+
\frac{\tau ^4}{36\gamma^2}\left(
\gamma^4\left(d_{2}+2d_{3}\right){\sin^4\left(\theta \right)}
+12d_{3}\gamma^2{\sin^2\left(\theta \right)}{\cos^2\left(\theta \right)}
\right. \notag \\ & \left.
+\left(d_{2}+2d_{3}\right){\cos^4\left(\theta \right)}
\right)
+O(\tau^6).
\end{align}
It follows from \eqref{S2.P01.01A}, \eqref{S2.P01.04A} and \eqref{S2.P01.05A} that
\begin{align}
\label{S2.P01.06A}
\left(\tilde{k}_{N}\right)^2&=k^{2}
-\frac{k^6h^4}{180}
\left(2-15\left(d_{2}+2d_{3} \right)+
2\left(g-1 \right)\left( g+1 \right)\left( \gamma^2+1 \right)\left( 15b_{1}-14 \right){\sin^6\left(\theta \right)}
\right. \notag \\ &
-3\left(
5\left(2b_{1}+d_{2}+2d_{3}-2\right)\gamma^4
-10\left(b_{1}+6d_{3}-1 \right)\gamma^2
-5\left(4b_{1}-d_{2}-2d_{3} \right)+18\right){\sin^4\left(\theta \right)}
\notag \\ & \left.
-6\left(5\left(b_{1}+6d_{3}-1\right)\gamma^2+5\left(b_{1}-d_{2}-2d_{3} \right)-4\right){\sin^2\left(\theta \right)}
\right)+O\left(k^{8}h^{6}\right),\quad kh\to 0.
\end{align}
Therefore,
\begin{align}
\left(\tilde{k}_{N}\right)^2&=k^{2}\left(1+O\left(k^{6}\,h^{6}\right)\right),&k\,h\to 0.
\end{align}
This completes the proof.
\end{proof}
The above proposition indicates that $k^*_{N}$ and $\tilde{k}_{N}$ approximate $k$ with fourth-order accuracy. Moreover, the terms associated with $k^6\,h^4$ (in equations \eqref{S2.P01.06} and \eqref{S2.P01.06A}) presents the pollution effect, which depends on the wavenumber $k$, the parameters of the finite difference formula and the wave’s propagation angle $\theta$ from the z-axis.

Based on minimizing the numerical dispersion, we next introduce how to choose the weights $a_1$, $c_1$, \ldots, $c_4$ and $b_1$, \ldots, $d_3$ for optimal point-weighting 25-point and 17-point finite difference methods, respectively.

Let $k_{N}$ the numerical wavenumber (either $k^*_{N}$ or $\tilde{k}_{N}$). Similarly, let $N$ represent either $N^{*}$ or $\tilde{N}$, and, similarly, let $D$  represent either $D^{*}$ or $\tilde{D}$. Since $h =2\pi/Gk$, the relationship of the numerical wavenumber $k_{N}$ and the exact wavenumber $k$ can be presented as follows,
\begin{align}\label{S2.eq.08}
\frac{k_{N}}{k}=\frac{G}{2\pi}\sqrt{\frac{N}{D}}.
\end{align}

Furthermore, the normalized numerical phase velocity and the normalized numerical group velocity can be found as follows, respectively \cite{jo1996optimal,shin1998frequency,trefethen1982group,chen2013optimal}.
\begin{align}
\label{S2.eq.09}
\frac{V^{N}_{ph}}{v}&=\frac{G}{2\pi}\sqrt{\frac{N}{D}},
\\
\label{S2.eq.10}
\frac{V^{N}_{gr}}{v}&=\frac{v}{V^{N}_{ph}}\left[
\frac{\left(\frac{1}{h}\frac{\partial N}{\partial k}\right)D-N\left(\frac{1}{h}\frac{\partial D}{\partial k}\right)}{D^2}
\right].
\end{align}

As can be seen, there would be no numerical dispersion if the normalized numerical phase velocity equals to one. Therefore, to minimize the error between $k_{N}$ and $k$, we need to estimate the parameters of optimal point-weighting 17-point and 25-point finite difference methods in a way that the normalized numerical phase velocity can take a value close to one.

Hence, the following functionals can be considered for minimizing the numerical dispersion of finite difference schemes \eqref{S1.eq.22} and \eqref{S1.eq.29}, respectively,
\begin{align}
\label{S2.eq.11}
J^{*}(a_1,\ldots,c_4;G,\theta)&=\frac{G}{2\pi}\sqrt{\frac{N^{*}}{D^{*}}}-1,\\
\label{S2.eq.11A}
\tilde{J}(b_1,\ldots,d_3;G,\theta)&=\frac{G}{2\pi}\sqrt{\frac{\tilde{N}}{\tilde{D}}}-1.
\end{align}
where $b_1 \in (0,1]$, and $c_2,c_3,c_4\in \mathbb{R}$, and $(G,\theta) \in I_G \times I_\theta$ with $I_G$ and $I_\theta$ are two intervals.
In general, one can choose $I_\theta = \left[0,\dfrac{\pi}{2}\right]$ and $I_G= [G_{\min},~G_{\max}] \subseteq [2,~400]$. We remark that the interval $\left[0,\pi/2\right]$ can be replaced by $\left[0,~\pi/4\right]$
because of the symmetry, and $G_{\min} \geq  2$ based on the Nyquist sampling limit (see \cite{shin1998frequency} for more details).

Furthermore, minimizing the numerical dispersion of optimal point-weighting 25-point and 17-point finite difference methods is equivalent to minimizing functionals \eqref{S2.eq.11} and \eqref{S2.eq.11A}, respectively. In doing so, we can use the least-squares method to estimate the optimal parameters of finite difference schemes \eqref{S1.eq.22} and \eqref{S1.eq.29}. Therefore, it follows from $J^{*}(a_1,\ldots,c_4;G,\theta)=0$ and $\tilde{J}(b_1,\ldots,d_3;G,\theta)=0$ that

\begin{align}\label{S2.eq.13}
\begin{cases}
\dfrac{G^2}{4\pi^2}\dfrac{N^{*}}{D^{*}}=1,
\\ \vspace{-0.4cm} \\
\dfrac{G^2}{4\pi^2}\dfrac{\tilde{N}}{\tilde{D}}=1.
\end{cases}
\end{align}
In other words,
\begin{align}\label{S2.eq.14}
&
\left(
\left( Q-1 \right)\left( Q-7 \right)\left( 2\,P^2-4\,P-1 \right)\,\eta
3\left( P-Q \right)\left( 4\,P\,Q-5\,Q-5\,P+12 \right)
\right)G^2
\notag \\ &
+36\pi ^2
-2\left( Q-1 \right)\left( P-1 \right)\left( \eta PQ+\left(6-7\eta \right)P+\left(-\eta -6\right)Q+7\eta\right)
a_{1}
\notag \\ &
-12 \pi^2\left( P^2-2P+Q^2-2Q+2 \right)c_{2}
-24\pi^2\left( 2P^2Q^2-P^2-2PQ-Q^2+2\right)c_{3}
\notag \\ &
+
8\pi^2\left(2P^2Q^2-4P^2Q-P^2-4PQ^2+8PQ+2P-Q^2+2Q-4 \right)c_{4}
=0,
\end{align}
and
\begin{align}\label{S2.eq.14A}
&
-2 \eta \left(Q-1 \right)\left( P-1 \right)\left( P+Q+PQ-3 \right)G^2b_{1}
+4\pi^2\left(P^2-2P+Q^2-2Q+2\right)d_{2}
\notag \\ &
+\left(\left( Q-1 \right)\left( 2P^2Q-Q-8P+2P^2-1 \right)\eta+\left( P-Q \right)\left( P+Q-8 \right)\right)G^2-12\pi ^2
\notag \\ &
+8\pi^2\left( 2P^2Q^2-P^2-2PQ-Q^2+2 \right)d_{3}=0
\end{align}
Let
\begin{align*}
\begin{cases}
\theta=\theta_m=\dfrac{(m-1)}{4(l-1)}\pi \in I_\theta=\left[0,\dfrac{\pi}{2}\right],&m=1,~2,~\ldots,l,
\\ \vspace{-0.7cm} \\
\dfrac{1}{G}=\dfrac{1}{G_n}=\dfrac{1}{G_{\max}}+(n-1)\dfrac{\frac{1}{G_{\min}}-\frac{1}{G_{\max}}}{r-1}\in \left[\dfrac{1}{G_{\max}},\dfrac{1}{G_{\min}}\right],
&n = 1,~2,~\ldots,r.
\end{cases}
\end{align*}
Equations \eqref{S2.eq.14} and \eqref{S2.eq.14A} can be expressed as the following linear systems, respectively,
\begin{align}\label{S2.eq.15}
\begin{bsmallmatrix}
S_{1,1}^1& S_{1,1}^2& S_{1,1}^3& S_{1,1}^4 \\
\vdots& \vdots& \vdots& \vdots \\
S_{1,r}^1& S_{1,r}^2& S_{1,r}^3& S_{1,r}^4 \\
\vdots& \vdots& \vdots& \vdots \\
S_{m,n}^1& S_{m,n}^2& S_{m,n}^3& S_{m,n}^4 \\
\vdots& \vdots& \vdots& \vdots \\
S_{l,r}^1& S_{l,r}^2& S_{l,r}^3& S_{l,r}^4 \\
\end{bsmallmatrix}
\begin{bsmallmatrix}
a_1\\
c_2\\
c_3\\
c_4
\end{bsmallmatrix}
=
\begin{bsmallmatrix}
S_{1,1}^5 \\
\vdots \\
S_{1,r}^5 \\
\vdots \\
S_{m,n}^5 \\
\vdots \\
S_{l,r}^5 \\
\end{bsmallmatrix},
\end{align}
and
\begin{align}\label{S2.eq.15A}
\begin{bsmallmatrix}
W_{1,1}^1& W_{1,1}^2& W_{1,1}^3 \\
\vdots& \vdots& \vdots \\
W_{1,r}^1& W_{1,r}^2& W_{1,r}^3 \\
\vdots& \vdots& \vdots \\
W_{m,n}^1& W_{m,n}^2& W_{m,n}^3 \\
\vdots& \vdots& \vdots \\
W_{l,r}^1& W_{l,r}^2& W_{l,r}^3 \\
\end{bsmallmatrix}
\begin{bsmallmatrix}
b_1\\
d_2\\
d_3\\
\end{bsmallmatrix}
\begin{bsmallmatrix}
W_{1,1}^4 \\
\vdots \\
W_{1,r}^4 \\
\vdots \\
W_{m,n}^4 \\
\vdots \\
W_{l,r}^4 \\
\end{bsmallmatrix},
\end{align}
where
\begin{align}\label{S2.eq.16}
&
\begin{cases}
S_{m,n}^{1}&=
-2G_{n}^2\left(Q_{m,n}-1 \right)\left( P_{m,n}-1 \right)\left(\eta\left(P_{m,n}-1\right)\left(Q_{m,n} -7\right)
\right. \\ & \left.
+ 6\left(P_{m,n}-Q_{m,n}\right)\right)
,\\
S_{m,n}^{2}&=
-12 \pi^2\left(P_{m,n}^2+Q_{m,n}^2 -2 \left( P_{m,n}+Q_{m,n}-1 \right)\right)
,\\
S_{m,n}^{3}&=
-24\pi^2\left(\left(2P_{m,n}^2-1\right)Q_{m,n}^2-2 \left(P_{m,n}Q_{m,n}-1\right)-P_{m,n}^2\right)
,\\
S_{m,n}^{4}&=
8\pi^2\left(\left(2P_{m,n}^2-4P_{m,n}-1\right)Q_{m,n}\left(Q_{m,n}-2\right)-P_{m,n}^2+2P_{m,n}-4\right) )
,\\
S_{m,n}^{5}&=
-\left(\left(Q_{m,n}-1\right)\left( Q_{m,n}-7 \right)\left( 2P_{m,n}^2-4P_{m,n}-1 \right)\eta
\right. \\ & \left.
+3\left(P_{m,n}-Q_{m,n} \right)\left(\left(4P_{m,n}-5\right)Q_{m,n}-5P_{m,n}+12\right)\right)G_{n}^2.
\end{cases}
\\
&
\label{S2.eq.16A}
\begin{cases}
W_{m,n}^{1}&=
-2 \eta \left(Q_{m,n}-1 \right)\left( P_{m,n}-1 \right)\left( P_{m,n}+Q_{m,n}+P_{m,n}Q_{m,n}-3 \right)G_{n}^2
,\\
W_{m,n}^{2}&=
4\pi^2\left(P_{m,n}^2-2P_{m,n}+Q_{m,n}^2-2Q_{m,n}+2\right)
,\\
W_{m,n}^{3}&=
8\pi^2\left( 2P_{m,n}^2Q_{m,n}^2-P_{m,n}^2-2P_{m,n}Q_{m,n}-Q_{m,n}^2+2 \right)
,\\
W_{m,n}^{4}&=
\left(\left( Q_{m,n}-1 \right)\left( 2P_{m,n}^2Q_{m,n}-Q_{m,n}-8P_{m,n}+2P_{m,n}^2-1 \right)\eta
\right. \\ & \left.
+\left( P_{m,n}-Q_{m,n} \right)\left( P_{m,n}+Q_{m,n}-8 \right)\right)G_{n}^2-12\pi ^2.
\end{cases}
\end{align}
with
\begin{align}\label{S2.eq.17}
P_{m,n}=\cos\left(\frac{2\pi}{G_{n}}\cos\left(\theta _{m}\right)\right), \quad
Q_{m,n}=\cos\left(\gamma \frac{2\pi}{G_{n}} \sin\left(\theta _{m}\right)\right).
\end{align}

We next present an algorithm for parameter selection and reducing the numerical dispersion and improve the accuracy of optimal point-weighting 17-point and 25-point finite difference methods. The optimal algorithm is based on the refined choice strategy (rule 3.8 from \cite{chen2013optimal}). According to the rule, first, the interval $I_G = \left[G_{min},~G_{max} \right]$ is estimated by using a priori information. For example, for a given step size of $h$, $I_G$ can be considered as follows,
\begin{align}\label{S2.eq.18}
I_G= \left[\frac{v_{\min}}{hf_{\max}},\frac{v_{\max}}{hf_{\min}}\right],
\end{align}
where $f \in \left[f_{\min},~f_{\max}\right]$ and $v \in \left[v_{\min},~v_{\max}\right]$ are the frequency and the velocity, respectively.
Then, the parameters of optimal point-weighting 25-point and 17-point finite difference methods, respectively, are estimated such that
\begin{align}
\label{S2.eq.19}
(a_1,c_2,c_3,c_4) &=
arg \min\left\{\|J^{*}(.,\ldots,.;G,\theta)\|_{I_G \times I_\theta}:~a_1 \in (0,1],~c_2,c_3,c_4\in \mathbb{R} \right\},
\\
\label{S2.eq.19A}
(b_1,d_2,d_3) &=
arg \min\left\{\|\tilde{J}(.,\ldots,.;G,\theta)\|_{I_G \times I_\theta}:~b_1 \in (0,1],~d_2,d_3\in \mathbb{R} \right\}.
\end{align}

In the remaining of this article, we refer optimal point-weighting 25-point and 17-point finite difference methods whose parameters are estimated using the refined choice strategy as the refined 25-point finite difference scheme (refined PW 25p) and the refined 17-point finite difference scheme (refined PW 17p), respectively. Moreover, As a summary consider algorithms \ref{Alg01} and \ref{Alg01A} for optimal parameters selection for refined PW 25p and refined PW 17p, respectively.

\begin{algorithm}[htp]
 \KwData{$v$, $f$ and $h$}
 \KwResult{$a_1$, $c_2$, $c_3$ and $c_4$}
 Identify $I_G$ using \eqref{S2.eq.18}\\
 Solve the least square problem \eqref{S2.eq.15} for $a_{1},\ldots,c_4$\;
 \caption{Optimal parameters selection for refined PW 25p}\label{Alg01}
\end{algorithm}
\begin{algorithm}[htp]
 \KwData{$v$, $f$ and $h$}
 \KwResult{$b_1$, $d_2$ and $d_3$}
 Identify $I_G$ using \eqref{S2.eq.18}\\
 Solve the least square problem \eqref{S2.eq.15A} for $a_{1},\ldots,c_4$\;
 \caption{Optimal parameters selection for refined PW 17p}\label{Alg01A}
\end{algorithm}

The normalized phase and group velocity curves of refined PW 25p and refined PW 17p are presented in the figures \ref{S2.fig.01} and \ref{S2.fig.02}, respectively. The optimal parameters are estimated on $I_G$ intervals $[2,2.5]$, $[2.5,3]$ $[4,5]$, $[5,6]$, $[6,8]$, $[8,10]$, and $[10,400]$. What can be seen from the figures is the these two schemes have pretty similar normalized phase velocity and normalized group velocity curves for various values of $\gamma$.

\begin{figure}[htp]
\centering
\subfloat[refined PW 25p ($\gamma=0.25$)]{\includegraphics[width=0.35\textwidth]{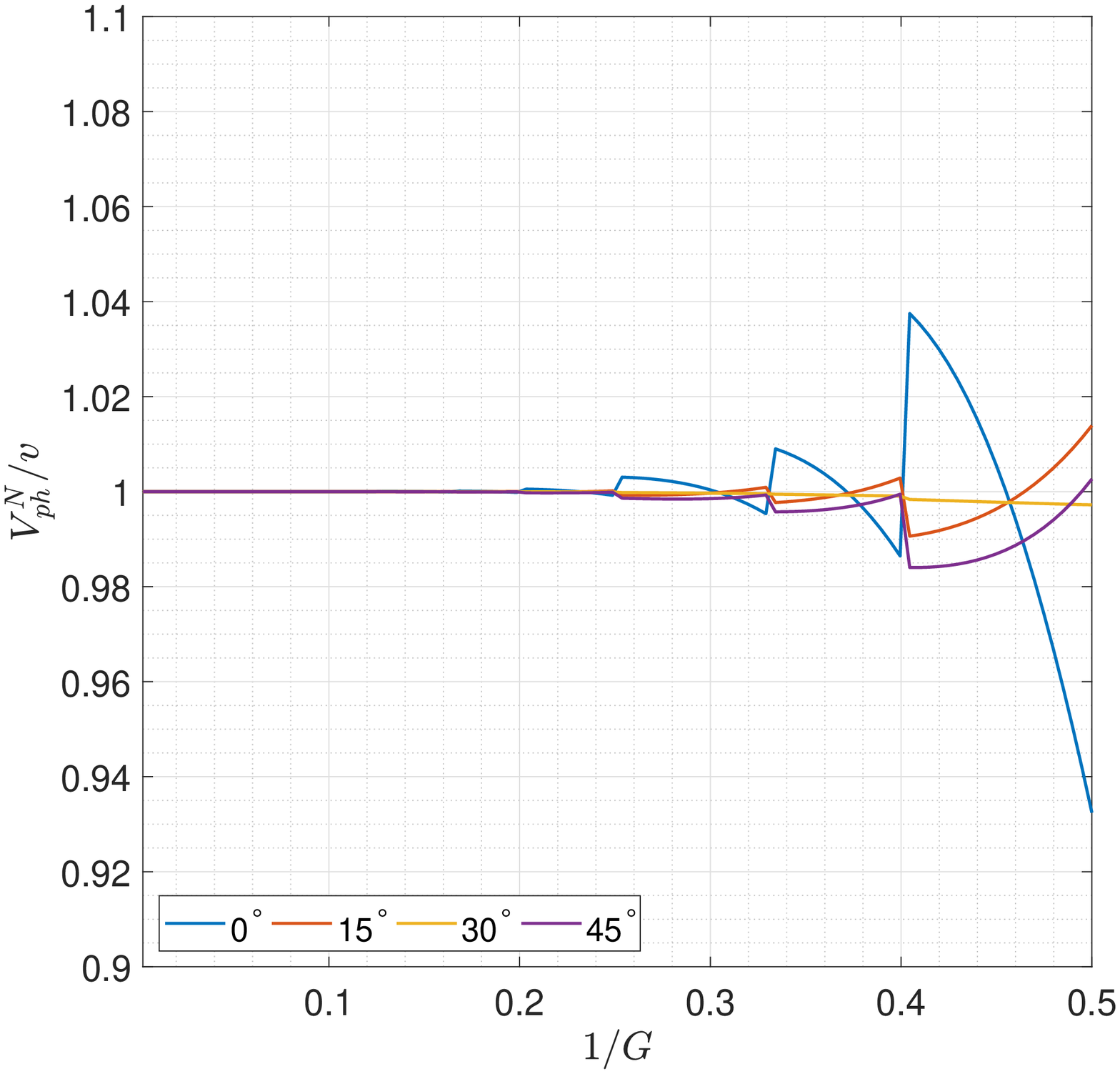}\label{S2.fig.01a}}
~
\subfloat[refined PW 25p ($\gamma=0.50$)]{\includegraphics[width=0.35\textwidth]{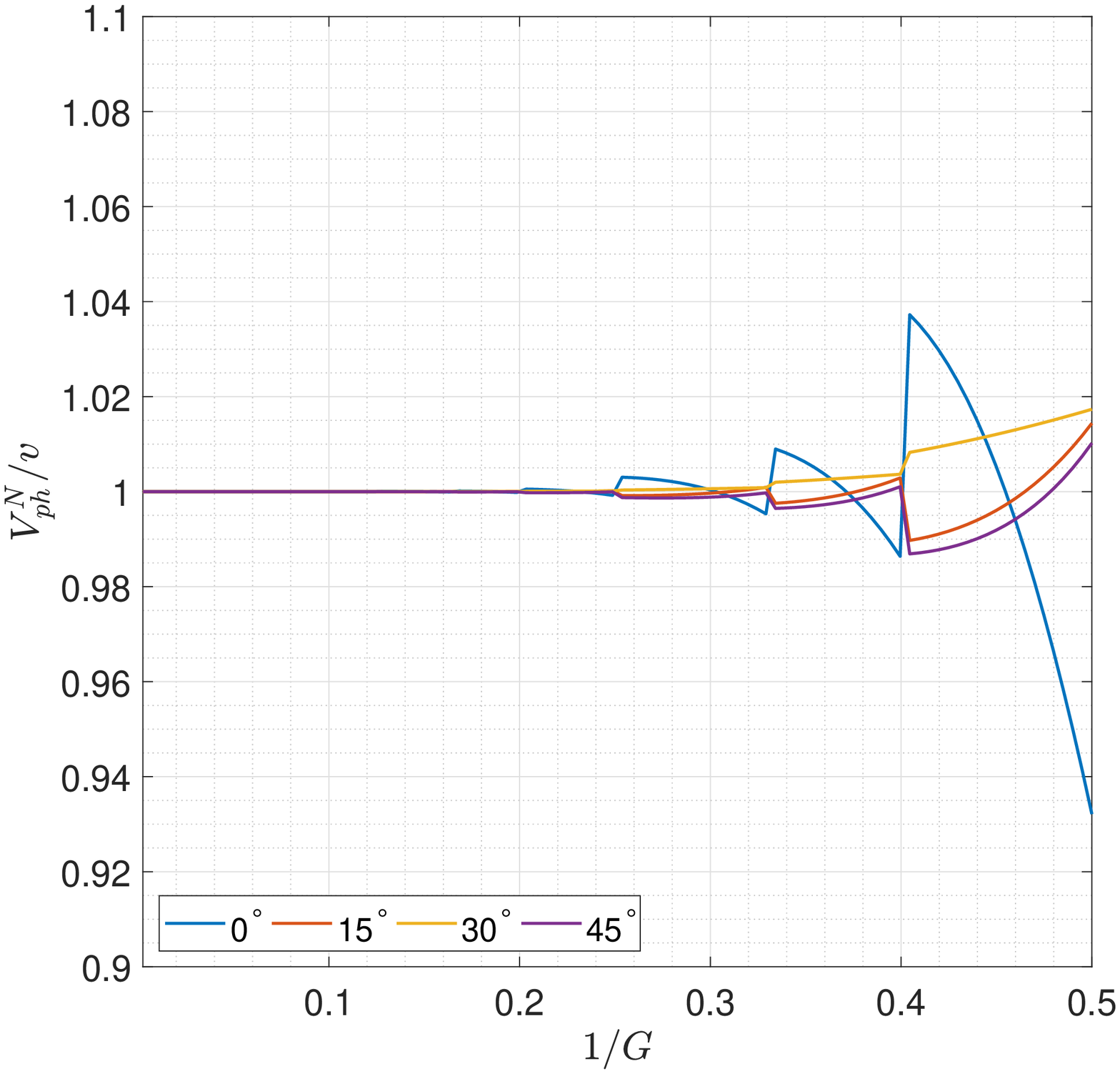}\label{S2.fig.01b}}
~
\subfloat[refined PW 25p ($\gamma=0.75$)]{\includegraphics[width=0.35\textwidth]{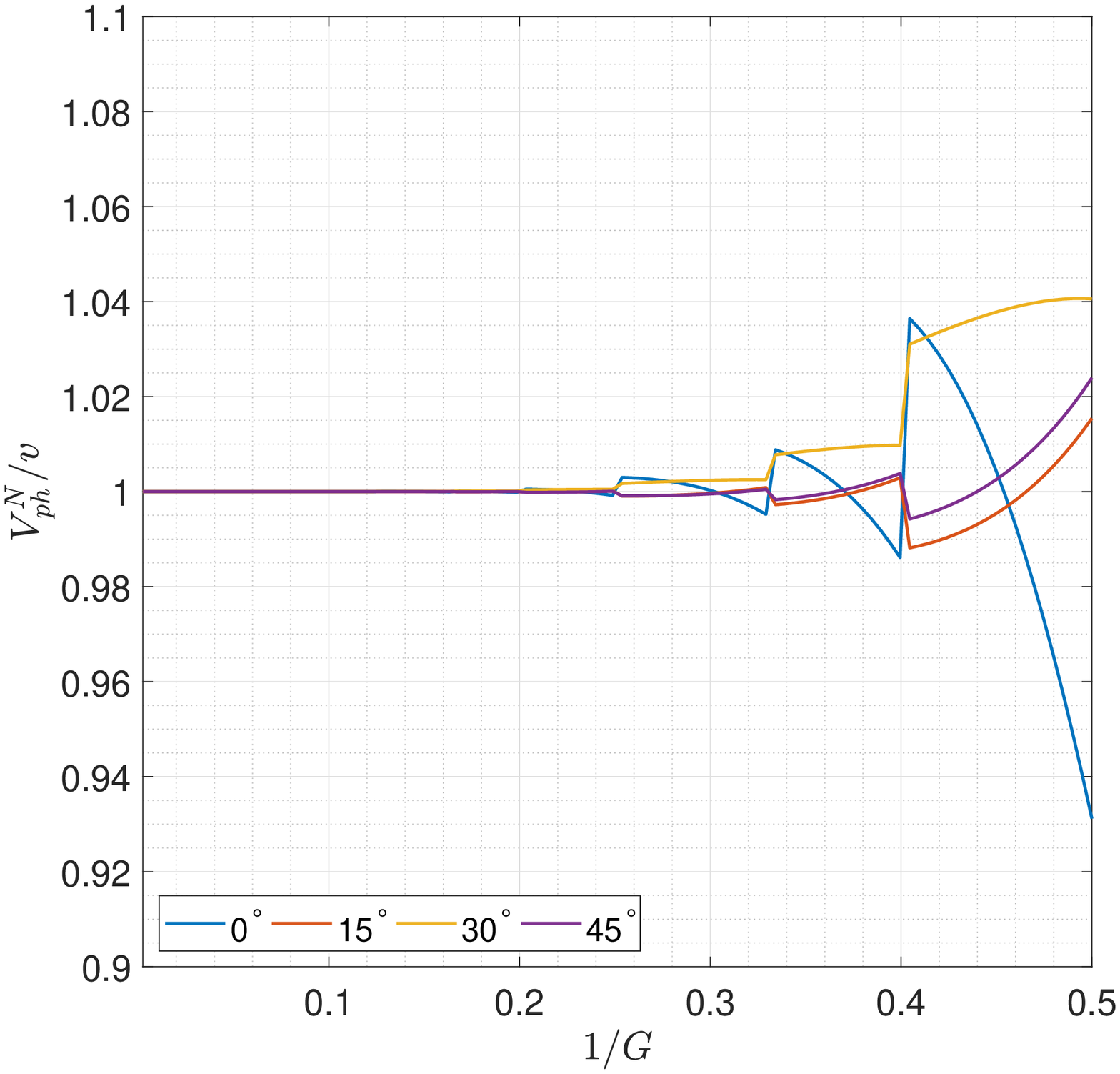}\label{S2.fig.01c}}

\subfloat[refined PW 25p ($\gamma=1.00$)]{\includegraphics[width=0.35\textwidth]{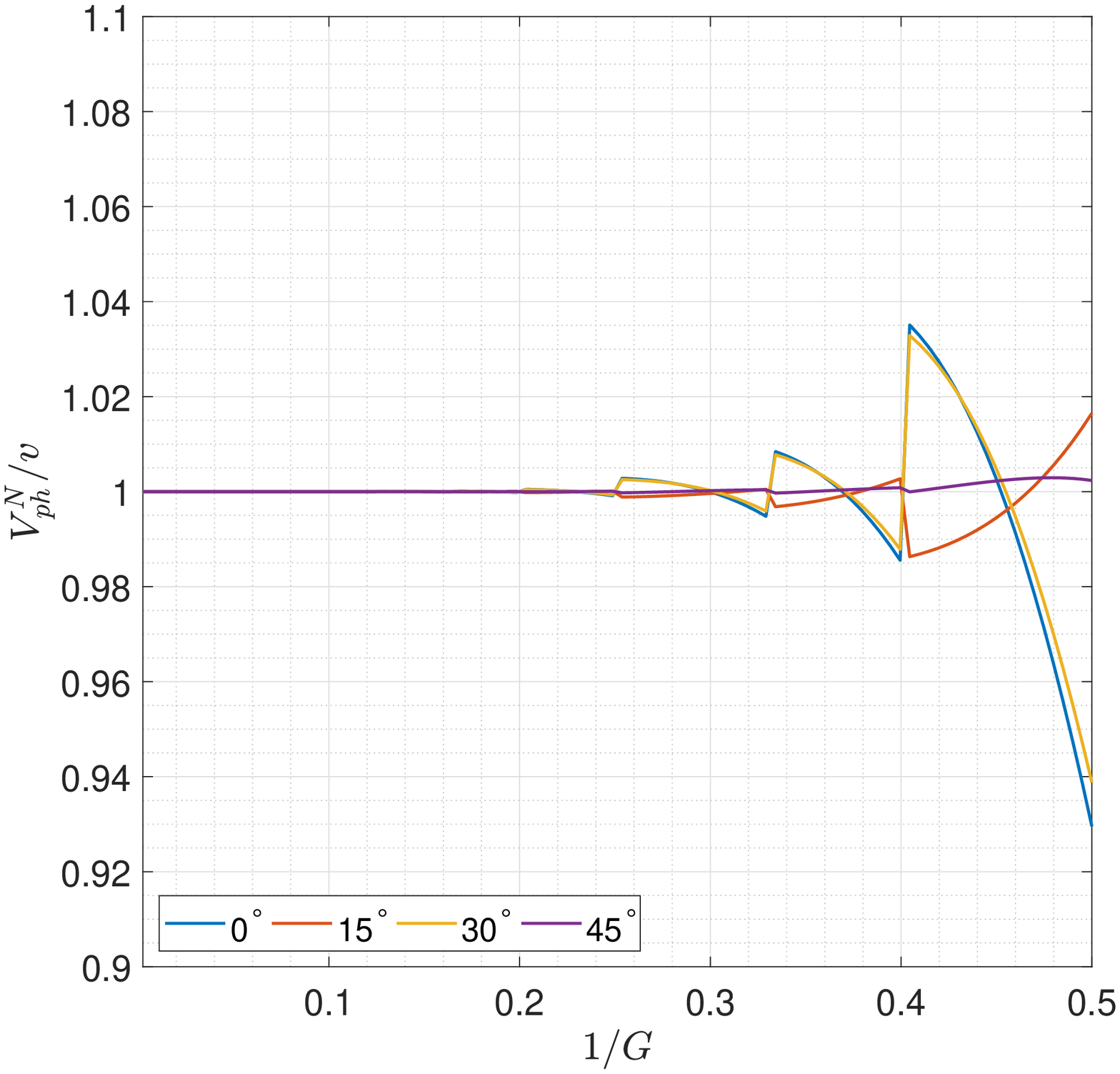}\label{S2.fig.01d}}
~
\subfloat[refined PW 17p ($\gamma=0.25$)]{\includegraphics[width=0.35\textwidth]{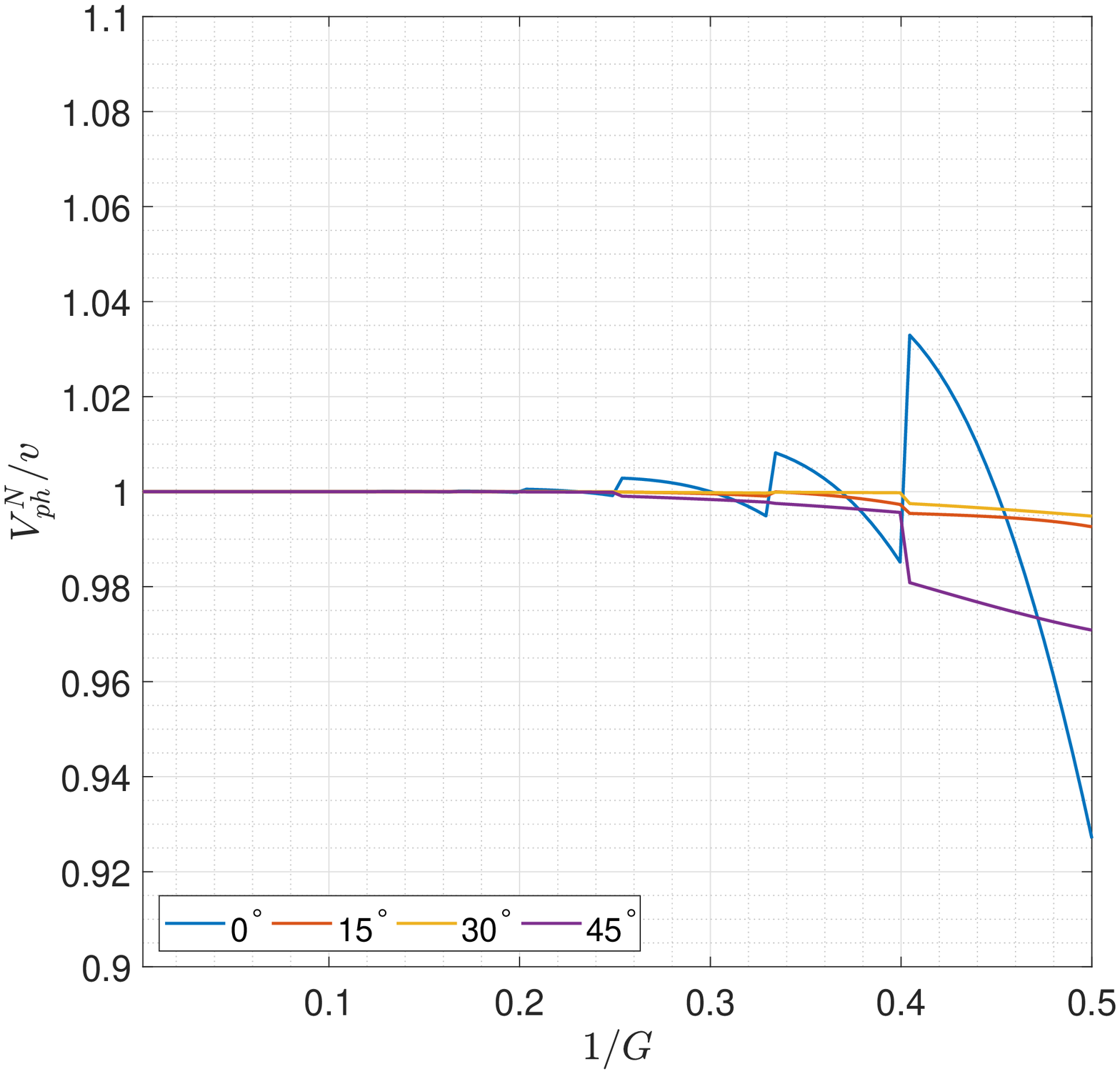}\label{S2.fig.01e}}
~
\subfloat[refined PW 17p ($\gamma=0.50$)]{\includegraphics[width=0.35\textwidth]{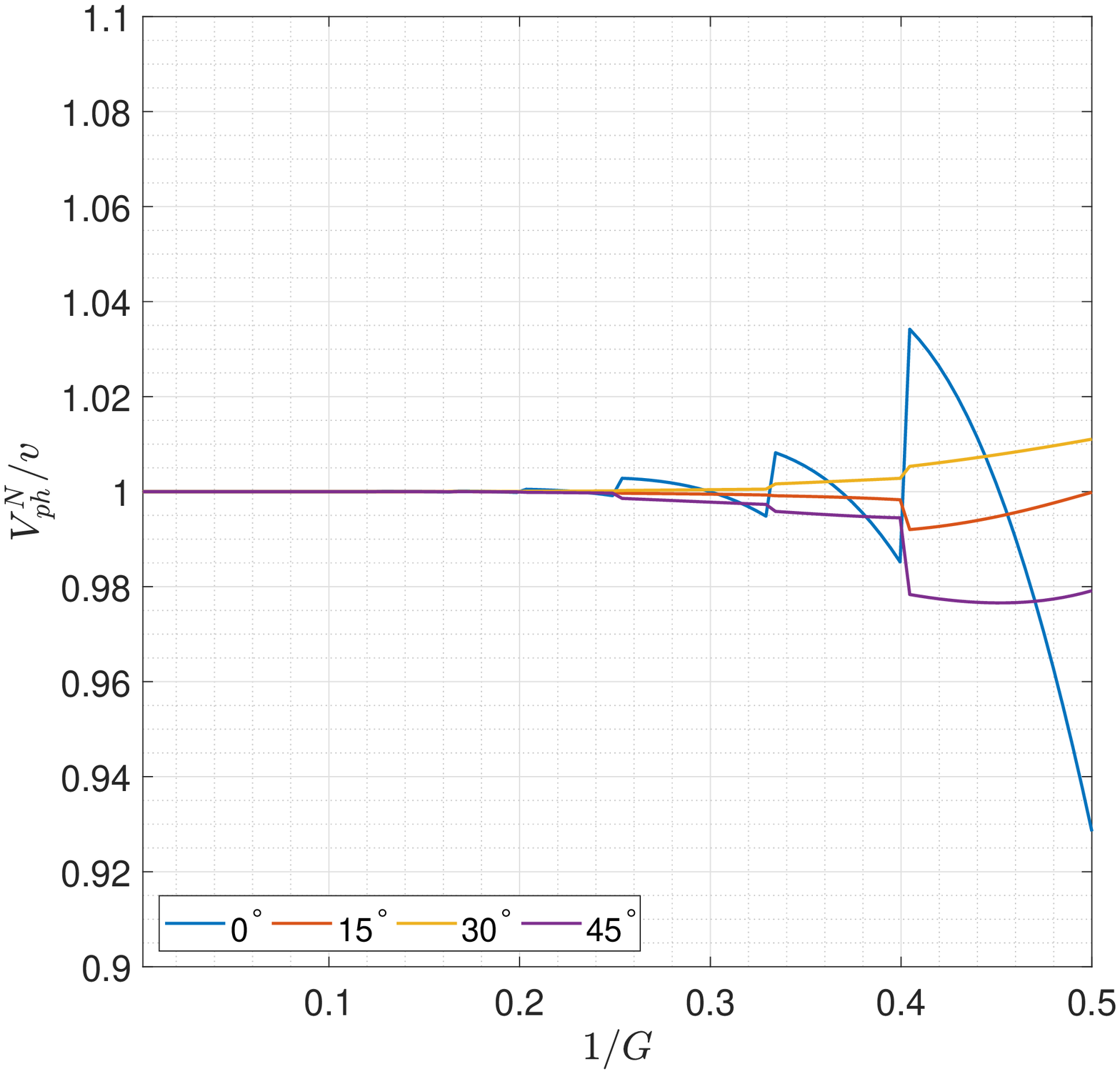}\label{S2.fig.01f}}

\subfloat[refined PW 17p ($\gamma=0.75$)]{\includegraphics[width=0.35\textwidth]{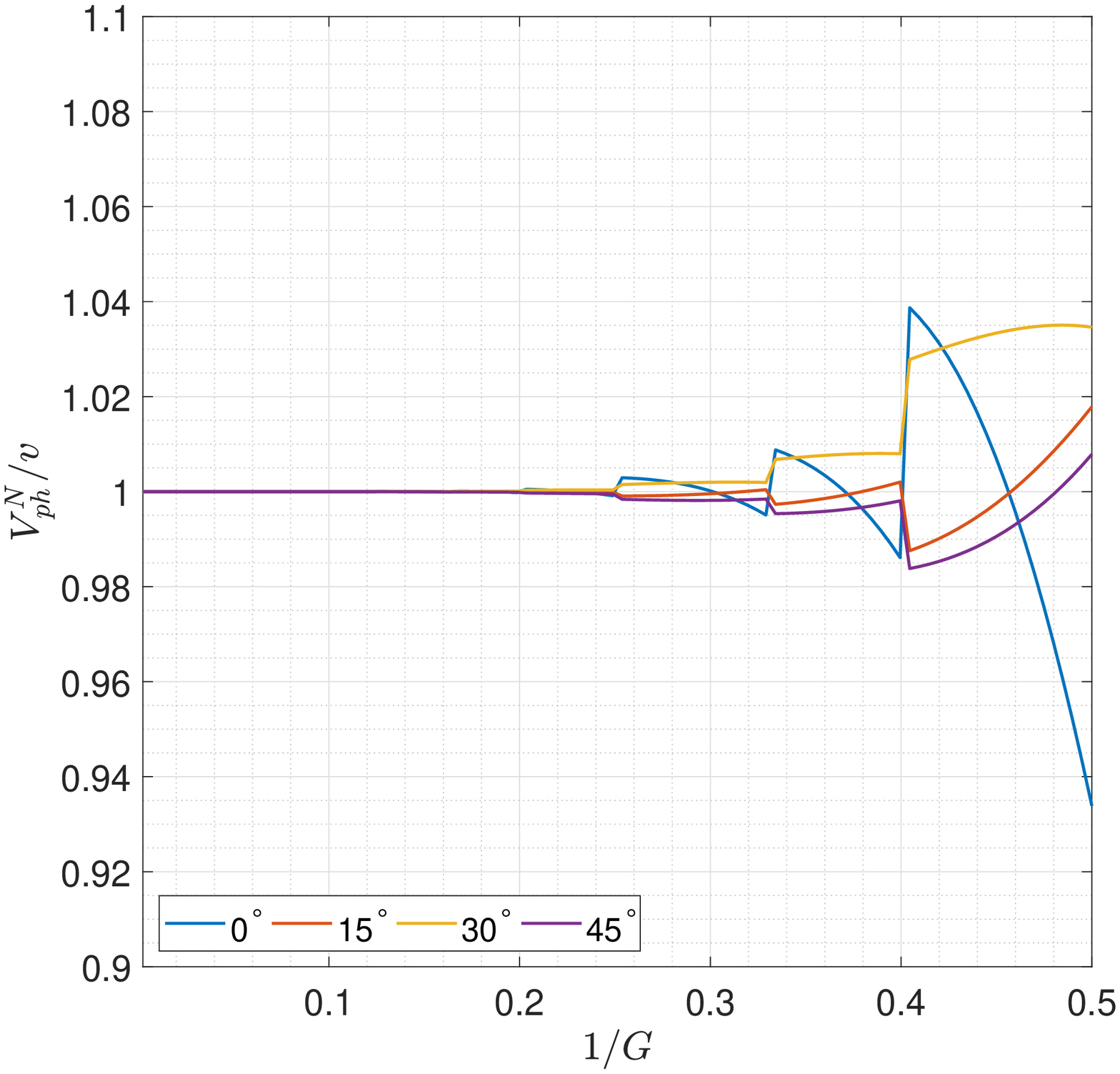}\label{S2.fig.01g}}
~
\subfloat[refined PW 17p ($\gamma=1.00$)]{\includegraphics[width=0.35\textwidth]{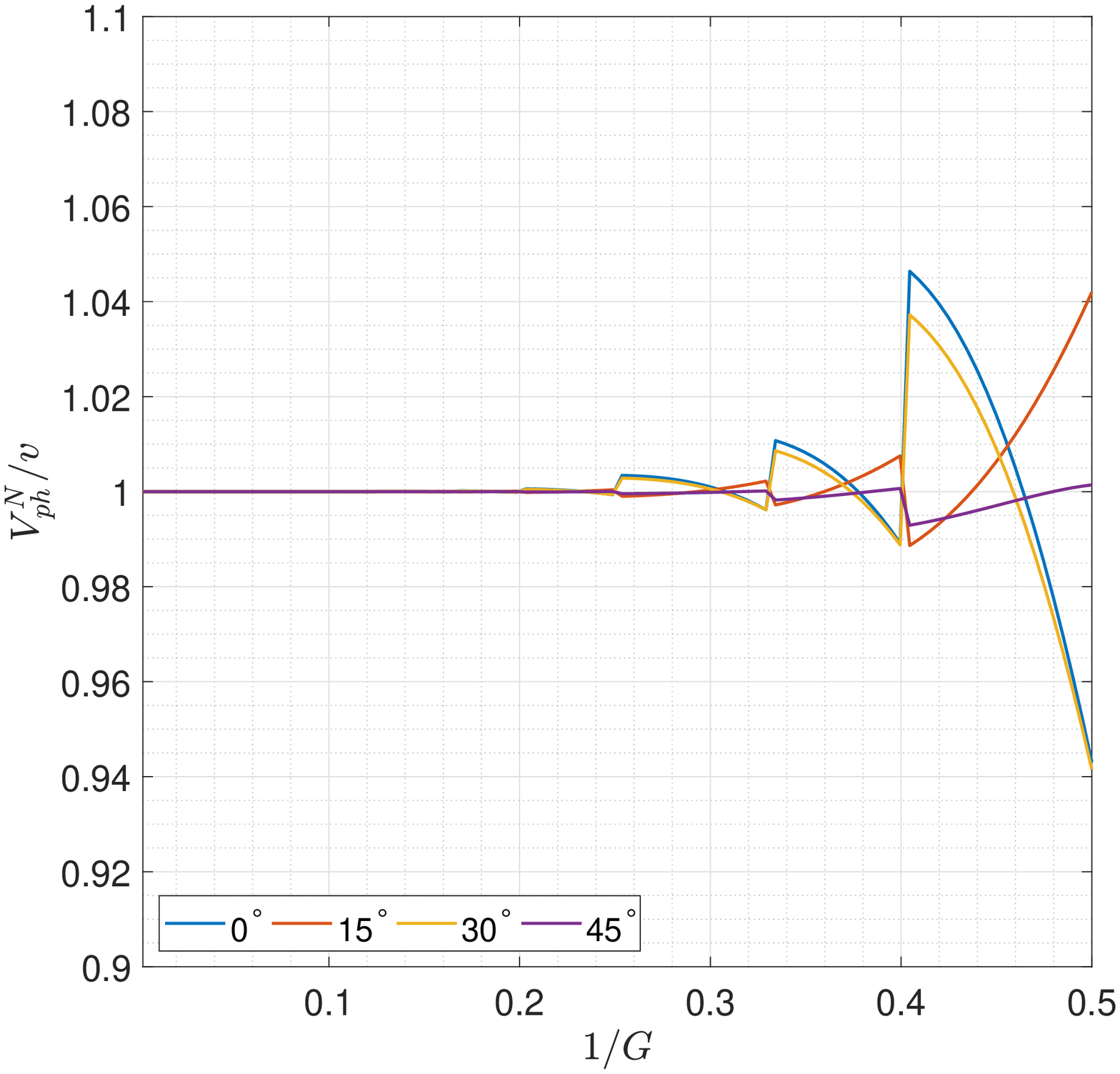}\label{S2.fig.01h}}
\caption{Normalized phase velocity curves for various schemes}\label{S2.fig.01}
\end{figure}

\begin{figure}[htp]
\centering
\subfloat[refined PW 25p ($\gamma=0.25$)]{\includegraphics[width=0.35\textwidth]{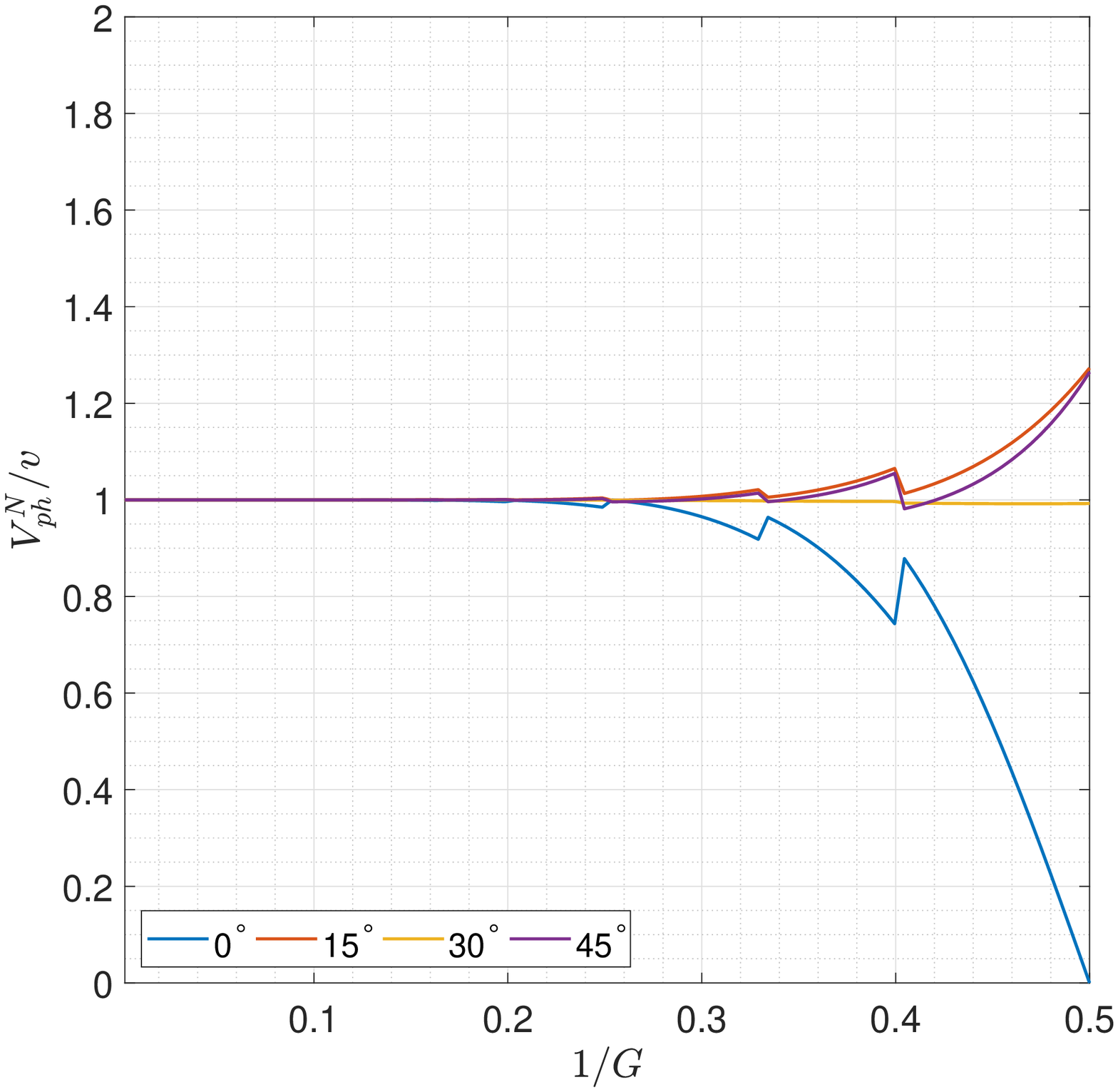}\label{S2.fig.02a}}
~
\subfloat[refined PW 25p ($\gamma=0.50$)]{\includegraphics[width=0.35\textwidth]{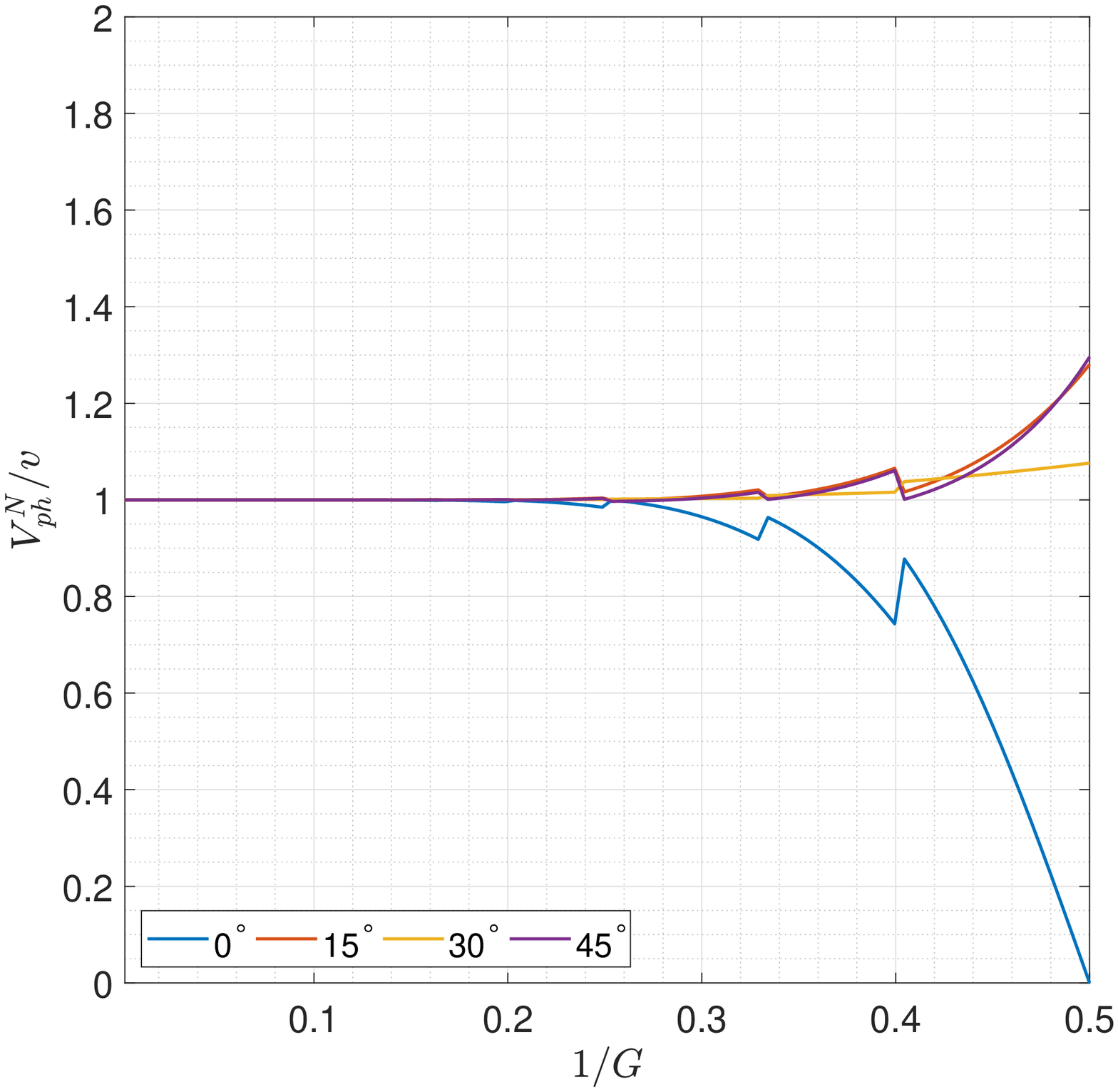}\label{S2.fig.02b}}
~
\subfloat[refined PW 25p ($\gamma=0.75$)]{\includegraphics[width=0.35\textwidth]{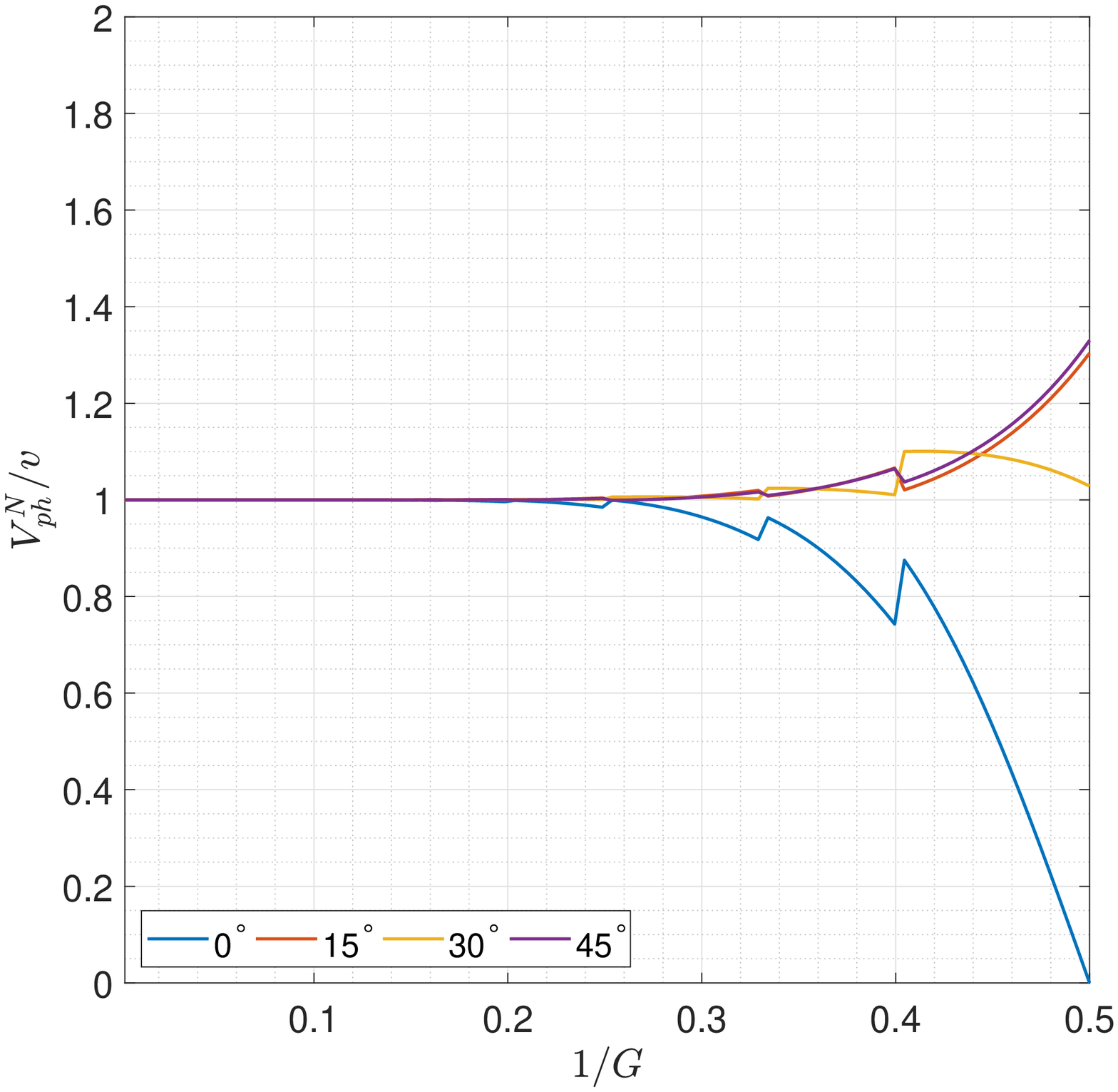}\label{S2.fig.02c}}

\subfloat[refined PW 25p ($\gamma=1.00$)]{\includegraphics[width=0.35\textwidth]{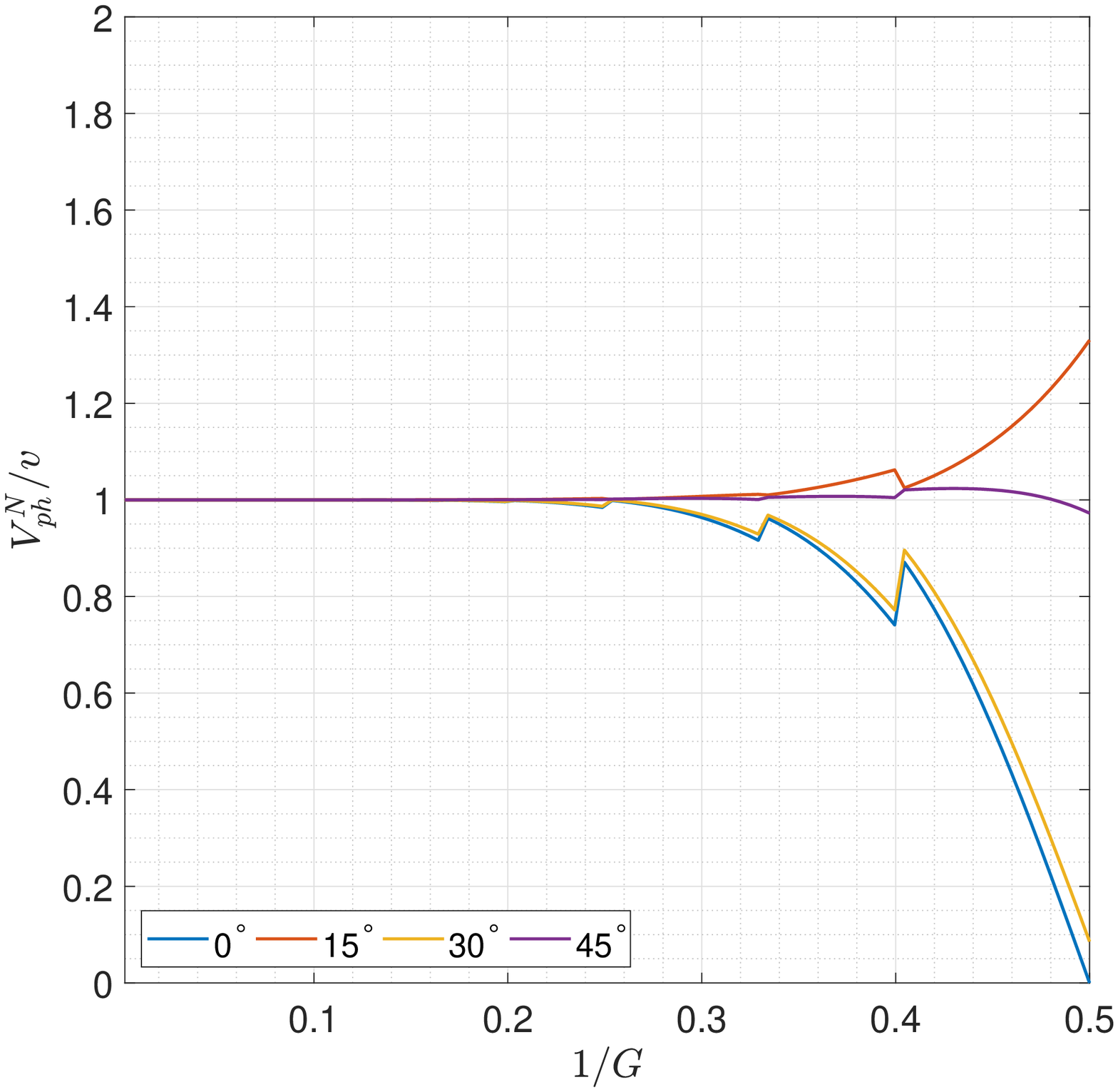}\label{S2.fig.02d}}
~
\subfloat[refined PW 17p ($\gamma=0.25$)]{\includegraphics[width=0.35\textwidth]{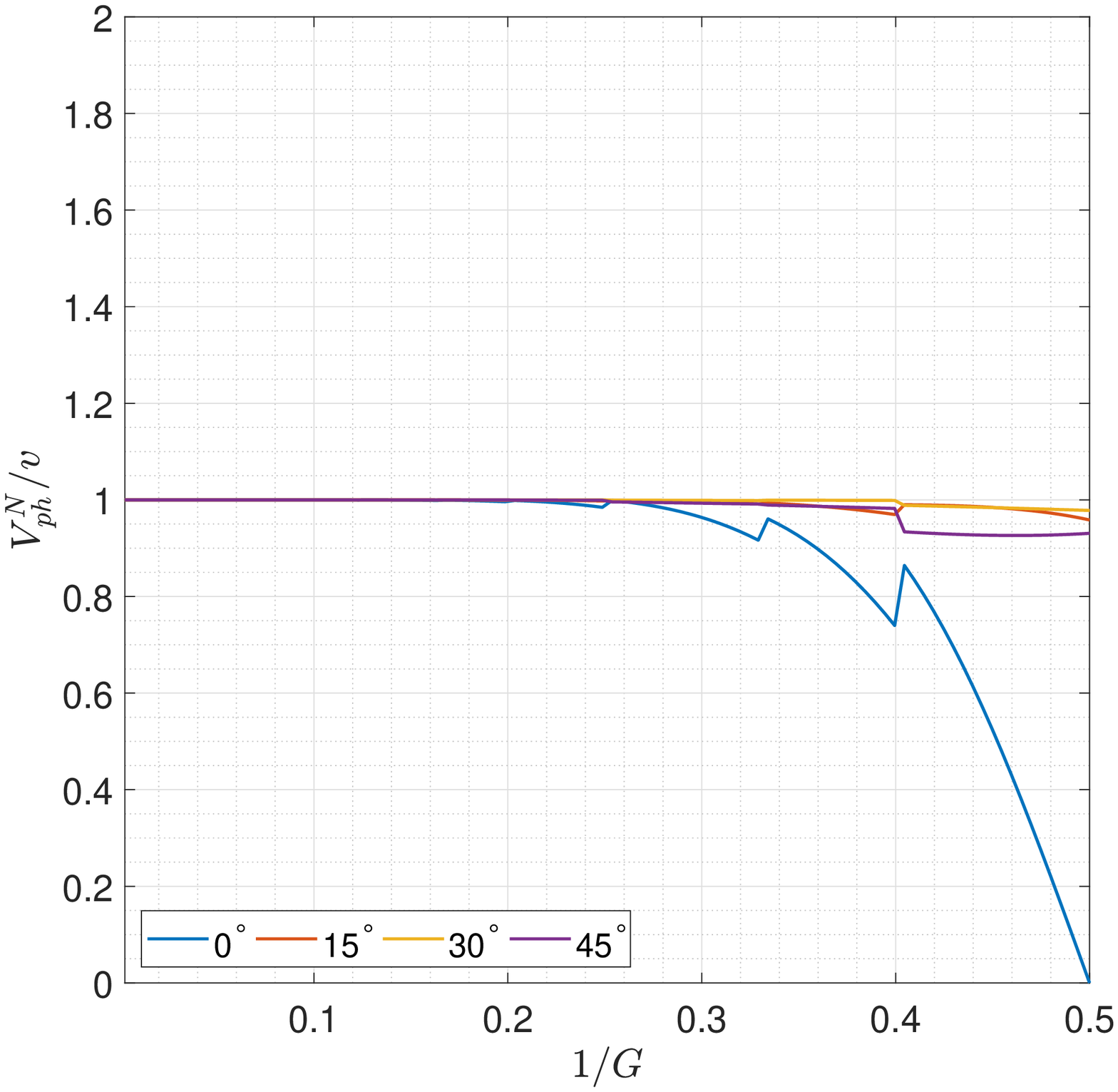}\label{S2.fig.02e}}
~
\subfloat[refined PW 17p ($\gamma=0.50$)]{\includegraphics[width=0.35\textwidth]{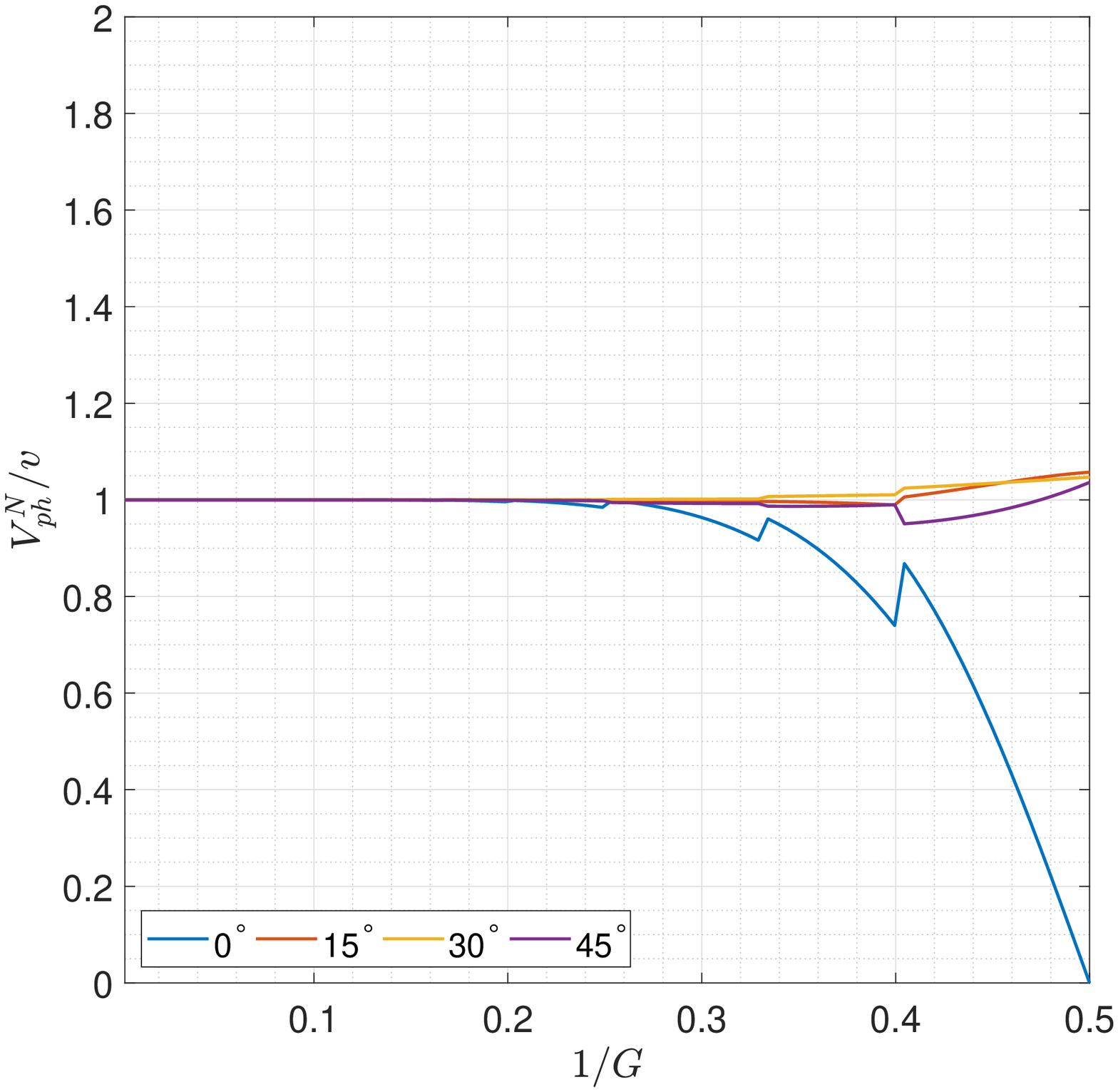}\label{S2.fig.02f}}

\subfloat[refined PW 17p ($\gamma=0.75$)]{\includegraphics[width=0.35\textwidth]{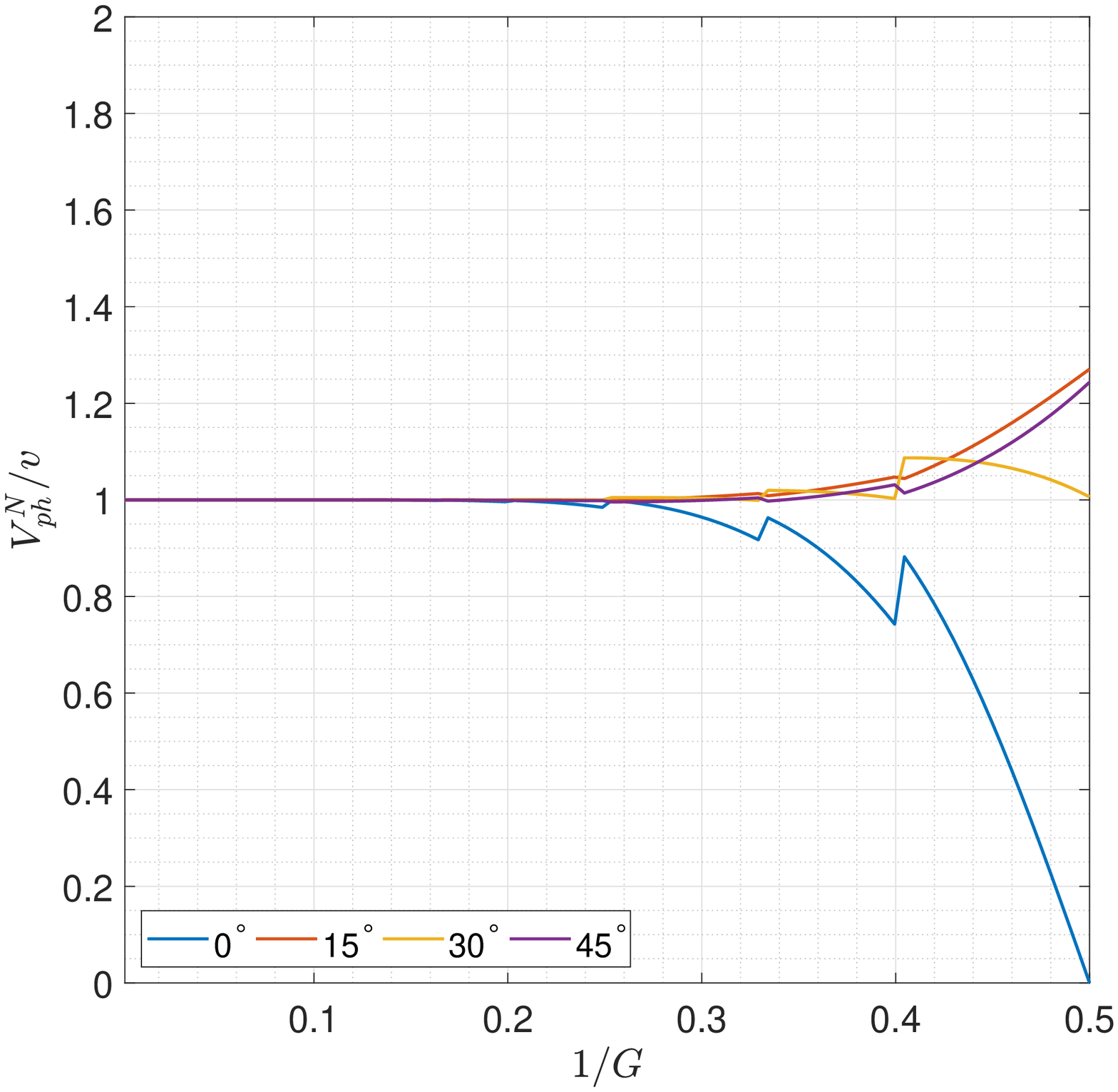}\label{S2.fig.02g}}
~
\subfloat[refined PW 17p ($\gamma=1.00$)]{\includegraphics[width=0.35\textwidth]{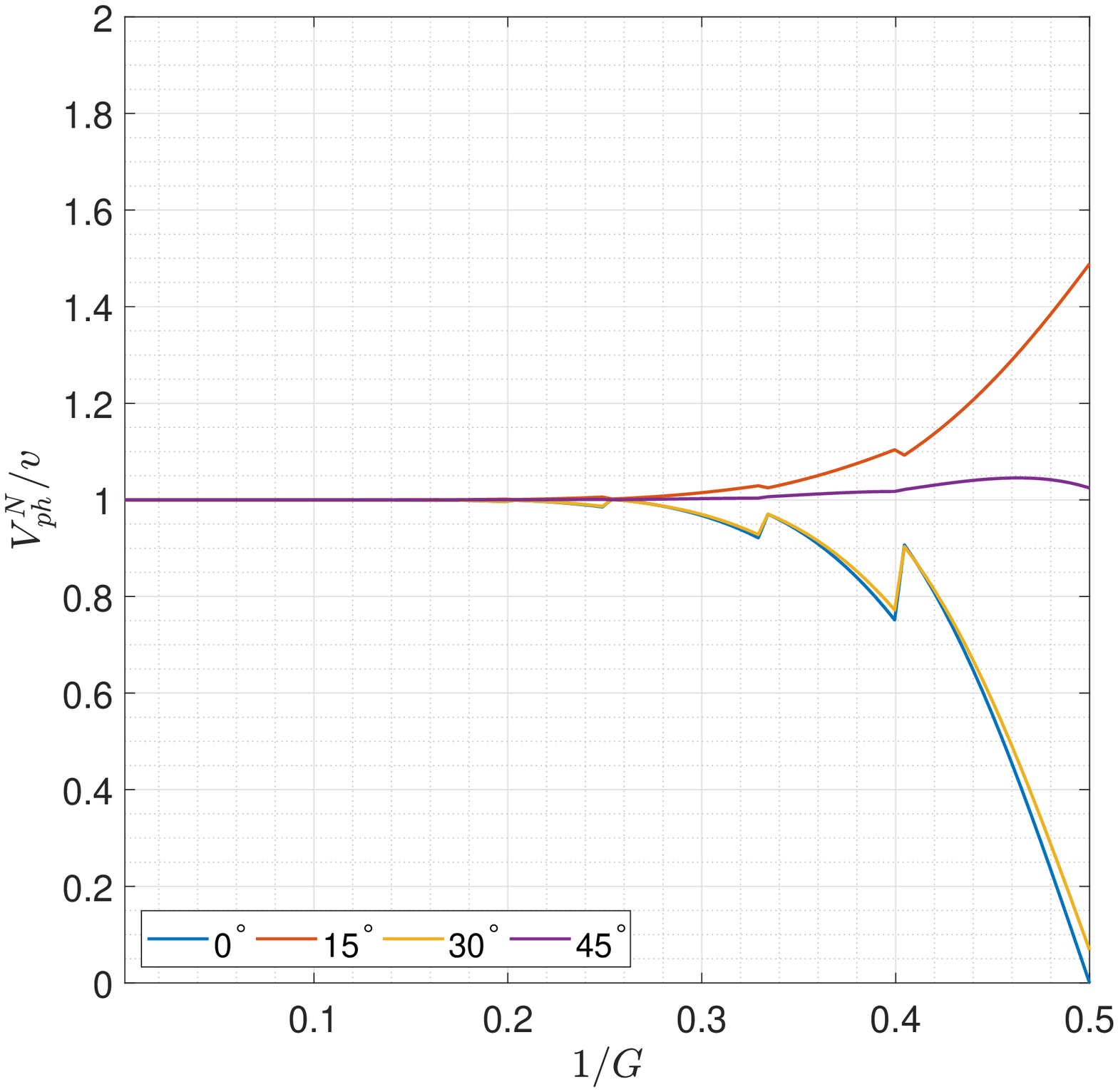}\label{S2.fig.02h}}
\caption{Normalized group velocity curves for various schemes}\label{S2.fig.02}
\end{figure}

\section{Numerical Examples}\label{S3}
In this section, we present three numerical examples. For the sake of simplicity, we present all results for the case that $\Delta x=\Delta z$. Example \ref{Ex.1} is meant for illustrating the accuracy and efficiency of the new schemes, refined PW 17p and refined PW 25p. We compare these schemes against some other existing optimal finite difference schemes that are widely used for solving the Helmholtz equation with PML.
In the followings examples, refined 25p, refined 17p and refined 9p represent the optimal 25-point finite difference scheme \cite{dastour2019}, the optimal 17-point finite difference scheme \cite{dastour2019} and the optimal 9-point finite difference scheme \cite{chen2013optimal}. The parameters of these refined finite difference schemes are estimated using the refined choice strategy (rule 3.8 from \cite{chen2013optimal}).
Moreover, global 25p and global 17p represent the optimal 25-point finite difference scheme \cite{dastour2019} and the optimal 17-point finite difference scheme \cite{dastour2019}, respectively, whose parameters are estimated using the global choice strategy (rules 3.4 from \cite{chen2013optimal} with $I_G=[4,400]$. In addition, optimal rotated 9p represents the rotated 9-point finite difference method \cite{jo1996optimal,chen2013optimal} with parameters $a = 0.5461$, $d = 0.3752$ and $e = -4\times 10^{-5}$. This group of optimal parameters was provided by Jo, Shin and Suh \cite{jo1996optimal,chen2013optimal} as optimal parameters for the rotated 9p finite difference method. Moreover, let point-weighting 9p represent the point-weighting scheme \cite{cheng2017dispersion} whose parameters are estimated using the flexible strategy for selection of weights (rule 3.3 at page 2354 of \cite{cheng2017dispersion}).

Furthermore, the error between the numerical solution and the exact solution is measured in C-norm \cite{chen2013optimal}, which is defined for any $M\times N$ complex matrix $\mathbf{Z}$ as follows
\begin{align}\label{S3.Norm}
\|\mathbf{Z}\|_C =\max_{1\leq i \leq M,~1\leq j \leq N} |z_{i,j}|.
\end{align}
where $|z_{i,j}|$ is the complex modulus of $z_{i,j}$.

In Example \ref{Ex.2}, we analyze the numerical solutions of the new schemes and compare them with the exact solution in a homogenous model. Finally, in Example \ref{Ex.3}, a more realistic problem is solved by refined PW 17p and refined 17p where we demonstrate the importance of consistently of an optimal scheme with the Helmholtz equation with PML.

\subsection{Example 1}\label{Ex.1}
Consider
\begin{align}\label{Ex.1.1}
\Delta  p + k^2 p = g(x,z),\quad \text{in~}\Omega : = (0,1) \times (0,1),
\end{align}
with
\begin{align}
\label{Ex.1.2}
(k_x,~k_z)&= k_{0}\left(e^{-k_0\left(x+z\right)}+1\right) \left( \cos (\theta),~\sin (\theta) \right),
\\
\label{Ex.1.3}
g(x,z)&=
e^{ik_{0}(x\cos (\theta)+z\sin (\theta))}
\left[
\sin\left(\pi x\right)\sin\left(\pi z\right)
\left({k_{0}}^2{\mathrm{e}}^{-2k_{0}\left(x+z\right)}\left(2{\mathrm{e}}^{k_{0}\left(x+z\right)}+1\right)-2\pi ^2\right)
\right. \notag \\ & \left.
-2\pi i k_{0}\left(\cos\left(\pi x\right)\sin\left(\pi z\right)\cos\left(\theta \right)+\cos\left(\pi z\right)\sin\left(\pi x\right)\sin\left(\theta \right)\right)
\right].
\end{align}
Dirichlet boundary conditions are imposed on the boundary, and its analytical solution is expressed as follows,
\begin{align}\label{Ex.1.4}
p(x,z) =
\sin\left(\pi x\right)\sin\left(\pi z\right)e^{ik_{0}(x\cos (\theta)+z\sin (\theta))}.
\end{align}

In this example, we compare the new schemes, refined PW 25p and refined PW 17p, with a number of popular optimal finite difference schemes that are used for the Helmholtz equation with PML. The new schemes are compared against refined 25p, global 25p, refined 17p, global 17p, refined 9p, optimal rotated 9p, point-weighting 9p, non-compact fourth-order (NC 4th-order) and conventional 5p. Moreover, the interval $I_G=\left[G_{\min},G_{\max}\right]= \left[\frac{2\pi}{hk_{\max}},\frac{2\pi}{hk_{\min}}\right]$ is estimated by using a priori information.

All the experiments in this example are performed with MATLAB 9.6.0.1135713 (R2019a) Update 3 on a Dell laptop equipped with Windows 10 Education Edition (64-bit), Intel(R) Core(TM) i5-4210U CPU, and 8.00 GB Physical Memory (RAM). We used an unsymmetric-pattern multifrontal (UMF) method for sparse LU factorization \cite{davis1997unsymmetric,davis2004algorithm} for solving linear systems generated by each finite difference method.

Table \ref{Ex1.table1} and \ref{Ex1.table2} demonstrate the error in the C-norm, defined in equation \eqref{S3.Norm}, for various schemes for different gridpoints $N$ per line and $\theta=\pi/4$ when $k_0 = 75$\linebreak and $k_0 = 150$, respectively. What stands out from these tables is that the new schemes, refined PW 25p and refined PW 17p, are indeed fourth order since the C-norm gets roughly $2^4$ smaller by decreasing the step-size by half each time. Moreover, we can see that the accuracy of the new schemes are comparable with refined 25p, global 25p, refined 17p and global 17p; however, refined PW 25p has shown the best level of accuracy in our experiment.  In addition, refined PW 17p and refined 17p have shown quite similar accuracy level for this example.
\begin{table}[htp]
\centering
\caption{The error in the C-norm for $k_0 = 75$.}
\label{Ex1.table1}
\begin{tabular}{cccc}
\toprule
 $N$& 131& 261& 521\\
\toprule
refined PW 17p& 7.6295e-04& 4.2110e-05& 2.5961e-06\\
refined PW 25p& 6.6847e-04& 2.6623e-05& 1.4675e-06\\
refined 25p& 7.3473e-04& 3.8492e-05& 2.3576e-06\\
global 25p& 7.4767e-03& 2.5834e-04& 1.9776e-05\\
refined 17p& 7.6295e-04& 4.2110e-05& 2.5961e-06\\
global 17p& 8.0968e-03& 2.8294e-04& 2.1769e-05\\
NC 4th-order& 3.8304e-02& 1.1364e-03& 7.8459e-05\\
refined 9p& 3.4344e-02& 8.1079e-03& 1.9689e-03\\
conventional 5p& 2.9867e+01& 3.2683e-01& 7.0565e-02\\
optimal rotated 9p& 1.2382e-01& 3.0495e-02& 7.6214e-03\\
point-weighting 9p& 3.4632e-02& 8.1434e-03& 1.9721e-03\\
\bottomrule
\end{tabular}
\end{table}

\begin{table}[htp]
\centering
\caption{The error in the C-norm for $k_0 = 150$.}
\label{Ex1.table2}
\begin{tabular}{cccc}
\toprule
 $N$& 241& 481& 961\\
\toprule
refined PW 17p& 1.1087e-03& 4.9325e-05& 3.9214e-06\\
refined PW 25p& 1.2022e-03& 3.1931e-05& 2.0554e-06\\
refined 25p& 1.2070e-03& 4.5132e-05& 3.5631e-06\\
global 25p& 2.9217e-02& 1.2571e-03& 4.8789e-05\\
refined 17p& 1.1087e-03& 4.9325e-05& 3.9214e-06\\
global 17p& 3.1802e-02& 1.3738e-03& 5.3365e-05\\
NC 4th-order& 1.7002e-01& 5.3904e-03& 1.9552e-04\\
refined 9p& 3.6788e-02& 8.6719e-03& 2.0928e-03\\
conventional 5p& 7.6177e+00& 5.2672e-01& 1.2031e-01\\
optimal rotated 9p& 2.4496e-01& 4.9931e-02& 1.2727e-02\\
point-weighting 9p& 3.7151e-02& 8.7202e-03& 2.0971e-03\\
\bottomrule
\end{tabular}
\end{table}

Furthermore, the accuracy of numerical solutions using NC 4th-order and conventional 5p again proves why one needs to consider using optimal finite difference schemes for the Helmholtz equation. As for optimal 9-point finite difference schemes, although these schemes are great for solving Helmholtz equation, optimal fourth-order finite difference schemes overall provide better efficiency since to maintain the same level of accuracy, fourth-order finite difference methods require a smaller number of gridpoints than optimal 9-point finite difference schemes \cite{dastour2019}. As can be seen from Tables \ref{Ex1.table1} and \ref{Ex1.table2} an almost the same accuracy can be obtained using a quarter number of gridpoints, the overall computational costs are comparable.

In addition, Tables \ref{Ex1.table3}, and \ref{Ex1.table4} show the error in the C-norm with $k_0=100$ and $\theta=0$, $\pi /16$,\ldots,$\pi/4$, for $N$ = $101$ and $N=201$, respectively. It can be seen that the new schemes still demonstrate fourth-order accuracy for various values of $\theta$ since the C-norm gets almost $2^4$ smaller by decreasing the step-size by half. Nonetheless, refined PW 25p has shown more consistent level of accuracy than the rest of schemes under study.

\begin{table}[htp]
 \centering
  \caption{The error in the C-norm for $k_0 = 100$ and $N = 101$.}
  \label{Ex1.table3}
 \begin{tabular}{cccccc}
 \toprule
$\theta$    & $0$ & $\pi/16$ & $\pi/8$ & $3\pi/16$ & $\pi/4$ \\
\toprule
refined PW 17p  & 1.0777e-02 &  1.3799e-02 & 8.3763e-03 & 4.5090e-03 & 1.2548e-01 \\
refined PW 25p  & 1.0501e-02 &  1.4570e-02 & 9.4068e-03 & 7.5605e-03 & 2.3524e-02 \\
refined 25p  & 1.0537e-02 &  1.4637e-02 & 9.1759e-03 & 7.2469e-03 & 3.3765e-02 \\
global 25p  & 9.7047e-02 &  3.6421e-02 & 1.9046e-02 & 1.5790e-02 & 3.3981e-01 \\
refined 17p  & 1.0777e-02 &  1.3799e-02 & 8.3763e-03 & 4.5090e-03 & 1.2548e-01 \\
global 17p  & 1.0099e-01 &  3.6463e-02 & 1.8545e-02 & 1.8253e-02 & 3.2219e-01 \\
NC 4th-order  & 3.9183e-01 &  7.2159e-01 & 5.1360e-01 & 2.4293e-01 & 5.7859e-01 \\
refined 9p  & 9.4147e-02 &  1.2328e-01 & 9.7295e-02 & 1.1429e-01 & 6.4245e-01 \\
conventional 5p  & 2.8297e+00 &  3.9821e+00 & 3.6443e+00 & 4.4336e+00 & 4.7288e+01 \\
optimal rotated 9p  & 2.4409e-01 &  1.9197e+00 & 1.8917e-01 & 4.5804e-01 & 4.8616e+00 \\
point-weighting 9p  & 9.4140e-02 &  1.2232e-01 & 9.7428e-02 & 1.1577e-01 & 6.5001e-01 \\
\bottomrule
\end{tabular}
\end{table}

\begin{table}[htp]
 \centering
  \caption{The error in the C-norm for $k_0 = 100$ and $N = 201$.}
  \label{Ex1.table4}
 \begin{tabular}{cccccc}
 \toprule
$\theta$    & $0$ & $\pi/16$ & $\pi/8$ & $3\pi/16$ & $\pi/4$ \\
\toprule
refined PW 17p  & 7.2884e-04 &  8.4785e-04 & 5.4066e-04 & 4.7556e-04 & 3.1612e-04 \\
refined PW 25p  & 7.3655e-04 &  8.3922e-04 & 5.0431e-04 & 4.1707e-04 & 2.6860e-04 \\
refined 25p  & 7.3531e-04 &  8.4408e-04 & 5.2335e-04 & 4.4990e-04 & 3.0142e-04 \\
global 25p  & 3.0820e-03 &  8.4830e-03 & 2.3530e-03 & 1.7268e-03 & 2.5818e-03 \\
refined 17p  & 7.2884e-04 &  8.4785e-04 & 5.4066e-04 & 4.7556e-04 & 3.1612e-04 \\
global 17p  & 3.2273e-03 &  8.8107e-03 & 2.3169e-03 & 1.9384e-03 & 2.8142e-03 \\
NC 4th-order  & 2.2008e-02 &  6.9477e-02 & 3.2024e-02 & 1.8711e-02 & 1.2017e-02 \\
refined 9p  & 2.1569e-02 &  2.4452e-02 & 2.2774e-02 & 2.7477e-02 & 1.8126e-01 \\
conventional 5p  & 1.4974e+00 &  1.2433e+00 & 1.3061e+00 & 2.6444e+00 & 1.4253e+00 \\
optimal rotated 9p  & 1.1361e-01 &  4.3817e-01 & 7.2447e-02 & 6.4544e-02 & 5.5283e-01 \\
point-weighting 9p  & 2.1567e-02 &  2.4293e-02 & 2.2840e-02 & 2.7603e-02 & 1.8795e-01 \\
\bottomrule
\end{tabular}
\end{table}

Generally speaking, all finite difference schemes require solving a linear system $AX=R$, where $A$ is the matrix of coefficients associated with each finite difference scheme, $X$ is a vector consists of the unknowns, and $R$ is the right-hand side. The matrix in the linear system associated with each finite difference scheme, matrix $A$, is a sparse matrix with complex values (see Figure \ref{Ex1_Fig1}). For a given $N$, the number of non-zero entries of the matrix $A$ (NNE) for  refined PW 25p and refined PW 17p are $25\,N^2 - 60\,N + 36$ and $17\,N^2-36\,N+20$, respectively. For example, when $N=10$, NNE for refined PW 25p  and refined PW 17p are 295936 and 201760, respectively. Therefore, refined PW 17p has less computational complexity than refined PW 25p and is more suitable for large practical applications. In \cite{dastour2019}, the authors performed a complete analysis regarding the CPU times of 25p schemes and 9p scheme (for details, readers are referred to reference \cite{dastour2019}).

\begin{figure}[htp]
\centering
\subfloat[refined PW 25p]{\includegraphics[width=0.45\textwidth]{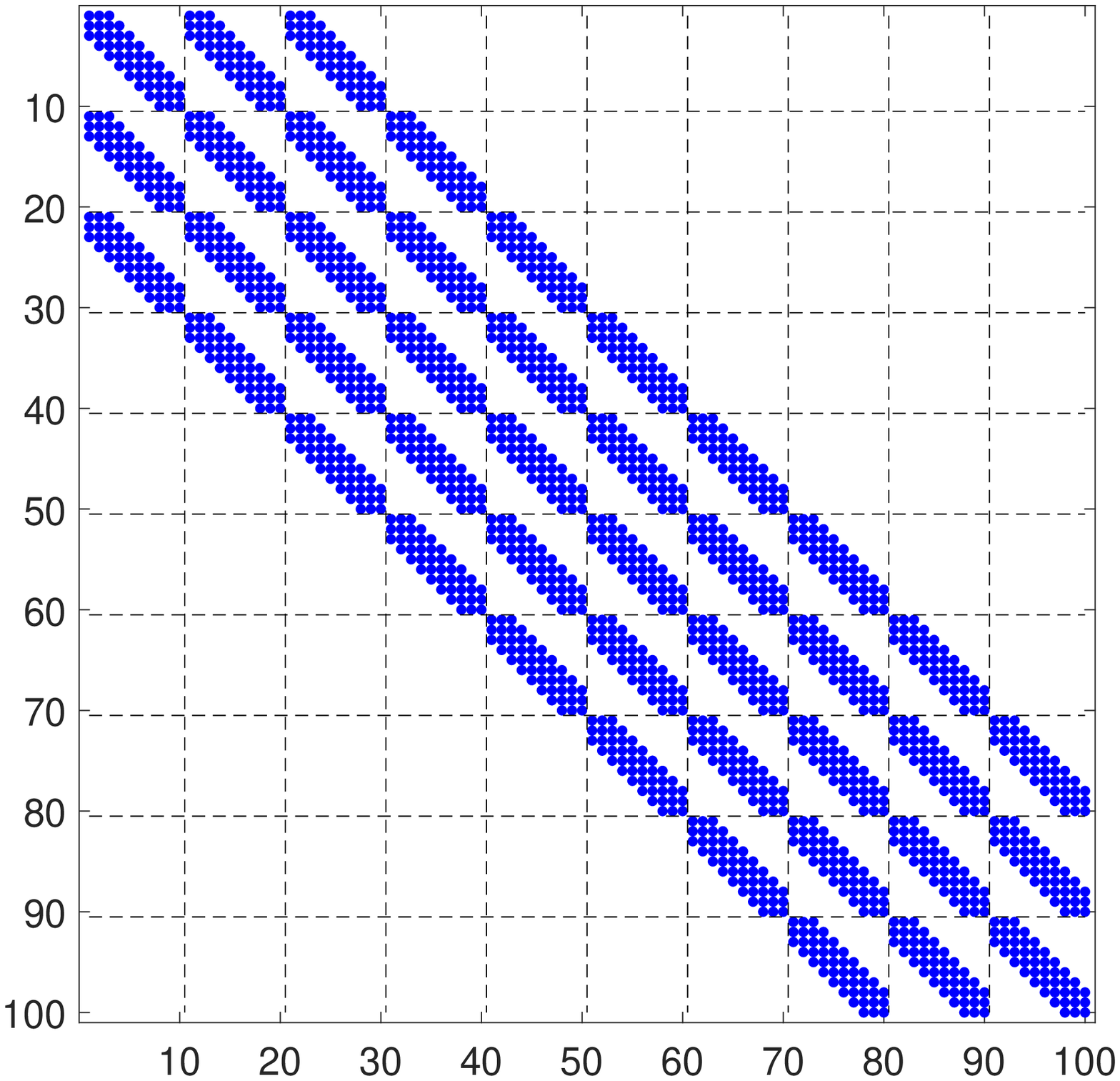}\label{Ex1_Fig1a}}
~
\subfloat[refined PW 17p]{\includegraphics[width=0.45\textwidth]{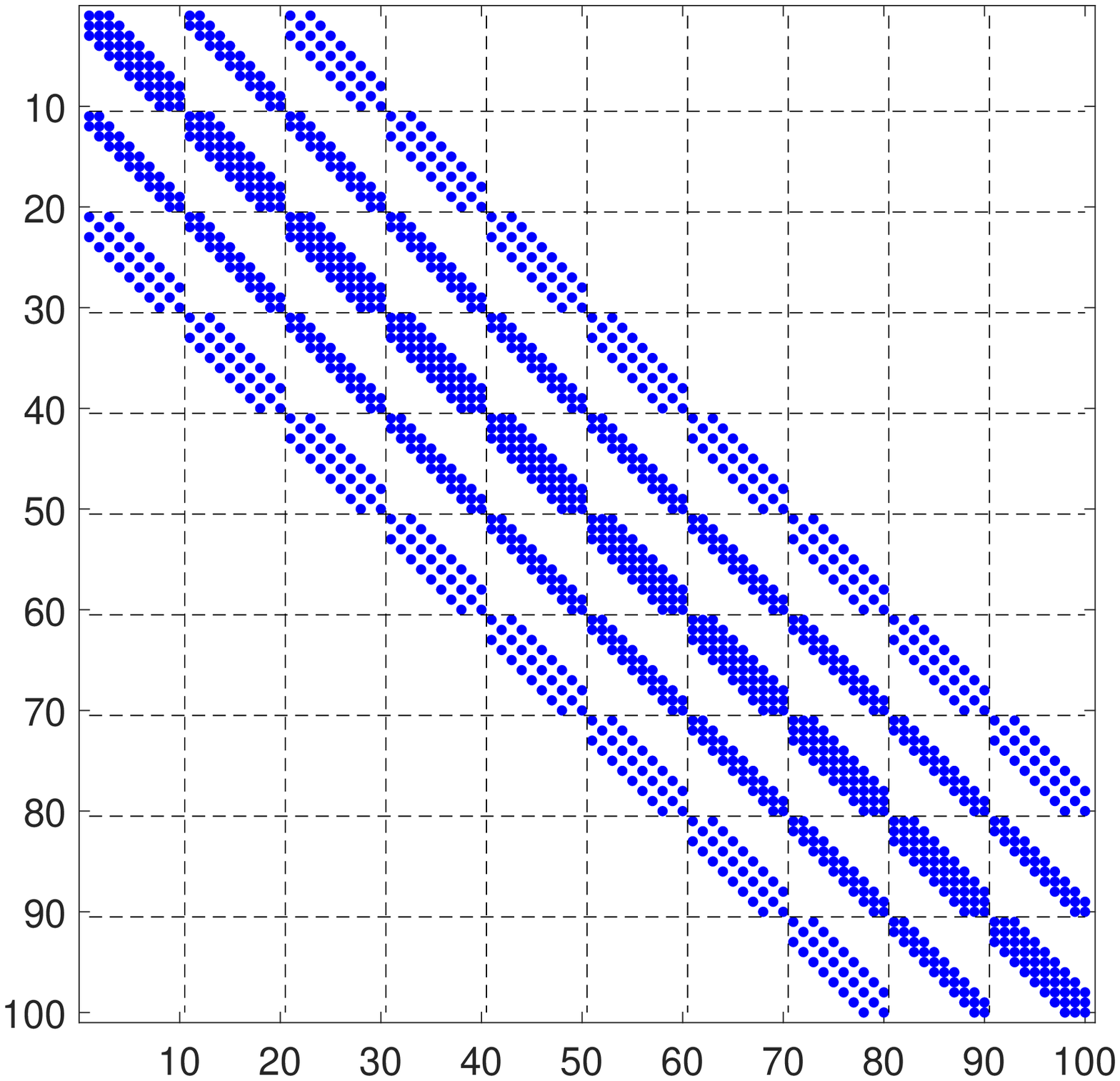}\label{Ex1_Fig1b}}
\caption{The matrix in the linear systems with $N=10$.}\label{Ex1_Fig1}
\end{figure}

The results of this example can be summarized as follows. The two new schemes, refined PW 25p, and refined PW 17p have shown a comparable level of accuracy with that of available from refined 25p and refined 17p; however, refined PW 25p has shown a more consistent level of accuracy than the rest of schemes under study. Although refined PW 17p and refined 17p have shown a quite similar accuracy level, the main differences between these two schemes are refined 17p is inconsistent in the presence of PML, and refined 17p cannot be extended to the case that $\Delta x\neq \Delta z$. Furthermore, refined PW 17p can provide a more efficient scheme than refined PW 25p in terms of computational complexity.

\subsection{Problem 2: a homogeneous model} \label{Ex.2}
In this example, a homogeneous velocity model with a velocity of is 2000 m/s is considered. The domain is $[0,~2km] \times [0km,~2km]$ with sampling intervals $h=\Delta x=\Delta z=20$  m. The time sampling here is $\Delta t=8$  ms. Moreover, a point source $\delta(x - x_s ,z -z_s)R(\omega,f_{M} )$ is located at the point $(700m,~500m)$, where $R(\omega,f_{M} )$ is the Ricker wavelet, defined in equation \eqref{Ex2.ricker}, with the peak frequency $f_{M}=15$ Hz.
\begin{equation}\label{Ex2.ricker}
R(t,f_{M} ) = \left(1 - 2\pi^2f_{M}^2 t^2 \right)/\exp\left(\pi^2f_{M}^2 t^2\right).
\end{equation}
From  \cite{alford1974accuracy}, the analytical solution of this homogeneous model is available as follows
\begin{equation}\label{Ex2.analytical}
p (x,z,t)=i\pi \mathcal{F}^{-1} \left(H^{(2)}_{0}\left(\frac{\omega}{v}\sqrt{(x-x_s)^2+(z-z_s)^2}\right)
\mathcal{F}\left(R(t,f_{M})\right)\right)
\end{equation}
where $\mathcal{F}$ and $\mathcal{F}^{-1}$ are respectively Fourier and inverse Fourier transformations with respect to time, and $H^{(2)}_{0}$ is the second Hankel function of order zero.

Since the numerical dispersion is dependent on the propagation angle, we have placed eight receivers. See Table \ref{Ex2.table0} and Figure \ref{Ex2.fig0} for coordinates of these receivers.

\begin{table}[htp]
\centering
\caption{The coordinate of various receivers}
\label{Ex2.table0}
\begin{tabular}{ccccccccc}
Receiver & 1   & 2   & 3   & 4   & 5   & 6   & 7   & 8   \\
\toprule
$x_{r}$     & 100 & 300 & 700 & 100 & 900 & 700 & 300 & 500 \\
$z{_r}$     & 500 & 300 & 700 & 700 & 500 & 300 & 900 & 900 \\
\bottomrule
\end{tabular}
\end{table}

\begin{figure}[htp]
\centering
\includegraphics[width=0.6\textwidth]{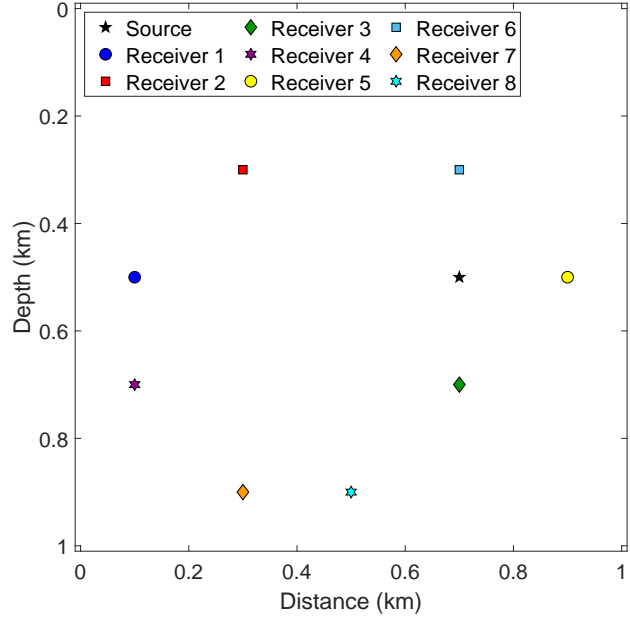}
\caption{Source and Receiver locations}\label{Ex2.fig0}
\end{figure}

In this example, we compare refined PW 17p, refined PW 25p, refined 25p, refined 17p, and NC 4th-order in terms of accuracy using the C-norm \eqref{S3.Norm}. Table \ref{Ex2.table1} highlights the error in the C-norm for receivers 1, 2, \ldots,8. What can be seen from Table \ref{Ex2.table1} is that refined PW 17p, refined PW 25p, refined 25p, and refined 17p have provided similar accuracy; however, refined PW 17p and refined 17p have shown a slightly better level of the accuracy than the other two schemes. One possible reason for this can be the size of the problem, $nx=nz=51$. Moreover, we included NC 4th-order in this example to emphasize the importance of using optimal finite difference methods for Helmholtz equation with PML since it can be seen that the NC 4th-order is the least efficient scheme amongst all of the schemes under study in this example.
\begin{table}[htp]
\centering
\caption{The error in the C-norm for various schemes.}
\label{Ex2.table1}
\begin{tabular}{ccccc}
\toprule
$(x_r,~z_r)$ &(100, 500) &(300, 300) &(700, 700) &(100, 700)\\
\toprule
refined PW 17p & 1.9495e-02 & 9.3023e-03 & 2.2237e-02 & 1.2545e-02\\
refined PW 25p & 2.0961e-02 & 1.3435e-02 & 1.9469e-02 & 1.5896e-02\\
refined 25p & 1.9892e-02 & 1.3034e-02 & 1.9484e-02 & 1.6327e-02\\
refined 17p & 1.9231e-02 & 9.2848e-03 & 2.2523e-02 & 1.2546e-02\\
NC 4th-order & 3.4299e-01 & 2.2120e-01 & 1.6034e-01 & 2.9828e-01\\
\toprule
$(x_r,~z_r)$ &(900, 500) &(700, 300) &(300, 900) &(500, 900)\\
\toprule
refined PW 17p & 2.2019e-02 & 2.2242e-02 & 6.0025e-03 & 9.0136e-03\\
refined PW 25p & 1.9500e-02 & 1.9468e-02 & 1.1749e-02 & 1.3388e-02\\
refined 25p & 1.9475e-02 & 1.9484e-02 & 1.0649e-02 & 1.2995e-02\\
refined 17p & 2.2434e-02 & 2.2526e-02 & 6.3076e-03 & 9.2414e-03\\
NC 4th-order & 1.6119e-01 & 1.6039e-01 & 1.4991e-01 & 2.2045e-01\\
\bottomrule
\end{tabular}
\end{table}

\begin{figure}[htp]
\centering
\subfloat[Exact and numerical solutions for Receiver 1]{\includegraphics[width=0.49\textwidth]{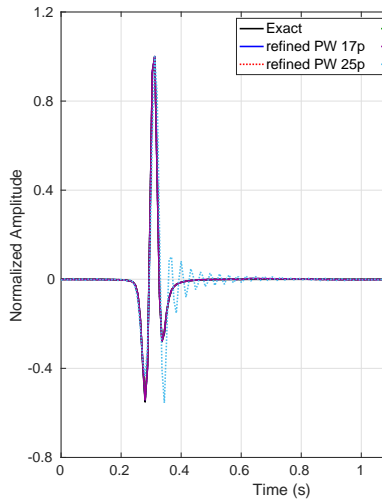}\label{Ex2.fig1_1}}
~
\subfloat[Exact and numerical solutions for Receiver 2]{\includegraphics[width=0.49\textwidth]{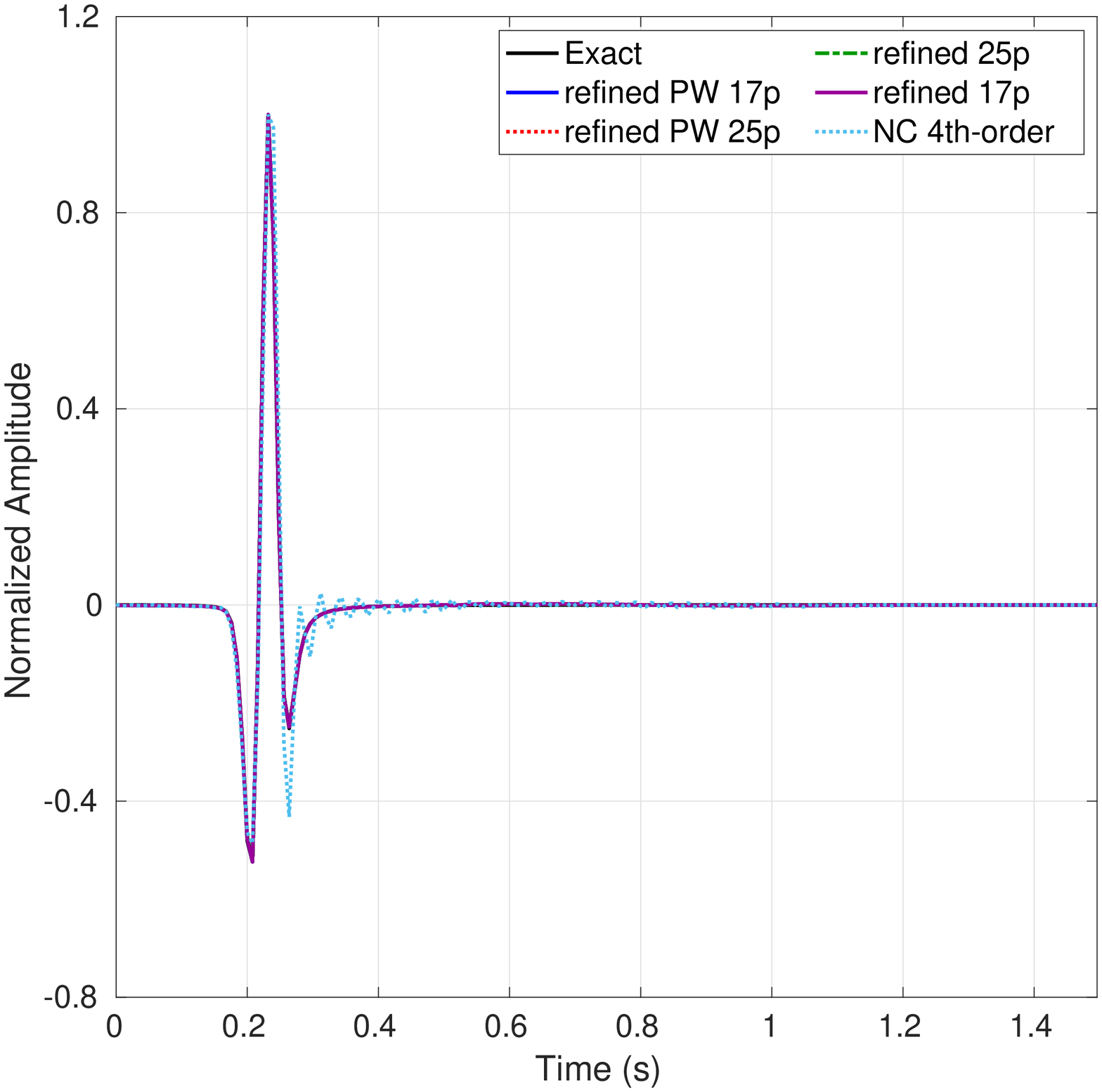}\label{Ex2.fig1_2}}

\subfloat[Exact and numerical solutions for Receiver 3]{\includegraphics[width=0.49\textwidth]{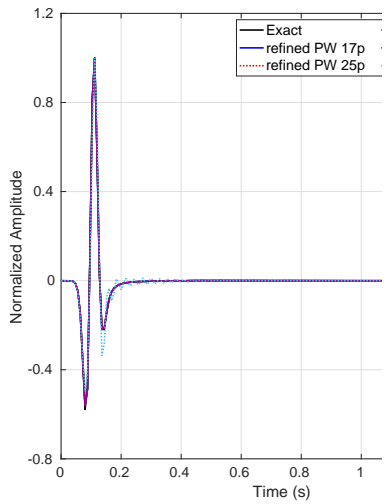}\label{Ex2.fig1_3}}
~
\subfloat[Exact and numerical solutions for Receiver 4]{\includegraphics[width=0.49\textwidth]{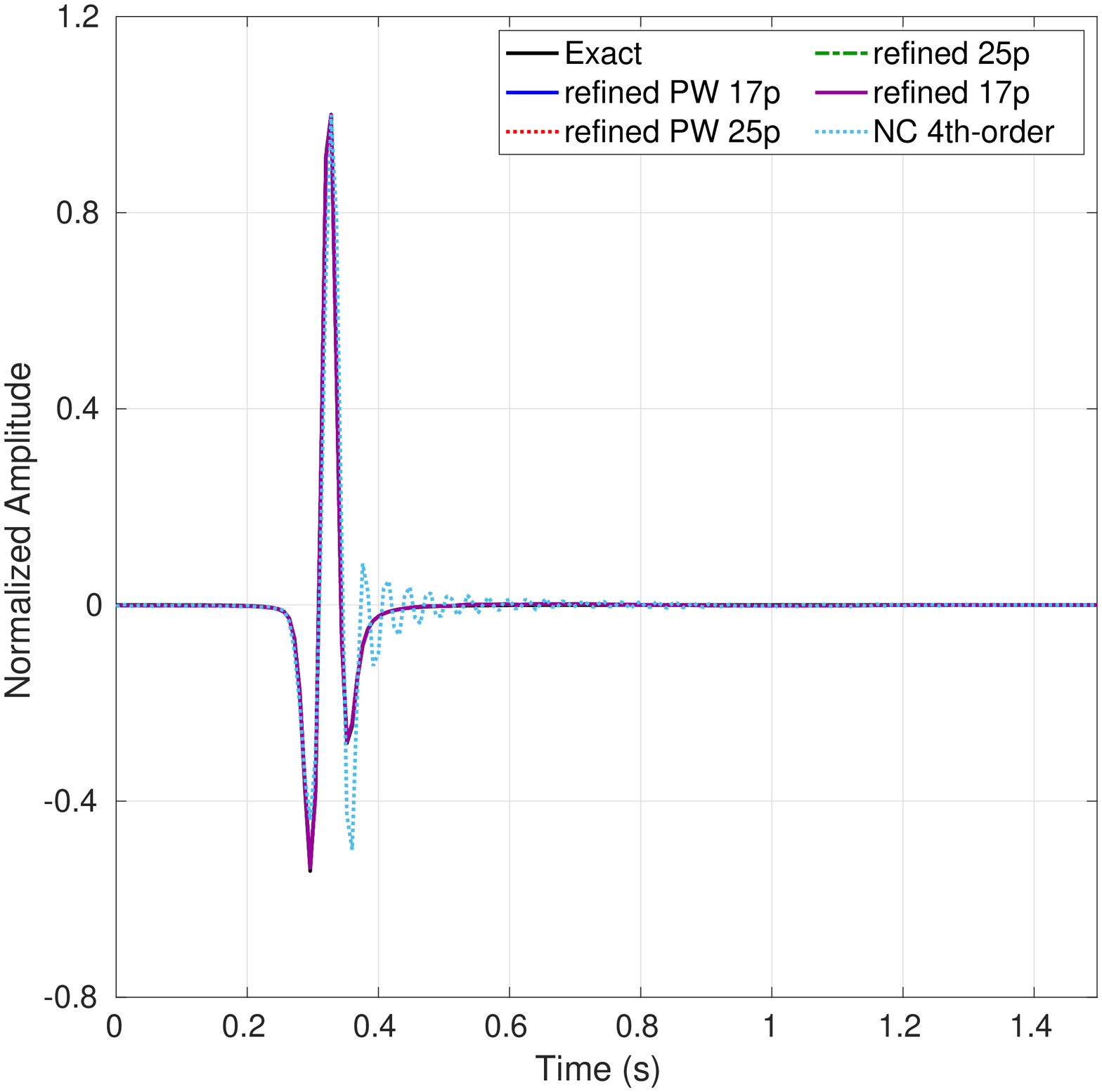}\label{Ex2.fig1_4}}
\caption{Exact and numerical solutions for Receivers 1,2,3 and 4}\label{Ex2.fig1}
\end{figure}

\begin{figure}[htp]
\centering
\subfloat[Exact and numerical solutions for Receiver 5]{\includegraphics[width=0.49\textwidth]{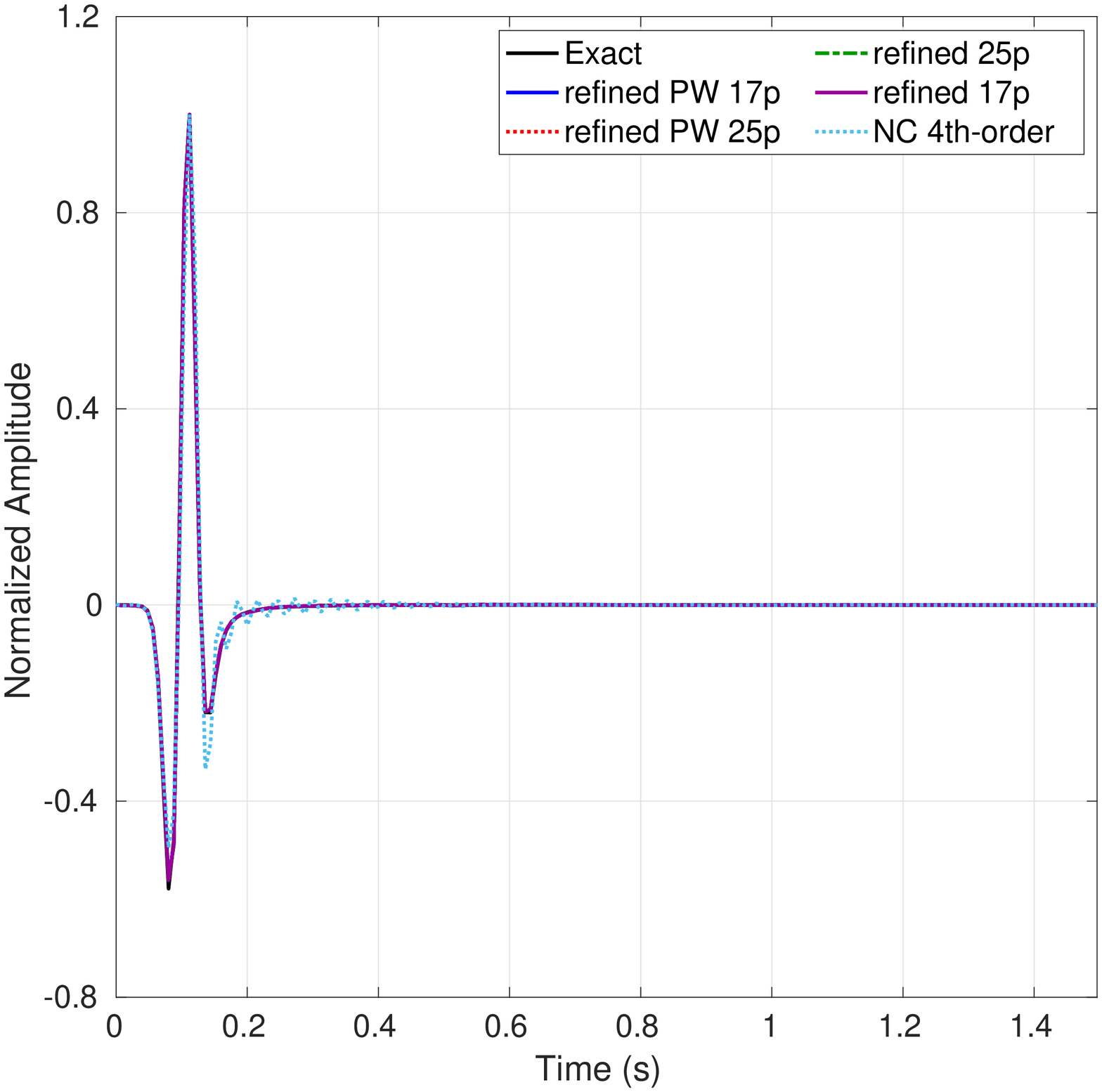}\label{Ex2.fig1_5}}
~
\subfloat[Exact and numerical solutions for Receiver 6]{\includegraphics[width=0.49\textwidth]{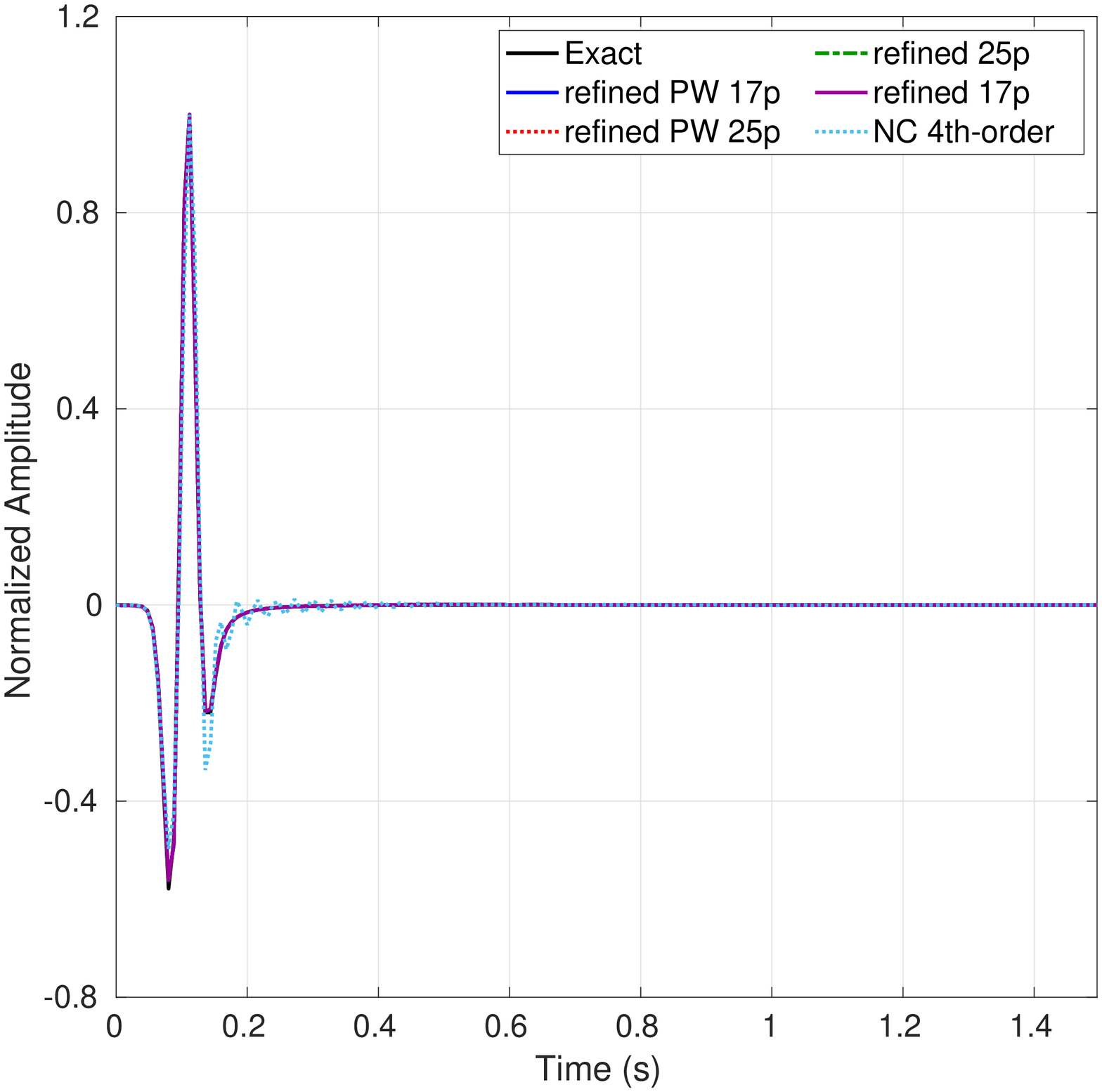}\label{Ex2.fig1_6}}

\subfloat[Exact and numerical solutions for Receiver 7]{\includegraphics[width=0.49\textwidth]{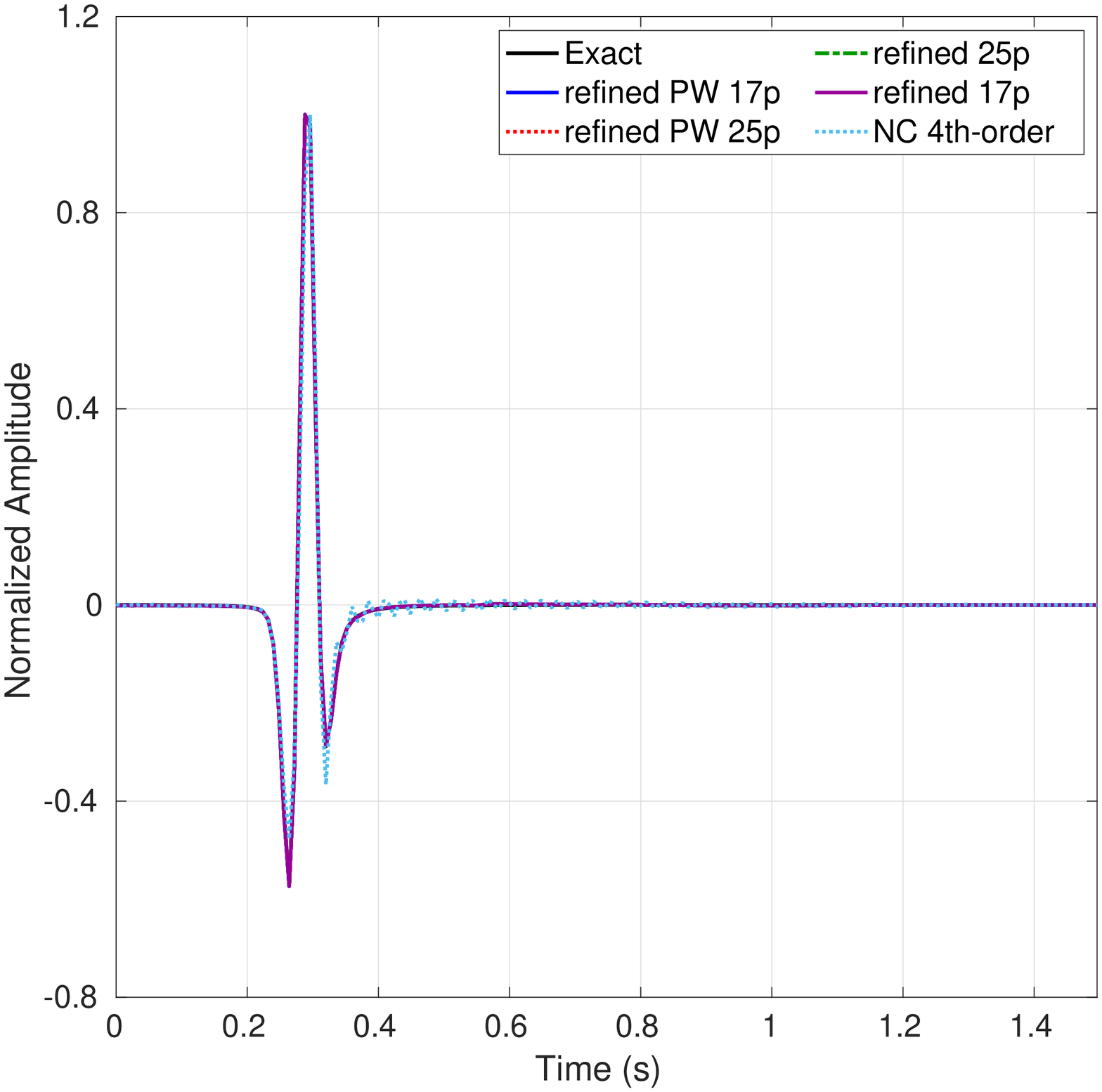}\label{Ex2.fig1_7}}
~
\subfloat[Exact and numerical solutions for Receiver 8]{\includegraphics[width=0.49\textwidth]{Ex2_Fig1_7.eps}\label{Ex2.fig1_8}}
\caption{Exact and numerical solutions for Receivers 5,6, 7 and 8}\label{Ex2.fig2}
\end{figure}

From \cite{dastour2019}, we know that not only refined 17p is inconsistent with Helmholtz equation with PML, but also cannot it approximate the Helmholtz equation when $\Delta x\neq \Delta z$. In this example, we could not observe any signs of inconsistency from refined 17p. Therefore, we demonstrate the importance of consistency in Example \ref{Ex.3}.

\subsection{Problem 3: a layered model} \label{Ex.3}
Consider a layered P-wave velocity model, shown in figure \ref{Ex3.fig0}. Here, horizontal and vertical samplings are $nx=nz=201$ with sampling intervals $\Delta x=\Delta z=10$  m, and the time sampling interval is $\Delta t=8$  ms. A point source $\delta(x - x_s ,z -z_s)R(\omega,f_{M} )$ is located at the point $(x_s,~z_s)=(1000m,~0m)$, where $R(\omega,f_{M} )$ is the Ricker wavelet with the peak frequency $f_{M}=20$ Hz.

\begin{figure}[htp]
\centering
\subfloat[A layered model]{\includegraphics[width=0.52\textwidth]{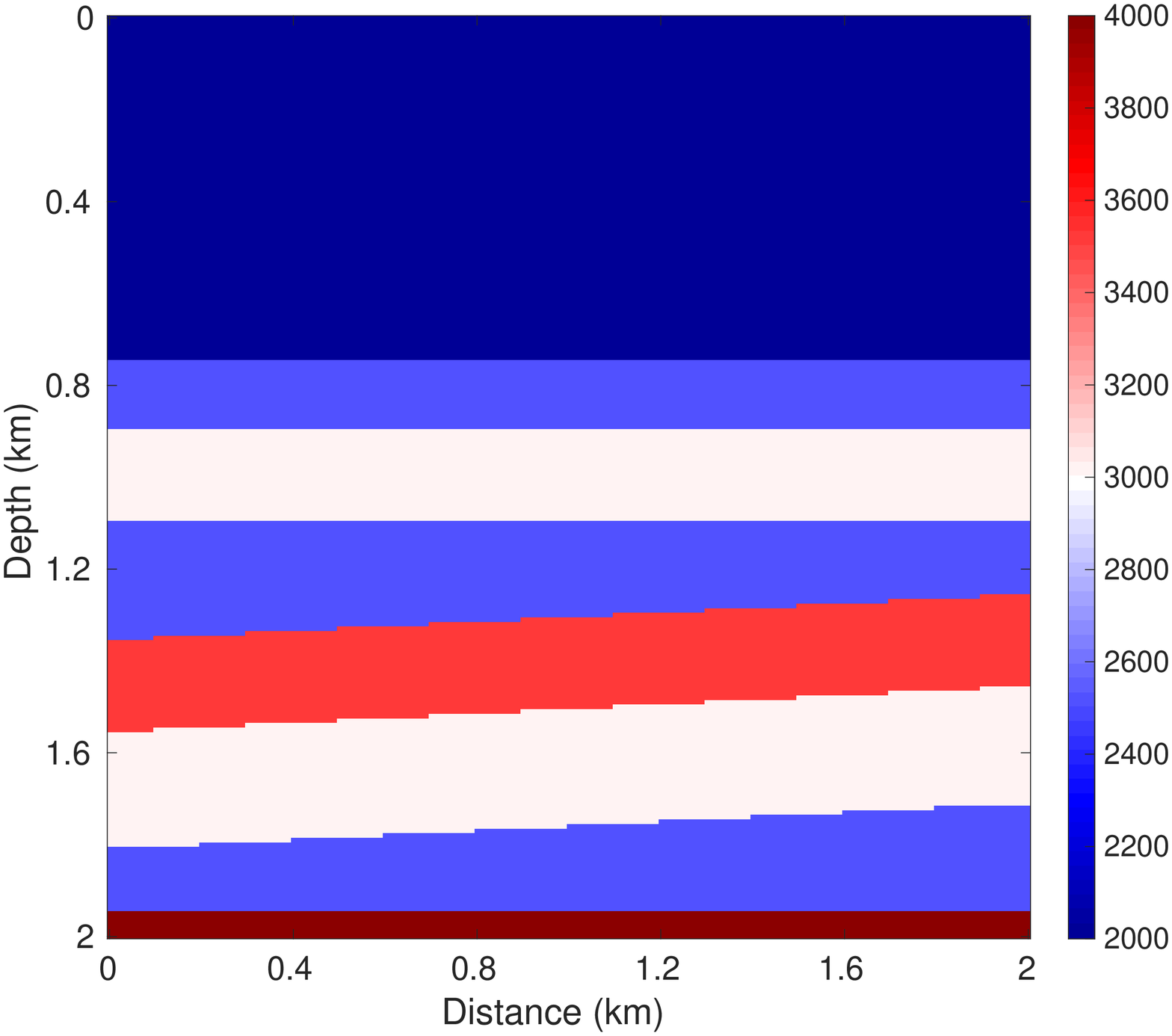}\label{Ex3.fig0}}
~
\subfloat[Source and Receiver locations]{\includegraphics[width=0.455\textwidth]{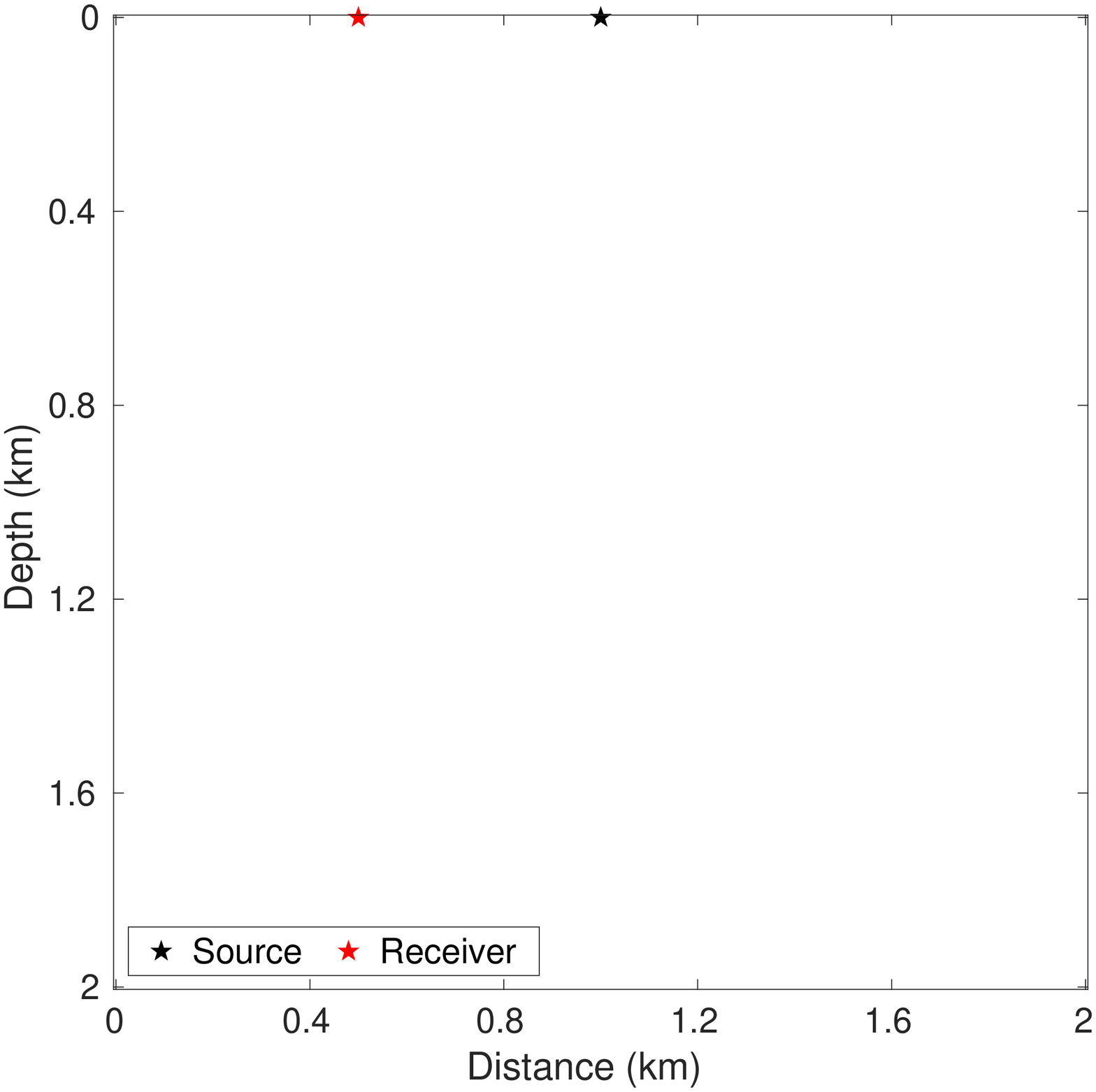}\label{Ex3.fig1}}
\end{figure}

Moreover, $I_G=\left[G_{\min},G_{\max}\right]$ is estimated by using a priori information before choosing parameters, that is $I_G= \frac{1}{f\,h}\left[v_{\min},v_{\max}\right]$. In addition, the damping profiles here are defined using $\sigma_x$ and $\sigma_z$ from \eqref{S1.eq.03} with $L_{PML}=500$ m and $a_0=1.790$.

Furthermore, the mono-frequency wavefields (real parts) for $f = 62.5$ Hz obtained by refined PW 17p and refined 17p  are available in Figure \ref{Ex3.fig2}.
We found the plot obtained by refined PW 17p more clear than that of obtained by refined 17p. Let's take a closer look at Figure \ref{Ex3.fig3b}. Especially, on the top-left and top-right of the image, we can observe some numerical particles.

\begin{figure}[htp]
\centering
\subfloat[refined PW 17p]{\includegraphics[width=0.45\textwidth]{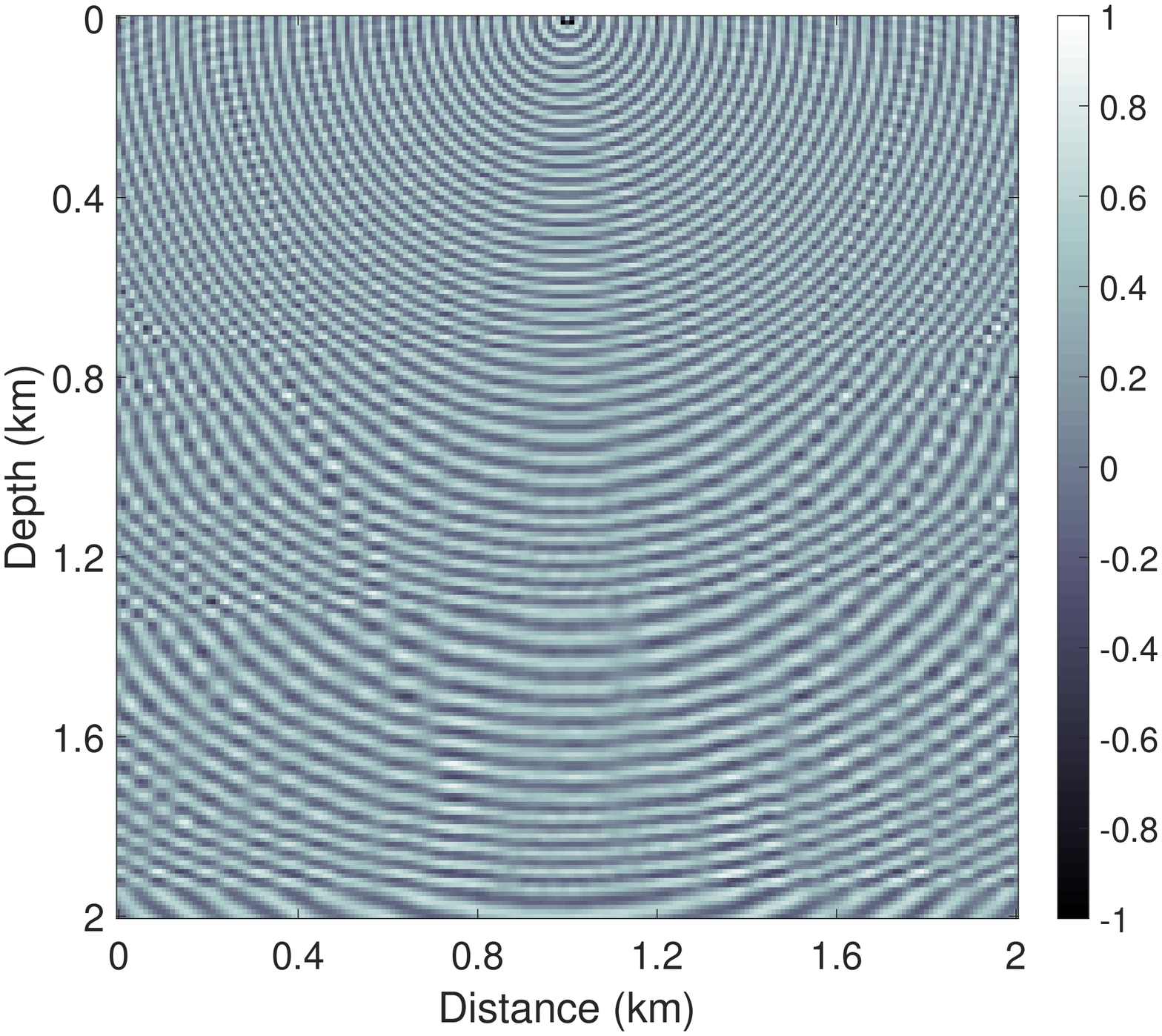}\label{Ex3.fig2a}}
~
\subfloat[refined 17p]{\includegraphics[width=0.45\textwidth]{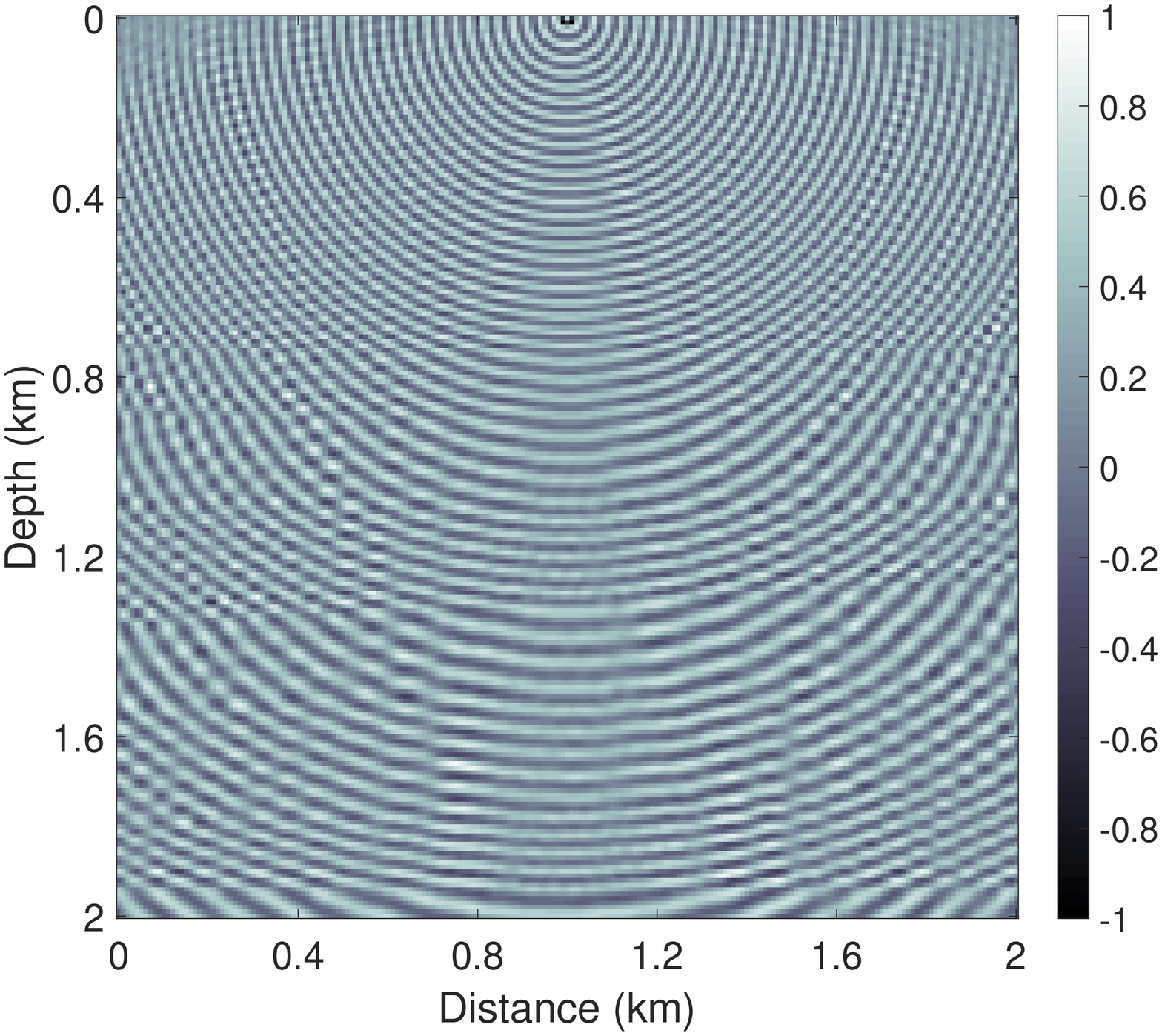}\label{Ex3.fig2b}}
\caption{The monofrequency wavefield (real part) for $f = 62.5$ Hz}\label{Ex3.fig2}
\end{figure}

In addition, snapshots for $t=520$ ms generated by refined PW 17p and refined 17p are available in Figure \ref{Ex3.fig3}. For the snapshot obtained by refined PW 17p (Figure \ref{Ex3.fig3a}), no boundary reflections can be observed, and the upward incident waves, the downward incident waves and transmissive waves are all clear. However, the same cannot be said for the snapshot obtained by refined 17p. As can be seen, we can find unclear parts on the top-left and top-right of the generated snapshot.

\begin{figure}[htp]
\centering
\subfloat[refined PW 17p]{\includegraphics[width=0.45\textwidth]{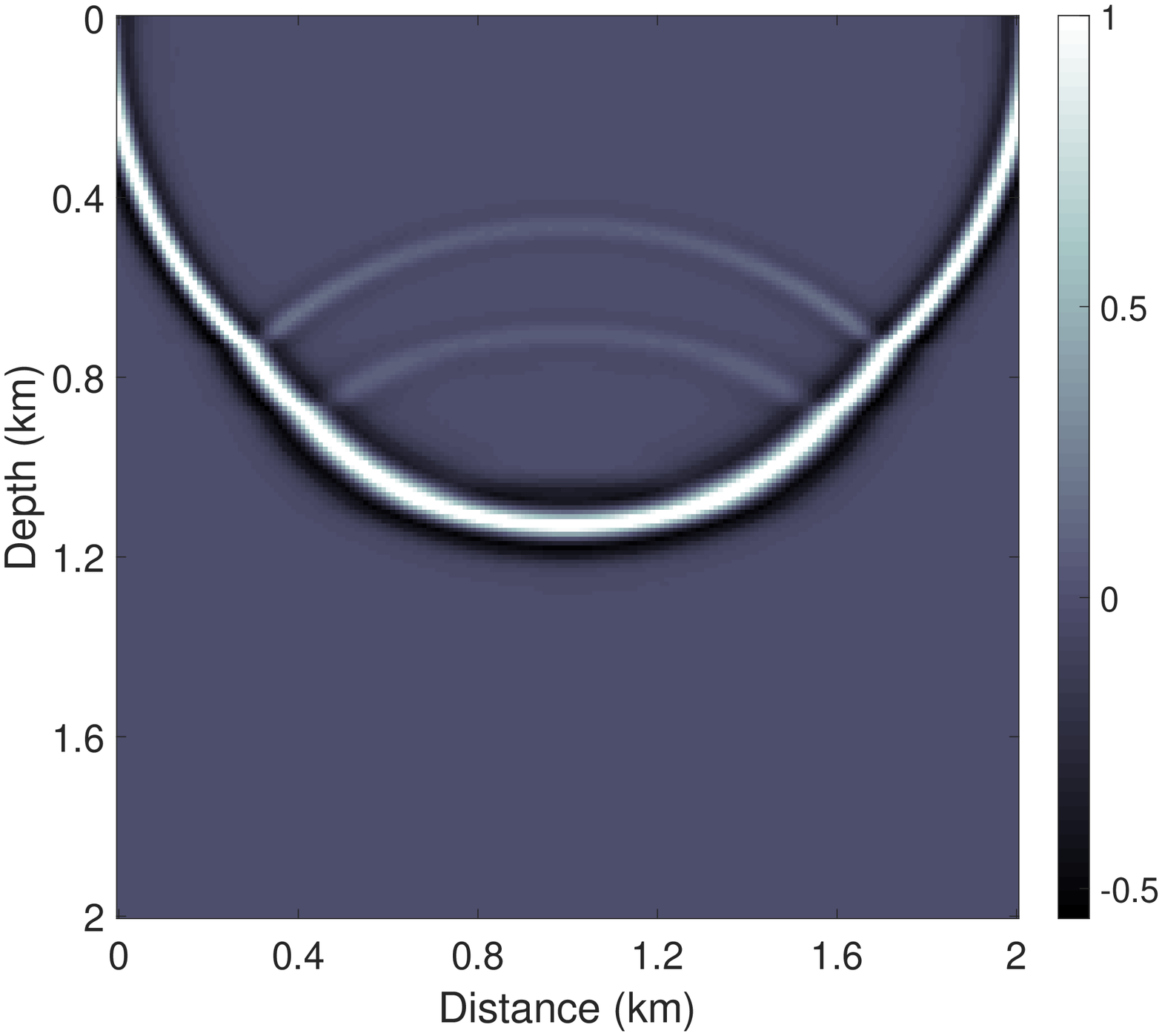}\label{Ex3.fig3a}}
~
\subfloat[refined 17p]{\includegraphics[width=0.45\textwidth]{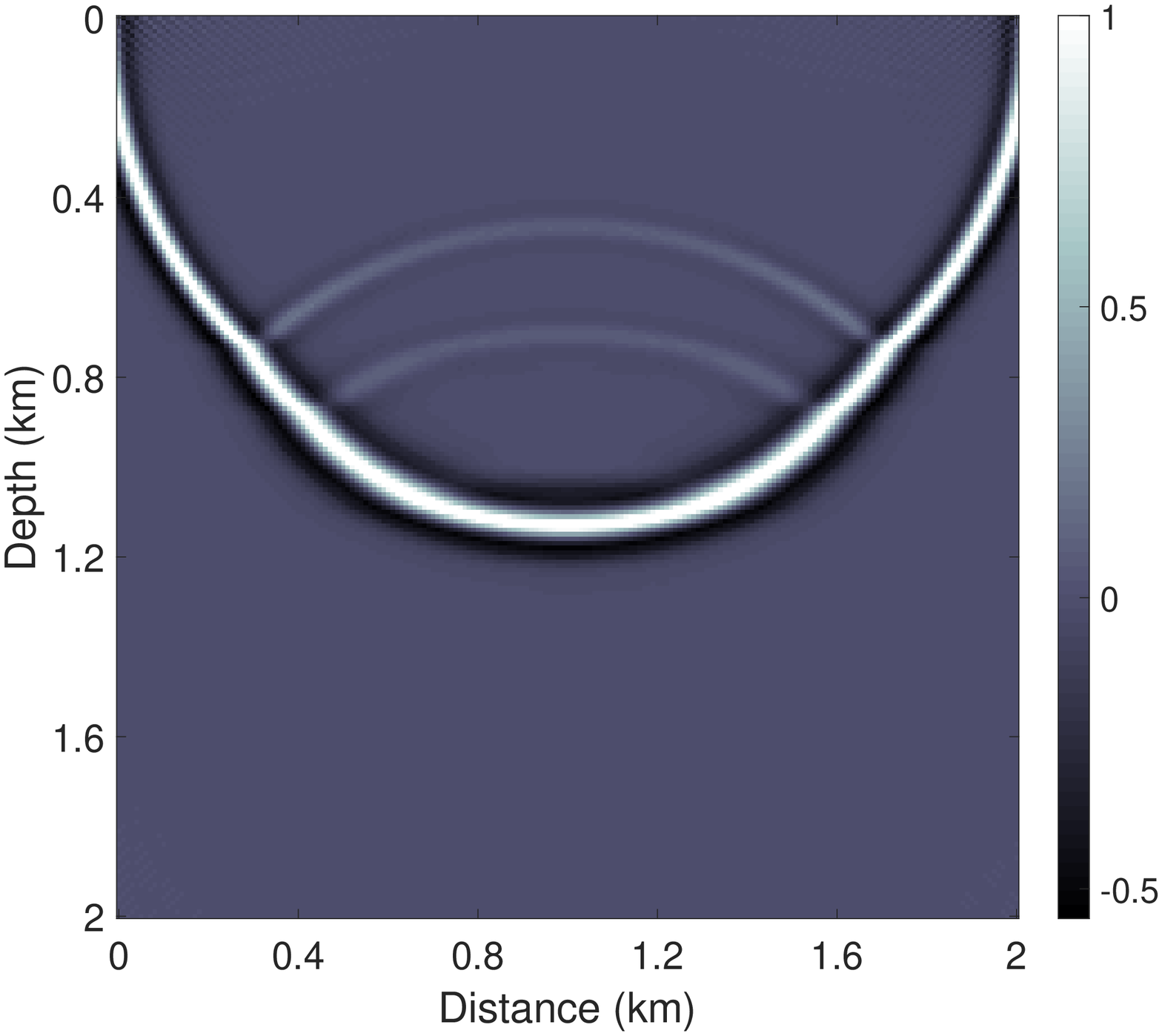}\label{Ex3.fig3b}}
\caption{Snapshots for $t=520$ ms generated by various schemes}\label{Ex3.fig3}
\end{figure}

We have also placed a receiver at $(x_r,z_r) =(500,0)$ (See Figure \ref{Ex3.fig1}) and plotted the corresponding numerical solution computed by refined PW 17p and refined 17p in Figure \ref{Ex3.fig4}. Consider the graph between 0.2s and 0.3s or (0.3s to 0.7s) of Figure \ref{Ex3.fig4b}. We can observe that there is some additional energy there. This additional energy is an artificial effect caused by the numerical solver, refined 17p.

\begin{figure}[htp]
\centering
\subfloat[refined PW 17p]{\includegraphics[width=0.45\textwidth]{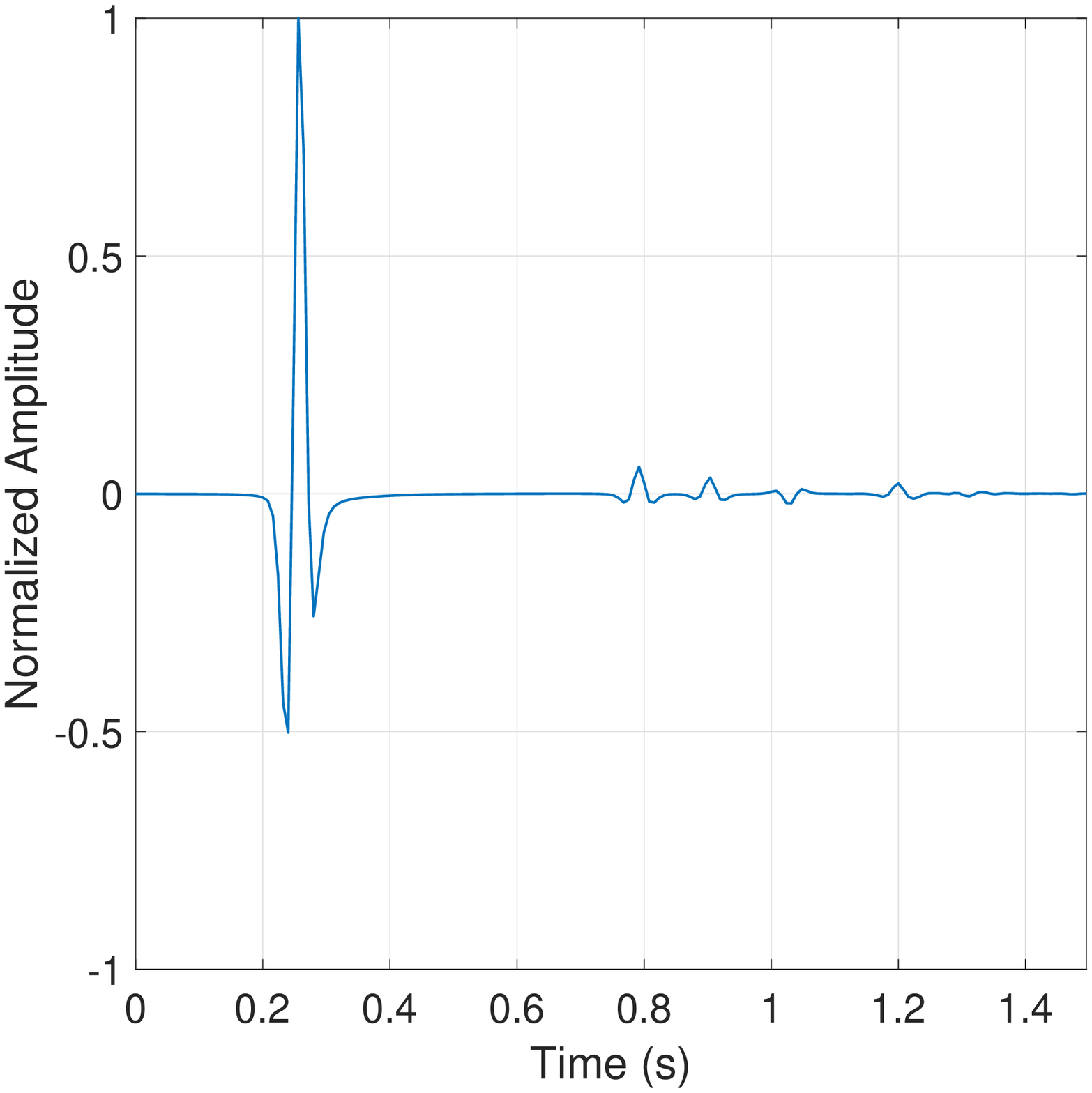}\label{Ex3.fig4a}}
~
\subfloat[refined 17p]{\includegraphics[width=0.45\textwidth]{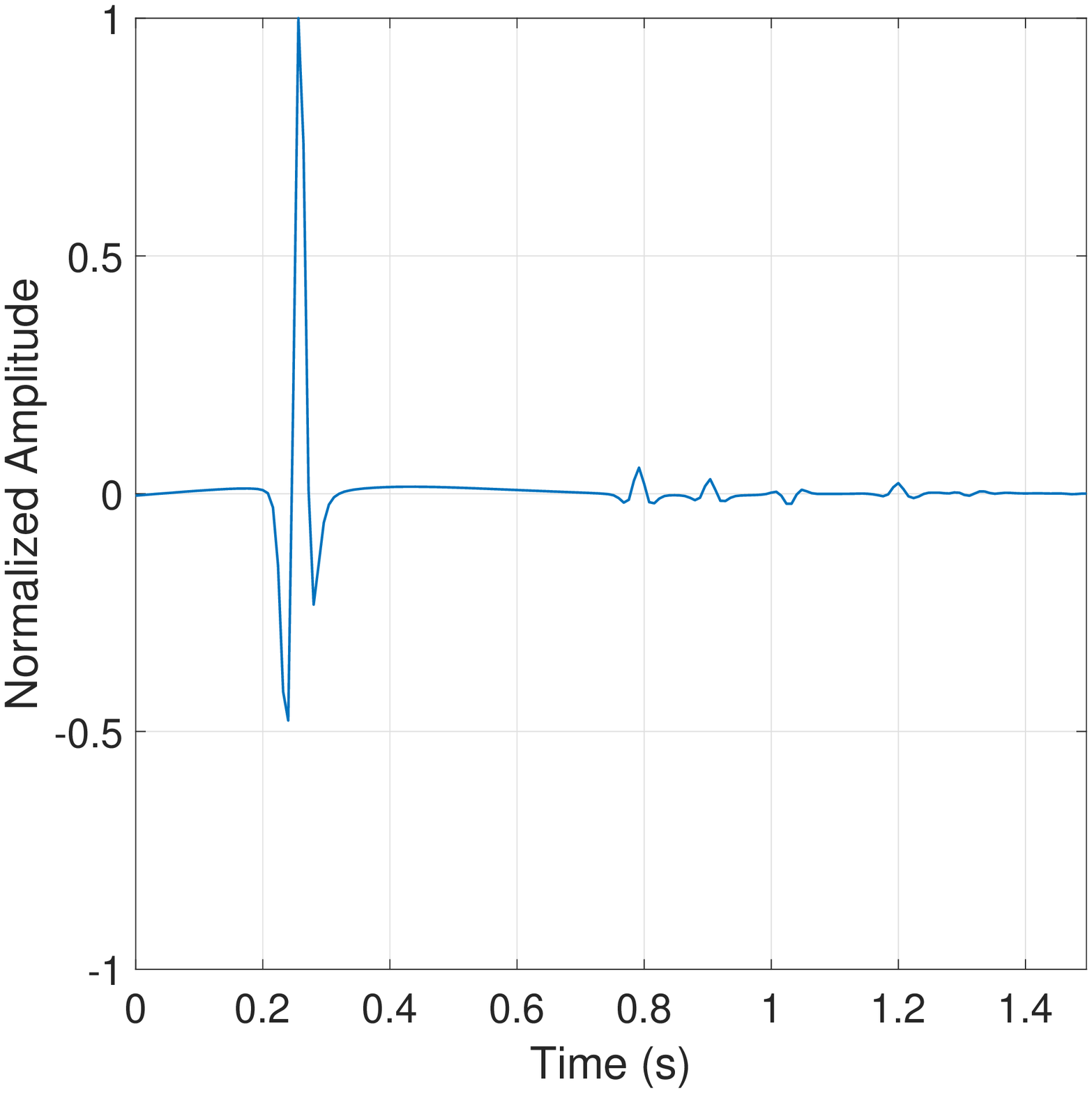}\label{Ex3.fig4b}}
\caption{The numerical solution by various schemes}\label{Ex3.fig4}
\end{figure}

From \cite{dastour2019}, we know that refined 17p is inconsistent with the Helmholtz equation with PML. In some examples, it is hard to observe this inconsistency. However, in some examples, when we deal with large wavenumbers, this inconsistency can be noticeable. In this example, we have shown that refined PW 17p can replace the need for refined 17p in practical applications.

\section{Conclusions}\label{S4}
In this article, we developed a general approach for constructing fourth-order finite difference schemes for the Helmholtz equation with PML using point-weighting strategy. Particularly, we developed two optimal finite difference methods, the optimal point-weighting 25-point finite difference and optimal point-weighting 17-point finite difference methods. We proved that the two new methods are indeed 4th-order and consistent with the Helmholtz equation with PML when different spacial increment along x-axis and z-axis are used. Two algorithms for parameter selection of these optimal methods presented based on minimization of numerical dispersion of the two finite difference schemes. Normalized phase and group velocity curves for the two schemes have shown significant level of reduction in numerical dispersion. Our numerical examples confirmed the theoretical results and demonstrated the necessity of using consistent finite different schemes. These examples also compared the accuracy and efficiency of the new schemes with a number of existing finite difference schemes that are widely used for the Helmholtz equation with PML. Our numerical results also confirmed that optimal point-weighting 17-point is a consistent fourth-order numerical solver for the Helmholtz equation with PML.

In the future, we plan to extend the new schemes for solving the Helmholtz equation with PML in three-dimension.

\section*{Acknowledgments}
The work is supported by the Natural Sciences \& Engineering Research Council of Canada (NSERC) through the individual Discovery Grant (RGPIN-2019-04830). The first author is also thankful for the Alberta Innovates Graduate Student Scholarship that he received during his Ph.D. studies.

\bibliographystyle{plain}
\biboptions{sort&compress}
\bibliography{FDFD25pPW.bib}
\end{document}